\documentclass{amsart}
\usepackage[utf8]{inputenc}
\usepackage[margin=1.25in]{geometry}
\usepackage{amsmath}
\usepackage{amssymb}
\usepackage{mathtools}
\usepackage{mathrsfs}
\usepackage{amsthm}
\usepackage{float}
\usepackage[dvipsnames]{xcolor}
\usepackage{booktabs}
\usepackage{multirow}
\usepackage{amsfonts}
\usepackage{shuffle}
\usepackage{graphicx}
\usepackage{tikz-cd}
\usepackage{enumerate}
\usepackage{verbatim}
\usepackage{ytableau} 
\usepackage{hyperref}
\usepackage{thmtools}
\usepackage{cleveref}
\usepackage{tikz}
\usetikzlibrary{calc, shapes, backgrounds,arrows,positioning,plotmarks,decorations.pathreplacing}
\tikzset{>=stealth',
  head/.style = {fill = white, text=black},
  plaque/.style = {draw, rectangle, minimum size = 10mm, fill=white}, 
     pil/.style={->,thick},
  junct/.style = {draw,circle,inner sep=0.5pt,outer sep=0pt, fill=black}
  }

\definecolor{amber}{rgb}{1.0,0.5,0} 
\definecolor{amethyst}{rgb}{.6,0.2,1} 
\definecolor{clay}{rgb}{1,.75,0} 
\definecolor{dark}{rgb}{0.05, 0.5, 0.06}
\definecolor{sky}{rgb}{0.52, 0.73, 0.4}
\definecolor{lavender}{rgb}{0.96, 0.73, 1.0}

\newtheorem{theorem}{Theorem}[section]
\newtheorem{lemma}[theorem]{Lemma}
\newtheorem{proposition}[theorem]{Proposition}
\newtheorem{corollary}[theorem]{Corollary}

\newtheorem{introthm}{Theorem}
\newtheorem*{claim}{Claim}

\theoremstyle{definition}
\newtheorem{definition}[theorem]{Definition}
\newtheorem{algorithm}[theorem]{Algorithm}

\theoremstyle{remark}
\newtheorem{remark}[theorem]{Remark}
\newtheorem{example}[theorem]{Example}

\numberwithin{equation}{section}

\DeclareMathOperator{\trip}{\mathsf{trip}}
\DeclareMathOperator{\col}{\mathsf{col}}
\DeclareMathOperator{\osc}{\mathsf{osc}}
\DeclareMathOperator{\slack}{\mathsf{slack}}
\DeclareMathOperator{\bw}{\mathbf{w}}

\DeclareMathOperator{\ft}{FT}
\DeclareMathOperator{\rft}{RFT}
\DeclareMathOperator{\id}{id}
\DeclareMathOperator{\aexc}{\mathrm{Aexc}}

\DeclareMathOperator{\prom}{\mathsf{prom}}
\DeclareMathOperator{\promotion}{\mathcal{P}}
\DeclareMathOperator{\devacuation}{\mathcal{E}^*}
\DeclareMathOperator{\evacuation}{\mathcal{E}}

\DeclareMathOperator{\PM}{\mathbf{M}}

\DeclareMathOperator{\sgn}{\mathsf{sgn}}

\DeclareMathOperator{\SL}{SL}
\newcommand{\bbb}{\mathsf{b}}
\newcommand{\fsl}{\mathfrak{sl}}
\DeclareMathOperator{\ASM}{\mathcal{ASM}}

\DeclareMathOperator{\PP}{\mathcal{PP}}
\DeclareMathOperator{\Inv}{\mathsf{Inv}}
\DeclareMathOperator{\Hom}{\mathrm{Hom}}

\DeclareMathOperator{\Sym}{\mathsf{Sym}}

\DeclareMathOperator{\JDT}{\mathsf{jdt}}

\DeclareMathOperator{\sep}{\mathsf{sep}}
\DeclareMathOperator{\rot}{\mathsf{rot}}
\DeclareMathOperator{\refl}{\mathsf{refl}}
\DeclareMathOperator{\crg}{CRG}

\DeclareMathOperator{\cg}{CG}
\DeclareMathOperator{\rg}{RG}
\DeclareMathOperator{\ssv}{SSV}
\DeclareMathOperator{\ssyt}{SSYT}
\DeclareMathOperator{\wgt}{\mathsf{wgt}}
\DeclareMathOperator{\wssv}{WSSV}

\DeclareMathOperator{\tlex}{l\widetilde{ex}}
\DeclareMathOperator{\grevlex}{grevlex}
\DeclareMathOperator{\Des}{Des}
\DeclareMathOperator{\sfH}{\mathsf{H}}
\DeclareMathOperator{\sfX}{\mathsf{X}}
\newcommand{\bs}{\boldsymbol}
\newcommand{\symm}{symmetrized }
\newcommand{\too}[1]{\stackrel{#1}{\to}}
\newcommand\precdot{\mathrel{\ooalign{$\prec$\cr
  \hidewidth\raise0.001ex\hbox{$\cdot\mkern0.6mu$}\cr}}}

\AtBeginDocument{%
   \def\MR#1{}
}

\title{Rotation-invariant web bases from hourglass plabic graphs}
\author[Gaetz]{Christian Gaetz}
\address[Gaetz]{Department of Mathematics, University of California, Berkeley, CA, USA.}
\email{gaetz@berkeley.edu}

\author[Pechenik]{Oliver Pechenik}
\address[Pechenik]{Department of Combinatorics \& Optimization, University of Waterloo, ON, Canada.}
\email{oliver.pechenik@uwaterloo.ca}

\author[Pfannerer]{Stephan Pfannerer}  
\address[Pfannerer]{Department of Combinatorics \& Optimization, University of Waterloo, ON, Canada.}
\email{math@pfannerer-mittas.net}

\author[Striker]{Jessica Striker}
\address[Striker]{Department of Mathematics, North Dakota State University, Fargo, ND, USA.}
\email{jessica.striker@ndsu.edu}

\author[Swanson]{Joshua P. Swanson}
\address[Swanson]{Department of Mathematics, University of Southern California, Los Angeles, CA, USA.}
\email{swansonj@usc.edu}

\thanks{Gaetz was partially supported by NSF fellowship DMS-2103121 and by NSF grant DMS-2452032. Pechenik and Pfannerer were partially supported by Pechenik's Discovery Grant (RGPIN-2021-02391) and Launch Supplement (DGECR-2021-00010) from
the Natural Sciences and Engineering Research Council of Canada. Pfannerer was also partially supported by  the Austrian  Science Fund (FWF) P29275, Olya Mandelshtam's Discovery Grant (RGPIN-2021-02568), and was a recipient of a DOC Fellowship of the Austrian Academy of Sciences. Striker was partially supported by Simons Foundation gifts MP-TSM-00002802 and 527204 and NSF grant DMS-2247089. Swanson was partially supported by NSF grant DMS-2348843.}

\date{\today}

\begin{document}
\begin{abstract}
Webs give a diagrammatic calculus for spaces of tensor invariants. We introduce \emph{hourglass plabic graphs} as a new avatar of webs, and use these to give the first rotation-invariant $U_q(\fsl_4)$-web basis, a long-sought object. The characterization of our basis webs relies on the combinatorics of these new plabic graphs and associated configurations of a symmetrized six-vertex model. We give growth rules, based on a novel crystal-theoretic technique, for generating our basis webs from tableaux and we use skein relations to give an algorithm for expressing arbitrary webs in the basis. We also discuss how previously known rotation-invariant web bases can be unified in our framework of hourglass plabic graphs.
\end{abstract}

\maketitle
\setcounter{tocdepth}{1}
\tableofcontents
\section{Introduction}

Over the last four decades, the classical theory of spaces of  $\SL_r(\mathbb{C})$-invariants has been extended with the aid of diagrams called \emph{webs}, whose calculus has powerful topological applications. For $\SL_2$, the \emph{Temperley--Lieb basis} consists of tensor invariants of non-crossing matchings of points around a disk and can be used to compute the \emph{Jones polynomial}. For $\SL_3$, Kuperberg \cite{Kuperberg} introduced the remarkable \emph{non-elliptic web basis} with many beautiful properties. However, a rotation-invariant extension of Kuperberg's basis to higher ranks has proven elusive ever since its introduction in 1996. Our main result provides the first such basis for $\SL_4$, and indeed for its quantum deformation $U_q(\fsl_4)$.

\begin{introthm}[See \Cref{thm:web-basis}]\label{thm:A}
The collection $\mathscr{B}_q^{\underline{c}}$ of tensor invariants of top fully reduced hourglass plabic graphs of type $\underline{c}$ is a rotation-invariant web basis for the invariant space $\Inv_{U_q(\fsl_4)}(\bigwedge\nolimits_q^{\underline{c}} V_q)$.
\end{introthm}
Such a rotation-invariant basis has long been desired by those developing the theory of webs; see, e.g., \cite{Westbury, Fontaine, Fraser-Lam-Le,Fraser:braid, Fraser-2-column} for such remarks. The lack heretofore of suitable generalizations of the $\SL_3$-web basis to higher rank has also been specifically lamented in applications of webs to cluster algebras \cite{Fomin-Pylyavskyy-advances}, enumerative combinatorics \cite{Petersen-Pylyavskyy-Rhoades}, representation theory \cite{Tymoczko-Russell-unitriangular}, quantum topology \cite[p.~10]{Le.Sikora}, and dimer models \cite{Douglas-Kenyon-Shi}.

The key insight behind \Cref{thm:A} is our introduction of a combinatorial framework we call \textit{hourglass plabic graphs}, so named because they allow certain half-twist multi-edges called \textit{hourglasses}. These graphs are an extension of Postnikov's plabic graphs; while Postnikov's graphs are governed by a single \emph{trip permutation}, ours crucially involve a tuple $\trip_{\bullet}$ of such permutations. To establish \Cref{thm:A}, we describe a map $\mathcal{T}$ from move-equivalence classes of certain hourglass plabic graphs to $4$-row \emph{fluctuating tableaux} (a class including standard tableaux) and a map $\mathcal{G}$, based on \emph{growth rules}, in the other direction (see \Cref{fig:intro-ex} for an example).  In \cite{fluctuating-paper}, we showed that orbits of \emph{promotion} on fluctuating tableaux are tracked by a tuple $\prom_{\bullet}$ of \emph{promotion permutations}.

\begin{introthm}[See \Cref{thm:main-bijection}]\label{thm:B}
The maps $\mathcal{T}$ and $\mathcal{G}$ are mutually inverse bijections between move-equivalence classes of fully reduced hourglass plabic graphs and $4$-row rectangular fluctuating tableaux. Furthermore, this bijection satisfies $\trip_{\bullet}(G)=\prom_{\bullet}(\mathcal{T}(G))$ and consequently intertwines promotion of tableaux with rotation of hourglass plabic graphs.
\end{introthm}

\begin{figure}[hbtp]
    \centering
    \begin{tikzpicture}
        \node[inner sep=0.5cm] (G) at (0,0)
          {\includegraphics[scale=1.1]{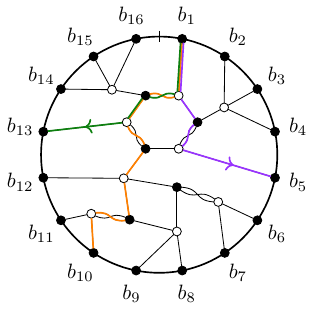}};
        \node[inner sep=0.5cm] (T) at (7,0)
          {\begin{ytableau}
          1 & 2 & 6 & 8 \\
          3 & 5 & 7 & 14 \\
          4 & 9 & 10 & 15 \\
          11 & 12 & 13 & 16
          \end{ytableau}};
        \begin{scope}[transform canvas={yshift=0.7em}]
          \draw[|->] (G.east) to[out=30, in=150] node[midway, above] {$\mathcal{T}$} (T.west);
        \end{scope}
        \begin{scope}[transform canvas={yshift=-0.7em}]
          \draw[|->] (T.west) to[out=210, in=330] node[midway, below] {$\mathcal{G}$} (G.east);
        \end{scope}
    \end{tikzpicture}

    \caption{A top fully reduced hourglass plabic graph $G$ and its corresponding $4$-row rectangular standard tableau $\mathcal{T}(G)$. The purple ($\textcolor{amethyst}{\blacksquare}$) $\trip_1$-, orange ($\textcolor{amber}{\blacksquare}$) $\trip_2$-, and green ($\textcolor{dark}{\blacksquare}$) $\trip_3$-strands are drawn, showing that $\trip_i(G)(1)=5, 10,$ and $13$ for $i=1,2,$ and $3$, respectively.}
    \label{fig:intro-ex}
\end{figure}

\noindent Associated to each tableau $T$ is a \emph{lattice word}, which can be thought of as a \emph{highest weight} element in a Kashiwara crystal of words \cite{Bump-Schilling, Kashiwara-Crystals}. In this setting, promotion is identified with the action of a \emph{crystal commutor} \cite{henriques-kamnitzer, lenart, Pfannerer-Rubey-Westbury}. Our growth rules are local substitutions in lattice words induced by crystal isomorphisms, together with diagrammatic building blocks of webs that we call \emph{plumbings}. The growth rules therefore interact well with crystal commutors. The trip permutations of the plumbings precisely account for the local change in promotion permutations caused by a growth rule, guaranteeing that $\prom_{\bullet}(T)=\trip_{\bullet}(\mathcal{G}(T))$. We apply the growth rules to deduce that the change of basis between $\mathscr{B}_q^{\underline{c}}$ from \Cref{thm:A} and the monomial basis is unitriangular over $\mathbb{Z}[q, q^{-1}]$ (see \Cref{thm:unitriangular}).

Our framework unifies the $\SL_r$-web bases for $r \leq 4$, as well as the ``$2$-column'' case in arbitrary rank (see \Cref{sec:two_column} and \cite{two-column}). It is also remarkably rich from a purely combinatorial perspective. For example, two extreme move-equivalence classes for $r=4$ are naturally identified with the lattices of \emph{alternating sign matrices} and \emph{plane partitions} in a box, respectively (see \Cref{sec:cc-asm-pp}). 

\subsection{Hourglass plabic graphs}\label{sec:thmB}

Previous efforts to generalize Kuperberg's work have largely focused on the representation theory of quantum groups (see e.g.~\cite{Kim,Morrison,Hagemeyer}). We instead take a combinatorial approach, which we develop throughout \Cref{sec:hourglass-and-six-vertex,sec:forwards-map,sec:backwards-map}; in \Cref{sec:bijection-and-basis} this combinatorial theory straightforwardly yields our rotation-invariant web basis of $\Inv_{U_q(\fsl_4)}(\bigwedge\nolimits_q^{\underline{c}} V_q)$ from \Cref{thm:A}. We now describe the combinatorial ingredients of \Cref{thm:B} before returning to discussion of web invariants.

\emph{Plabic graphs} were introduced by Postnikov to study the totally nonnegative Grassmannian \cite{Postnikov-arxiv, Postnikov-ICM}. These are planar graphs embedded in a disk with internal vertices colored black or white. The \emph{trip permutation} of a plabic graph is the bijection obtained by starting at a boundary vertex and traveling through the graph, taking a left at white vertices and a right at black vertices, until another boundary vertex is reached. Postnikov showed that \emph{reduced} plabic graphs are classified up to \emph{move-equivalence} by their trip permutation.

Building on connections between webs and plabic graphs due to Fraser--Lam--Le \cite{Fraser-Lam-Le}, Hopkins--Rubey \cite{Hopkins-Rubey,Hopkins:talk} recently observed that Kuperberg's non-elliptic $\SL_3$ basis webs can be interpreted as reduced bipartite plabic graphs.
Khovanov--Kuperberg \cite{Khovanov-Kuperberg} introduced a bijection between $\SL_3$ basis webs and \emph{sign and state strings}, which are well-known to correspond to $3$-row rectangular tableaux (see, e.g.,~\cite{Petersen-Pylyavskyy-Rhoades,Patrias}). Petersen--Pylyavskyy--Rhoades \cite{Petersen-Pylyavskyy-Rhoades} showed that this bijection intertwines rotation with tableau promotion in the standard case; Patrias \cite{Patrias} extended this to the \emph{(generalized) oscillating} case. The recent work of Hopkins--Rubey \cite{Hopkins-Rubey,Hopkins:talk} 
encodes the structure of tableau promotion in a \emph{promotion permutation} and shows that the Khovanov--Kuperberg bijection carries this permutation to the trip permutation of an $\SL_3$ basis web. 

Motivated by results such as \cite{Rhoades}, combinatorialists have long desired such a planar model encoding promotion of general rectangular tableaux as rotation. 
We observe that a consequence of the Hopkins--Rubey result is that the $\SL_3$ web basis (and moreover its bijection to $3$-row rectangular standard tableaux) is uniquely characterized among reduced plabic graphs by this ``$\trip=\prom$'' property. Our key guiding principle is that this property should extend to higher rank web bases and other classes of tableaux.

In \cite{fluctuating-paper}, we
introduce fluctuating tableaux as a simultaneous generalization of standard, dual-semistandard, and the (generalized) oscillating tableaux of Patrias~\cite{Patrias}. These are the natural combinatorial objects whose number is the dimension of the invariant space of interest. For rectangular tableaux with $r>3$ rows, one sees that the promotion permutation determines only the first and last row of the tableau; hence, it is natural to consider analogues of the promotion permutation encoding other rows of the tableau. We do this by associating to an $r$-row rectangular fluctuating tableaux an entire sequence of \emph{promotion permutations} $\prom_\bullet(T) = (\prom_1(T), \ldots, \prom_{r-1}(T))$. 

In light of this extension, we desire a graphical model for $\SL_r$ webs that admits a sequence of \emph{trip permutations} $\trip_\bullet(G) = (\trip_1(G), \ldots, \trip_{r-1}(G))$. For this reason we introduce \emph{hourglass plabic graphs}, as illustrated in \Cref{fig:intro-ex}. These are $r$-valent properly bicolored graphs embedded in a disk. A special feature are the \emph{hourglass} edges, which are multiple edges twisted so that the clockwise orders of their strands around the two incident vertices are the same. The $i$-th trip permutation $\trip_i$ is obtained by taking the $i$-th left at white vertices and the $i$-th right at black vertices. Here, the twists of the hourglass edges are essential to yield the desired behavior of $\trip_\bullet(G)$. Indeed, one notices their necessity already in the smallest examples if one wants that, for each $4$-row rectangular tableau $T$, there is a $4$-valent hourglass plabic graph $G$ such that $\prom_\bullet(T) = \trip_\bullet(G)$. Remarkably, the same statement holds for all other known rotation-invariant $\SL_r$ web bases:  the $r < 4$ cases and Fraser's \cite{Fraser-2-column} $2$-column web basis for arbitrary $\SL_r$. (See \Cref{sec:two_column} for further discussion.) In future work, we hope to extend our approach to all $r$.

It remains for us to determine which subset of hourglass plabic graphs is naturally in bijection with tableaux via $\trip_{\bullet}=\prom_{\bullet}$. Hourglass plabic graphs $G$ for $r = 4$ admit a notion of \emph{full reducedness}, defined by forbidding certain faces in graphs equivalent to $G$ under \emph{moves}, most notably \textit{square moves} and \textit{benzene moves} (see \Cref{fig:main-hourglass-moves}). Full reducedness is in many ways analogous to the well-studied notion of reducedness on plabic graphs, while moves on hourglass plabic graphs lift moves on plabic graphs. In particular, \Cref{thm:B} entails that two fully reduced hourglass plabic graphs are move-equivalent if and only if their trip permutations coincide, mirroring Postnikov's celebrated characterization of move-equivalence for reduced plabic graphs \cite[Thm.~13.4]{Postnikov-arxiv}. We moreover establish that full reducedness can be characterized by forbidding certain crossings of trip strands, just as Postnikov demonstrated for reduced plabic graphs \cite[Thm.~13.2]{Postnikov-arxiv}. This characterization is critical to many of our arguments and gives another explanation for the appearance of hourglass edges, since a bipartite plabic graph with parallel edges fails to be reduced.  

The definition of full reducedness naturally extends to the $r <4$ case, forbidding certain uniformly defined face configurations. This definition directly recovers the Temperley--Lieb noncrossing condition when $r=2$ and Kuperberg's non-elliptic condition when $r=3$, giving a uniform description of basis webs in all three settings. 

\begin{figure}[hbtp]
    \centering
    \includegraphics[scale=1.1]{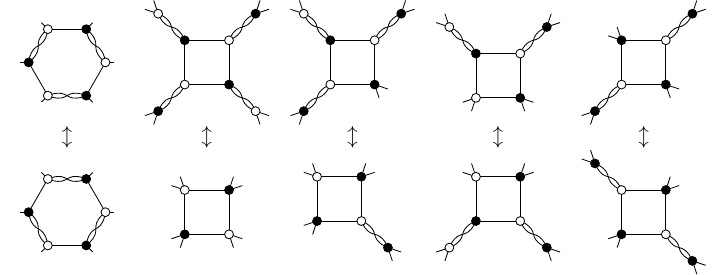}
    \caption{A benzene move (leftmost) and the square moves for hourglass plabic graphs. The color reversals of these moves are also allowed.}
    \label{fig:main-hourglass-moves}
\end{figure}

When $r=4$, we may transform hourglass plabic graphs into directed graphs that we call \emph{symmetrized six-vertex configurations} (see \Cref{sec:six-vertex-correspondence}) and we move back and forth between these realizations as convenient. These configurations appeared independently in Hagemeyer's thesis \cite{Hagemeyer}. In contrast with the hourglass plabic graph setting, where $\trip_1$ (and $\trip_3$) is easiest to understand, the symmetrized six-vertex configurations make $\trip_2$ particularly transparent. Symmetrized six-vertex configurations are closely related to the usual six-vertex (\emph{square ice}) model of statistical physics (see e.g.\ \cite{BaxterBook}) and we thereby obtain intriguing connections to the combinatorics of \emph{plane partitions} and \emph{alternating sign matrices}; see \Cref{sec:asm,sec:plane-partitions}.

Our proof of \Cref{thm:B} involves constructing (\Cref{sec:backwards-map}) the map $\mathcal{T}$ from fully reduced hourglass plabic graphs $G$ to fluctuating tableaux. This map involves giving an edge-labeling of $G$ that is obtained from information about the crossings of $\trip_\bullet$-strands. Similar labelings for $\SL_3$ appear in \cite{Khovanov-Kuperberg,LACIM}. In the reverse direction, we give in \Cref{sec:forwards-map} a \emph{growth algorithm} $\mathcal{G}$, analogous to that of \cite{Khovanov-Kuperberg} for $\SL_3$, that produces an $\SL_4$-basis web given the labeling of the boundary edges. The growth algorithm is the most combinatorially intricate part of our work. The proof of correctness relies critically on a novel Kashiwara crystal-theoretic technique that we introduce. 

\subsection{Bases of tensor invariants}\label{sec:intro_bases}
Consider $\mathsf{G} = \SL_r(\mathbb{C})$  with its defining representation $V = \mathbb{C}^r$, its fundamental representations $V^{\omega_k} = \bigwedge^k V$, and their duals, which we write as $\bigwedge^{-k} V \coloneqq (\bigwedge^k V)^*$.
Given any tensor product 
\begin{equation}\label{eq:tensor}
    \bigwedge\nolimits^{\underline{c}} V \coloneqq \bigwedge\nolimits^{c_1} V \otimes \cdots \otimes \bigwedge\nolimits^{c_n} V
\end{equation}
of such representations, the space of linear functionals on $\bigwedge\nolimits^{\underline{c}} V$ also carries a $\mathsf{G}$-action. It is a classical problem of invariant theory to describe the subspace $\Inv(\bigwedge\nolimits^{\underline{c}} V) \coloneqq \Hom_\mathsf{G}(\bigwedge\nolimits^{\underline{c}} V, \mathbb{C})$ of invariant multilinear forms. It is typical to consider only tensor products of fundamental representations as in \eqref{eq:tensor}, since the category of all finite-dimensional representations may be recovered from these through the \emph{Karoubi envelope} construction (see e.g.~\cite[\S3.3.1]{Morrison}, \cite[\S3.1]{Cautis-Kamnitzer-Morrison}).  

Replacing $\mathsf{G}$ with the corresponding quantum group $U_q(\fsl_r)$, we obtain the space $\Inv_{U_q(\fsl_r)}(\bigwedge\nolimits_q^{\underline{c}} V_q)$ of quantum invariants, deforming $\Inv(\bigwedge\nolimits^{\underline{c}} V)$. Such quantum invariants and the associated representation theory of quantum groups give rise to many of the most powerful invariants of knots, links, and tangles (see, for example, \cite{Jones,Kauffman,Turaev,Khovanov:Jones,Khovanov:sl3,Morrison.Nieh}).

For many purposes, one needs to have explicit bases for $\Inv(\bigwedge\nolimits^{\underline{c}} V)$ or its quantum deformation. Likely the first such basis to be discovered was the \emph{standard monomial basis} constructed by Hodge \cite{Hodge} and greatly developed by Seshadri and collaborators (e.g., \cite{Seshadri:1,Seshadri:2,Seshadri:3,Seshadri:4}). While relatively elementary and computable, the standard monomial basis does not enjoy many of the other properties one desires in a basis. As a result, many further authors have developed a panoply of less tractable bases with various assortments of desirable properties. We have, for example, the \emph{dual canonical basis}, developed independently by Lusztig \cite{Lusztig:canonical} and Kashiwara \cite{Kashiwara:duke_crystals}, as well as the \emph{dual semicanonical basis} of Lusztig \cite{Lusztig:semicanonical_1, Lusztig:semicanonical_2} and the \emph{geometric Satake basis} (see e.g.\ \cite{Zhu}). More recently, the theory of cluster algebras has engendered further bases, such as the \emph{theta basis} of Gross, Hacking, Keel, and Kontsevich \cite{Gross.Hacking.Keel.Kontsevich}, in the cases (see \cite{Fomin-Pylyavskyy-advances}) where $\Inv(\bigwedge\nolimits^{\underline{c}} V)$ is known to carry a cluster structure. For a recent survey of cluster algebra bases and their relations, see \cite{Qin:survey}.

There is a natural cyclic shift isomorphism 
\[
 \shuffle : \Inv\left(\bigwedge\nolimits^{(c_1, \dots, c_n)} V\right) \xrightarrow{\sim} \Inv\left(\bigwedge\nolimits^{(c_2, \ldots, c_n, c_1)} V\right)
\] 
induced by the natural isomorphism with $\Hom_\mathsf{G}(\bigwedge\nolimits^{(c_2, \dots, c_n)} V, (\bigwedge\nolimits^{c_1} V)^*)$.
Given a choice of basis as above, one might hope that this isomorphism would preserve the basis, mapping basis elements to scalar multiples of basis elements. In such a case, we say the basis is \emph{$\shuffle$-invariant}. Indeed, instances of $\shuffle$-invariance for the dual canonical, geometric Satake, and theta bases have been used to great effect in \cite{Rhoades}, \cite{Fontaine.Kamnitzer}, and \cite{Shen.Weng}, respectively. Unfortunately, these three bases are notoriously difficult for computation. It is much easier to compute with the standard monomial basis; however, the standard monomial basis is far from $\shuffle$-invariant even in very simple cases and is not well-suited to analysis of quantum link invariants.

\subsection{Web bases}
\emph{Web bases} are diagrammatic bases for $\Inv(\bigwedge\nolimits^{\underline{c}} V)$ or its quantum deformation that seek to surmount these challenges. Each basis element $[W]$ is encoded by a planar bipartite graph $W$ embedded in a disk, so that important algebraic operations are realized as simple graph-theoretic manipulations.

When $\mathsf{G} = \SL_2$, the Temperley--Lieb web basis of non-crossing matchings has long been known (see \cite{Temperley.Lieb,Kung-Rota,Kauffman.Lins}). When $\mathsf{G} = \SL_3$, Kuperberg \cite{Kuperberg} introduced a remarkable web basis involving a non-elliptic condition on trivalent graphs and gave similar constructions for the other simple Lie groups of rank $2$. The non-crossing and non-elliptic conditions are clearly rotation-invariant, so the web bases for $U_q(\fsl_2)$ and $U_q(\fsl_3)$ are $\shuffle$-invariant. While Kuperberg's $\SL_3$ web basis was primarily introduced as a tool for efficient computation of quantum link invariants, it has since found application in areas as diverse as the geometric Satake correspondence \cite[Thm.~1.4]{Fontaine-Kamnitzer-Kuperberg}, cluster algebras \cite{Fomin-Pylyavskyy-advances}, and cyclic sieving \cite{Petersen-Pylyavskyy-Rhoades}. The transition matrices to the bases discussed in \Cref{sec:intro_bases} have also been of great recent interest (see \cite{Frenkel.Khovanov,Khovanov-Kuperberg,Russell.Tymoczko:2row,Rhoades:polytabloid,Im.Zhu,Tymoczko-Russell-unitriangular,Hwang.Jang.Oh}).

Since Kuperberg's introduction of rank $2$ web bases, a primary focus has been his statement:
\begin{quote}
    \small
    ``The main open problem related to the combinatorial rank 2 spiders is how to generalize them to higher rank.'' --\cite[p.~146]{Kuperberg}
\end{quote}
Kim \cite{Kim} conjectured web generators and relations for the case of $U_q(\fsl_4)$ and Morrison \cite{Morrison} extended Kim's conjecture to $U_q(\fsl_r)$. Cautis--Kamnitzer--Morrison \cite{Cautis-Kamnitzer-Morrison} later proved these conjectures (with a different proof subsequently given in \cite{Fraser-Lam-Le}). While the diagrammatic relations have thus been determined, generalizing Kuperberg's web basis in a rotation-invariant way has proved difficult. 

Several web bases for simple Lie groups of higher rank have appeared in the literature, though none is rotation-invariant. In \cite{Westbury}, for $\mathsf{G} = \SL_r$ with representations restricted to $c_i \in \{1, r-1\}$, Westbury gave a recipe to construct a generator with each possible leading term, thereby obtaining a basis of web diagrams. Unfortunately, this basis is not rotation-invariant.\footnote{Some brief remarks in \cite[p.~94]{Westbury} suggest that Westbury's basis is rotation-invariant. This is not correct. The smallest example of non-invariance occurs for $r=4$ and $\underline{c} = (1,1)$. We thank Bruce Westbury for very helpful correspondence on this point. See also \cite[p.~6121]{Fraser-Lam-Le} for related discussion.} Westbury's basis was extended to products of arbitrary fundamental representations by Fontaine \cite{Fontaine}. Fontaine's basis also is not rotation-invariant; moreover, in general, it involves making arbitrary choices, so is not well-defined in all cases. Elias \cite{Elias} independently obtained what appear to be the quantum group deformations of these bases; again, his bases involve arbitrary choices and are not rotation-invariant. A further obstacle to the use of these non-rotation-invariant bases is that there is no known topological or combinatorial property that enables determining whether a given web is a basis element, let alone an efficient method for expressing invariants in the basis. More recently, in his thesis, Hagemeyer gave another family of web bases for $U_q(\fsl_4)$ \cite[Thm.~1.8]{Hagemeyer}, which, like our basis, are related to the symmetrized six-vertex model. However, he required arbitrary choices and thus did not obtain a rotation-invariant basis.

Previous work has typically focused on trivalent graph models of representation categories of classical or quantum groups (see e.g.~\cite{MOY,Khovanov-Kuperberg,Kim,Morrison,Cautis-Kamnitzer-Morrison}). Intuitively, the trivalent vertices encode multiplication or comultiplication maps in the exterior algebra of $V$, and webs are formed by composing such maps (along with evaluation and coevaluation morphisms). See \cite{Selinger} for an overview of the graphical calculus of such \emph{pivotal categories}. Our hourglass plabic graphs build instead on \cite{Fraser-Lam-Le} by using an $r$-valent model. (Such $r$-valent models also appear in the topology literature \cite{Sikora,Le.Sikora}; for the connection of that work with \cite{Cautis-Kamnitzer-Morrison}, see \cite{Poudel}.) This $r$-valent approach is better suited to rotation invariance. For example, under our approach the determinant map corresponds to $r$ simple edges around a single internal vertex, rather than a full binary tree, as in a trivalent model. See \Cref{sec:prelim} for details.

Our main algebraic result, \Cref{thm:A}, gives the first rotation-invariant web basis for $\SL_4$, and indeed for $U_q(\fsl_4)$. This is the first construction of a rotation-invariant web basis for any simple Lie group of rank greater than $2$ (although some apparent web bases for algebras of block-triangular matrices appear in \cite{Patrias-Pechenik-Striker}).

The family of basis webs of \Cref{thm:A} for $r=4$ is not preserved under reflection, since the \emph{top} condition (see \Cref{def:top-fully-reduced}) requires that benzene rings are oriented as in the upper-left diagram of \Cref{fig:main-hourglass-moves}. Reflection interchanges the upper-left and lower-left benzene rings, so one may dually define a \emph{bottom} condition using the lower-left diagram, together with an attendant \emph{bottom basis}. In this way, while we focus in this paper on the top basis, we actually introduce two families of rotation-invariant bases of $\Inv(\bigwedge\nolimits_q^{\underline{c}} V_q)$, one for the top condition and one for the bottom condition. At $q=1$, these families are interchanged under reflection, up to reversing the type $\underline{c}$. In the standard case when $q=1$, reflection corresponds to the action of the long permutation $w_0$. By contrast, the sets of $r=2$ and $r=3$ basis webs are both rotation- and reflection-invariant (see \cite{Patrias-Pechenik} for discussion). From the behavior of our basis under the action of $w_0$, we suspect that our basis differs from all of the bases discussed in \Cref{sec:intro_bases}. We further suspect the combinatorial top and bottom conditions can be interpreted geometrically in terms of affine $A_3$ buildings (cf.\ \cite{Fontaine-Kamnitzer-Kuperberg}).

Our results for $\Inv(\bigwedge^{\underline{c}} V)$ can also be extended (see \Cref{thm:semistandard}) to give a web basis for spaces $\Inv(\Sym^{\underline{\mu}} V)$ of invariants in tensor products of symmetric powers of $V$. These spaces connect a range of important objects in algebra and geometry. For example, consider a Grassmannian $\mathrm{Gr}_r(\mathbb{C}^n)$ with respect to its Pl\"ucker embedding into projective space. We may recover multi-homogeneous parts of the homogeneous coordinate ring of $\mathrm{Gr}_r(\mathbb{C}^n)$ as $\Inv(\Sym^{\underline{\mu}} V)$ for appropriate choices of $\underline{\mu}$. Moreover, certain multi-homogeneous parts of this coordinate ring yield concrete constructions of Schur and Specht modules, irreducible representations of $\SL_n(\mathbb{C})$ and the symmetric group $\mathfrak{S}_n$, respectively. The homogeneous coordinate ring of $\mathrm{Gr}_r(\mathbb{C}^n)$ also carries an important cluster algebra structure \cite{Fomin.Zelevinsky:cluster2}, conjecturally related to web bases, as do other rings of invariant polynomials \cite{Fomin-Pylyavskyy-advances}.
 
\subsection{Outline}
\Cref{sec:prelim} gives preliminaries on webs and their associated quantum invariants and on plabic graphs. We use a less standard presentation of webs---due essentially to \cite{Fraser-Lam-Le} in the $q=1$ case---in which trivalence is not assumed, so even experts may wish to consult this section. 

In \Cref{sec:hourglass-and-six-vertex}, we develop the theory of hourglass plabic graphs and full reducedness. This is applied in \Cref{sec:backwards-map} to define and study the map $\mathcal{T}$ of \Cref{thm:B}. The growth rules defining the inverse map $\mathcal{G}$ are developed in \Cref{sec:forwards-map}; the growth rules also allow for the leading terms of basis elements to be read off. In \Cref{sec:bijection-and-basis}, these results are combined to obtain the bijection of \Cref{thm:B} and subsequently the web basis $\mathscr{B}_q^{\underline{c}}$ of \Cref{thm:A}. 

In \Cref{sec:reduction-rules}, we show how arbitrary web invariants may be expressed in the basis $\mathscr{B}_q^{\underline{c}}$. These skein relations are of considerable interest in many applications of webs, particularly in relation to quantum link invariants.

\Cref{sec:cc-asm-pp} contains several combinatorial applications of our work to cyclic sieving, alternating sign matrices, and plane partitions.

Finally, in \Cref{sec:two_column}, we discuss how the previously-known rotation-invariant web bases fit into our unified framework of fully reduced hourglass plabic graphs with $\trip_{\bullet}=\prom_{\bullet}$ and how our results allow for the extension of these bases to the semistandard setting.

An extended abstract describing part of this work appears in the proceedings of FPSAC 2023 \cite{fpsac-abstract}.

\section{Preliminaries}
\label{sec:prelim}

Let $[r] \coloneqq \{1, 2, \dots, r\}$, $\overline{[r]} \coloneqq \{-1, -2, \dots, -r\}$, and $\pm[r]\coloneqq \{\pm1,\pm2,\ldots,\pm r\}$. We define $\mathcal{A}_r$ as the collection of subsets of~$\pm [r]$ whose elements are all of the same sign. We often write $\overline{i}$ for $-i$.  

\subsection{Webs and tensor invariants}\label{sec:webs_and_tensors}
The quantum group $U_q(\fsl_r)$ is a $\mathbb{C}(q)$-algebra that arose from the theory of quantum integrable systems; see \cite{Lusztig-Quantum-Groups} for a definition by generators and relations.  The quantum group $U_q(\fsl_r)$ deforms the universal enveloping algebra $U(\fsl_r)$, so the classical theory of $\fsl_r$ is recovered at $q=1$. For all of the objects introduced below which are subscripted by $q$, we make the convention that the subscript is omitted when specializing $q=1$.

Explicit diagrammatic generators and relations for the category of finite-dimensional $U_q(\fsl_r)$-modules were obtained in \cite[\S2, \S3]{Cautis-Kamnitzer-Morrison} (see also \cite{Kim, Morrison}), where the generators are drawn as certain trivalent graphs. For our purposes, it is important to use instead an essentially equivalent formulation in terms of \emph{tagged} $r$-valent graphs, due essentially to \cite{Fraser-Lam-Le}.

We are primarily interested in the following $U_q(\fsl_r)$-modules and their tensor products. 
\begin{itemize}
    \item The standard $U_q(\fsl_r)$-module $V_q$ (deforming the defining representation of $\fsl_r$) has standard $\mathbb{C}(q)$-basis $v_1, \ldots, v_r$.
    \item The \emph{quantum exterior algebra} $\bigwedge_q^\bullet V_q$ (see \cite{Berenstein-Zwicknagl,Cautis-Kamnitzer-Morrison})  is a $U_q(\fsl_r)$-module (deforming the classical exterior algebra). On the generators $v_i$, the quantum exterior product satisfies the $q$-commutation relations 
    \[v_i \wedge_q v_j =
      \begin{cases}
          (-q)\  v_j \wedge_q v_i  &\text{if }i < j, \\
          0 &\text{if }i = j.
      \end{cases}
      \]
   The quantum exterior power $\bigwedge\nolimits_q^c V_q$ has $\mathbb{C}(q)$-basis $\{v_I : I \subseteq [r], |I| = c\}$, where we write $v_I \coloneqq v_{i_1} \wedge_q \cdots \wedge_q v_{i_c}$ for $I = \{i_1 > \cdots > i_c\}$, following~\cite[p.363]{Cautis-Kamnitzer-Morrison}.
    \item The linear dual $(\bigwedge\nolimits_q^c V_q)^*$ is a $U_q(\fsl_r)$-module with dual basis $\{ v_I^* \}$.
\end{itemize}

As a shorthand, we write $\bigwedge\nolimits^{-c}_q V_q \coloneqq (\bigwedge\nolimits_q^c V_q)^*$.
 For $I \subseteq [r]$, we write $\overline{I} \in \mathcal{A}_r$ for the corresponding subset of negative numbers and set $v_{\overline{I}} \coloneqq v_I^*$. 

\begin{definition}\label{def:wedge_c}
Given a \emph{type} $\underline{c} = (c_1, \ldots, c_n)$ where each $c_j \in \pm [r]$, let
  \[ \bigwedge\nolimits_q^{\underline{c}} V_q \coloneqq \bigwedge\nolimits_q^{c_1} V_q \otimes \cdots \otimes \bigwedge\nolimits_q^{c_n} V_q. \]
\end{definition}

In this paper, we limit attention to the  $U_q(\fsl_r)$-modules described in \Cref{def:wedge_c}. As the \emph{Karoubi envelope} of the category of such modules recovers the entire category of finite-dimensional representations, this is not a major restriction.

The natural product map $\bigwedge\nolimits_q^{c_1} V_q \otimes \bigwedge\nolimits_q^{c_2} V_q \to \bigwedge\nolimits_q^{c_1+c_2} V_q$ given by $u \otimes v \mapsto u \wedge_q v$ is $U_q(\fsl_r)$-equivariant \cite[\S3.1]{Cautis-Kamnitzer-Morrison}. It may be described explicitly by $v_I \wedge_q v_J \mapsto (-q)^{\ell(I, J)} v_{I \cup J}$ when $I \cap J = \varnothing$ and $\ell(I, J) = |\{(i, j) \in I \times J : i < j\}|$. The more general $n$-ary product map is associative.

Elements of $\mathcal{H}=\Hom_{U_q(\fsl_r)}\left(\bigwedge\nolimits_q^{\underline{c}} V_q, \bigwedge\nolimits_q^{\underline{d}} V_q\right)$ may be obtained by composition and tensoring from the natural product map as well as natural $U_q(\fsl_r)$-module maps arising from coproducts, duals, evaluation, coevaluation, and the identity (plus the \emph{pivotal isomorphism} between a representation and its double dual). Each such morphism has explicit transition coefficients\label{pg:transition_coeff} which are all integral powers of~$q$. Webs as in \cite{Morrison,Cautis-Kamnitzer-Morrison} are planar graphs between two horizontal lines which are locally given by the building blocks in \Cref{fig:CKM-building-blocks} and which represent elements of $\mathcal{H}$; we will call these \emph{CKM-style webs}. The  morphism corresponding to a CKM-style web is unchanged under planar isotopies fixing the boundary lines. The domain and codomain types $\underline{c}$ and $\underline{d}$ may be read off from edges incident to the upper and lower horizontal lines, respectively, by reading edge multiplicities from left to right, where upward oriented edges pick up a negative and correspond to dual exterior powers.

\begin{figure}[H]
    \centering
    \includegraphics[width=0.8\linewidth]{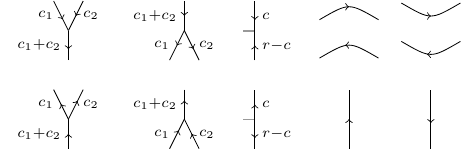}
    \caption{Building blocks of CKM-style webs for $U_q(\fsl_r)$ corresponding to products, coproducts, duals, etc. Multiplicities are in $[r]$, and multiplicities on the rightmost diagrams are omitted. Upward arrows indicate duals.}
    \label{fig:CKM-building-blocks}
\end{figure}

  The middle two diagrams in \Cref{fig:CKM-building-blocks} include a short, undirected edge called a \emph{tag}, introduced in \cite{Morrison}. These middle diagrams may be thought of as degenerate cases of the trivalent vertex diagrams where $c_1+c_2=r$. When $c=r$, the middle diagrams correspond to the $U_q(\fsl_r)$-isomorphism $\bigwedge\nolimits_q^r V_q \cong \mathbb{C}(q)$ and its dual \cite[\S3.2]{Cautis-Kamnitzer-Morrison}. Unlike other edges, tags end as leaves in the middle of the diagram.
  
  As we will see, tags may be moved across an edge of multiplicity $c$ at the cost of multiplying the invariant by $(-1)^{c(r-c)}$. When $r$ is odd, this factor is always $1$, and tags may be omitted completely. Tags are hence a comparatively unimportant bookkeeping device and are frequently ignored or de-emphasized. Our web bases for $r=4$ will come with essentially canonical choices of tags, allowing us to ultimately omit them from our main result.

\begin{remark}
  The conventions in \cite{Morrison} and \cite{Cautis-Kamnitzer-Morrison} differ in their handling of certain negative signs, and other sources such as \cite{MOY} differ in their handling of powers of $q$ or other details. We follow the algebraic conventions of \cite{Cautis-Kamnitzer-Morrison} throughout. However, for consistency with standard numbering conventions on plabic graphs, we draw CKM-style webs with the domain on top rather than on bottom. That is, our CKM-style webs are obtained by vertically flipping the diagrams from \cite{Cautis-Kamnitzer-Morrison}.
\end{remark}

We now turn to webs as in \cite[\S6]{Fraser-Lam-Le} (where they are referred to as \emph{$r$-weblike subgraphs}). While \cite{Fraser-Lam-Le} explicitly consider only the case $q=1$, they note that their diagrammatic construction nonetheless applies over $U_q(\fsl_r)$ using the framework of \cite{Cautis-Kamnitzer-Morrison}. For our purposes, it suffices to consider only the \emph{(tensor) invariants}
  \[ [W]_q \in \Inv_{U_q(\fsl_r)}\left(\bigwedge\nolimits_q^{\underline{c}} V_q \right) \coloneqq \Hom_{U_q(\fsl_r)}\left(\bigwedge\nolimits_q^{\underline{c}} V_q, \mathbb{C}(q)\right), \]
for $W$ a CKM-style web without vertices on the lower horizontal boundary; we draw such webs by closing the upper horizontal boundary to a circle. See \Cref{ex:CKM-FLL-example}.

\begin{definition}
\label{def:U-q-web}
A \emph{$U_q(\fsl_r)$-web} (or just \emph{web}) is a planar graph $W$ embedded in the disk such that
\begin{itemize}
    \item $W$ is properly bicolored, with black and white vertices;
    \item all vertices on the boundary circle have degree $1$;
    \item boundary vertices are labeled $b_1, b_2, \ldots, b_n$ in clockwise order;
    \item all internal vertices have a special ``tag'' edge, which immediately terminates in the interior and not at a vertex;
    \item non-tag edges have multiplicities from $[r]$;
    \item all vertices in the interior of the disk have incident edge multiplicities summing to $r$.
\end{itemize}
These webs are defined up to planar isotopy fixing the boundary circle.
\end{definition}

Given a $U_q(\fsl_r)$-web $W$, the corresponding tensor invariant $[W]_q$ is obtained by converting $W$ to a CKM-style web as in \Cref{fig:CKM-FLL-translations}. A boundary vertex with edge multiplicity $c$ corresponds to a domain factor $\bigwedge\nolimits_q^c V_q$ for a black vertex and $\bigwedge\nolimits_q^{-c} V_q$ for a white vertex. The internal vertices of the $U_q(\fsl_r)$-web correspond to composites \[
\bigwedge\nolimits_q^{c_1} V_q \otimes \cdots \otimes \bigwedge\nolimits_q^{c_k} V_q \to \bigwedge\nolimits_q^r V_q \to \mathbb{C}(q).
\]

\begin{figure}[ht]
    \centering
    \includegraphics[width=0.9\linewidth]{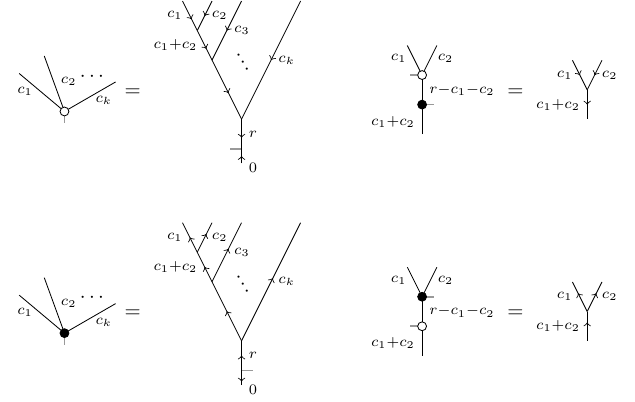}
    \caption{Translations between   $U_q(\fsl_r)$-webs and CKM-style webs. Since $0$-edges in CKM-style webs correspond to the unit object $\mathbb{C}(q)$, they are typically unwritten.}
    \label{fig:CKM-FLL-translations}
\end{figure}

\begin{example}\label{ex:CKM-FLL-example}
Consider the following CKM-style web and the corresponding $U_q(\fsl_r)$-web.
\begin{center}
    \includegraphics[scale=1.2]{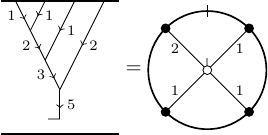}
\end{center}
They represent the same tensor invariant in $\Hom_{U_q(\fsl_5)}\left(V_q \otimes V_q \otimes V_q \otimes \bigwedge\nolimits_q^2 V_q, \mathbb{C}(q)\right)$, which is similar to a determinant. Boundary vertices in the right web are implicitly labeled $b_1, b_2, b_3, b_4$ like the face of a clock. The ``point at infinity'' separating vertex $b_1$ from vertex $b_4$ is marked on the boundary. Later in \Cref{ex:FLL-HPG-example}, we interpret this as a $U_q(\fsl_5)$-hourglass plabic graph.
\end{example}

As noted above, the positions of tags around an internal vertex only affect the sign of the corresponding tensor invariant. More precisely, we have the following.

\begin{lemma}\label{lem:tag-sign}
    The relations in \Cref{fig:tag-sign} hold, as do those obtained by reversing all arrows or switching all vertex colors.
\end{lemma}

\begin{proof}
    The first relation is \cite[(2.3)]{Cautis-Kamnitzer-Morrison}. The second is a consequence of \cite[(2.3), (2.6), and (2.7)]{Cautis-Kamnitzer-Morrison}.
\end{proof}

\begin{figure}[H]
    \centering
    \includegraphics[scale=1.2]{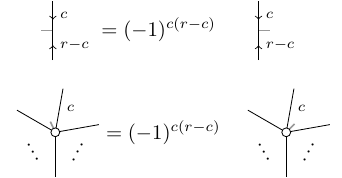}
    \caption{Tag sign relations for CKM-style webs and $U_q(\fsl_r)$-webs.}
    \label{fig:tag-sign}
\end{figure}

\subsection{Polynomial invariants and proper labelings}\label{sec:q.Laplace}

A $U_q(\fsl_r)$-web $W$ of type $\underline{c}$ may be considered as a $\mathbb{C}(q)$-multilinear map
  \[ \bigwedge\nolimits_q^{c_1} V_q \times \cdots \times \bigwedge\nolimits_q^{c_n} V_q \to \mathbb{C}(q). \]
Write elements in $\bigwedge\nolimits_q^c V_q$ as $\sum_{|S| = |c|} x_S v_S$ for scalars $x_S$ with $S\in \mathcal{A}_r$. Thus we may interpret the invariant $[W]_q$ as a polynomial
  \[ [W]_q \in \mathbb{C}(q)[x_{S, i}] \]
where for each $i$, $S$ ranges over cardinality $|c_i|$ elements of $\mathcal{A}_r$ of the same sign as $c_i$. The polynomial ring is naturally $\mathbb{Z}_{\geq 0}^n$-graded where $\deg(x_{S, i}) = \mathbf{e}_i$, and the invariants $[W]_q$ lie in the multilinear component $(1, \ldots, 1)$.

\begin{example}
  The web in \Cref{ex:CKM-FLL-example} corresponds to the polynomial
    \[ \sum_{a, b, c, \{d, e\}} (-q)^{\ell(a, b, c, \{d, e\})} x_{a, 1} x_{b, 2} x_{c, 3}, x_{\{d, e\}, 4} \]
  where the sum is over all partitions $\{a\} \sqcup \{b\} \sqcup \{c\} \sqcup \{d, e\} = \{1, \ldots, 5\}$ and the exponent $\ell$ is the number of non-inversions of the permutation $(a, b, c, d, e)$ if we choose $d > e$.
\end{example}

It is difficult to find a precise statement in the literature of a monomial expansion of general $U_q(\fsl_r)$-web invariants; see \cite{MOY}, \cite[(4.1)]{Fomin-Pylyavskyy-advances}, \cite[(5.5)]{Fraser-Lam-Le} for certain cases. We next give such a formula which is sufficient for our purposes.

\begin{definition}
\label{def:proper-labeling}
A \emph{labeling} $\phi$ of a $U_q(\fsl_r)$-web $W$ is an assignment of subsets of $[r]$ to each edge of $W$, where an edge of multiplicity $m$ is assigned a subset of size $m$. A labeling is \emph{proper} if the edges incident to each internal vertex partition the set $[r]$.

The \emph{boundary word} $\partial(\phi)$ of a labeling $\phi$ is the word $w = w_1 \cdots w_n$ in the alphabet $\mathcal{A}_r$ given by reading the labels of the edges adjacent to $b_1,b_2,\ldots,b_n$, where the sign is positive if $b_i$ is black and negative if $b_i$ is white.
\end{definition}

\begin{definition}
Given subsets $S_1, \ldots, S_m$ of $[r]$, their \emph{coinversion number} is
    \[ \ell(S_1, \ldots, S_m) = |\{(a, b) : a \in S_i, b \in S_j, a \leq b, \text{ and }i < j\}|. \]
Given a $U_q(\fsl_r)$-web $W$ with a labeling $\phi$, the coinversion number of an internal vertex $v$ is $\ell(S_1, \ldots, S_n)$ where $S_1, \ldots, S_n$ are the labels from $\phi$ around $v$ read in clockwise order starting from the tag on $v$. The coinversion number $\ell_W(\phi)$ is the sum of the coinversion numbers of all internal vertices.

Finally, suppose $w$ is a boundary word. Create a word $\hat{w}$ over $[r]$ by replacing each $\overline{S} \subseteq \overline{[r]}$ with $[r] \setminus S$ and then writing the elements of the subsets in $w$ in increasing order. The \emph{coinversion number} $\ell_W(w)$ of $w$ is $\ell_W(\hat{w})$, plus the sum of $\ell(S, [r] \setminus S)$ over all such replaced $S$.
\end{definition}

\begin{theorem}\label{thm:q.Laplace}
Let $W$ be a $U_q(\fsl_r)$-web. There are statistics $\wgt_W(\phi) \in \mathbb{Z}$ on proper labelings $\phi$ of~$W$ and $\sgn_W(w) \in \{\pm 1\}$ on boundary words $w$ such that
\begin{equation}\label{eq:q.Laplace}
\begin{split}
    [W]_q
      &= \sum_\phi (-1)^{\ell_W(\phi)} q^{\wgt_W(\phi)} x_{\partial(\phi)} \\
      &= \sum_w \left(\sum_{\phi : \partial(\phi) = w} q^{\wgt_W(\phi)}\right) \sgn_W(w) x_w,
\end{split}
\end{equation}
where $\phi$ ranges over the proper labelings of $W$, $w$ ranges over the boundary words of proper labelings, and $x_w \coloneqq x_{w_1, 1} \cdots x_{w_n, n}$. Furthermore:
\begin{itemize}
    \item If $\phi$ and $\phi'$ are proper labelings of $W$ with boundary words $w$ and $w'$, respectively, then
    \begin{equation}\label{eq:q.Laplace.2}
        (-1)^{\ell_W(\phi) - \ell_W(\phi')} = (-1)^{\ell_W(w) - \ell_W(w')}.
    \end{equation}
    In particular, $\sgn_W(w) = (-1)^{\ell_W(\phi)}$ for any proper labeling $\phi$ of $W$ with boundary word $w$.
    \item The statistic $\wgt_W(\phi)$ does not depend on the tags of $W$.
    \end{itemize}
\end{theorem}

\begin{proof}
    For the first equality in \eqref{eq:q.Laplace}, consider the transition coefficients in the $v_I$ and $v_I^*$ bases of morphisms coming from composites of products, duals, evaluations, coevaluations, and pivotal isomorphisms (see p.~\pageref{pg:transition_coeff}). An edge labeled $S$ corresponds to $v_S$, and the proper labeling condition precisely ensures the relevant product map around an internal vertex is non-zero. The boundary labeling corresponds to the input variables and hence to the monomial $x_{\partial(\phi)}$. The transition coefficients of these morphisms are of the form $(\pm q)^i$. The negative signs arise only from the product maps around internal vertices, and the overall sign is $(-1)^{\ell_W(\phi)}$. Powers of $q$ without accompanying negative signs arise only when the pivotal isomorphism is used. The first equality in \eqref{eq:q.Laplace} now follows.

    The second equality in \eqref{eq:q.Laplace} asserts that $\ell_W(\phi)$ depends only on $\partial(\phi)$. This is a consequence of \eqref{eq:q.Laplace.2}, which is equivalent to \cite[Lem.~5.4]{Fraser-Lam-Le}, in the case that $W$ has type $(1, \ldots, 1)$. The general type case follows from this one by ``attaching claws'' and tracking signs.

    Finally, let $W'$ be obtained from $W$ by moving a tag past an edge of multiplicity $k$. By \Cref{lem:tag-sign},
      \[ [W']_q = (-1)^{k(r-k)} [W]_q. \]
    If $S_1 \sqcup \cdots \sqcup S_n = [r]$ and $|S_n|=k$, then
      \[ (-1)^{\ell(S_n, S_1, \ldots, S_{n-1})} = (-1)^{k(r-k)} (-1)^{\ell(S_1, \ldots, S_n)}. \]
    Hence
      \[ (-1)^{\ell_W(\phi)} = (-1)^{k(r-k)} (-1)^{\ell_{W'}(\phi)}. \]
    Applying \eqref{eq:q.Laplace} to $[W]_q$ and $[W']_q$ now yields $\wgt_W(\phi) = \wgt_{W'}(\phi)$.
\end{proof}

Finally, we introduce a term order on $\mathbb{C}(q)[x_{i, S}]$ with respect to which our $U_q(\fsl_4)$-web bases will be unitriangular.

\begin{definition}\label{def:grevlex}
    Suppose $w = w_1 \cdots w_n$ is a sequence of elements of $\pm[r]$. Let $\widetilde{i}$ be a new symbol representing either $i$ or $\overline{r-i+1}$. Let $\widetilde{w} = \widetilde{w_1} \cdots \widetilde{w_n}$. Define an order $w >_{\tlex} w'$ if and only if $\widetilde{w}$ is lexicographically greater than $\widetilde{w'}$, where $\widetilde{1} < \widetilde{2} < \cdots < \widetilde{r}$. Extend $\widetilde{\bullet}$ and $>_{\tlex}$ to sequences $w = w_1 \cdots w_n$ of elements of $\mathcal{A}_r$, writing subsets in increasing $\mathbb{Z}$-order. The \emph{total degree} $\deg(w)$ is~$n$, which is also the total degree of $x_w$.

    The \emph{grevlex order} on $\mathbb{C}(q)[x_{i, S}]$ is given by
      \[ x_w <_{\grevlex} x_{w'} \qquad\Leftrightarrow\qquad \deg(w) < \deg(w') \text{ or } (\deg(w) = \deg(w') \text{ and } \widetilde{w} >_{\tlex} \widetilde{w'}). \]
\end{definition}

The tensor invariant associated to a web has $\grevlex$-leading monomial given by proper labelings with the $\tlex$-minimal boundary word. In \Cref{sec:forwards-map}, we will show that there is a unique proper labeling with $\tlex$-minimal boundary word associated to each of our $U_q(\fsl_4)$-basis webs.

\subsection{Plabic graphs} A major theme in this work is the view of webs as versions of \emph{plabic graphs}.

\begin{definition}[Postnikov \cite{Postnikov-arxiv}]
\label{def:plabic-graph}
A \emph{plabic graph} is a planar undirected graph, embedded in a disk, whose boundary vertices have degree $1$ and are labeled $b_1,b_2,\ldots$ in clockwise order and whose internal vertices are each colored black or white.
\end{definition}

Plabic graphs were introduced by Postnikov \cite{Postnikov-arxiv} in his study of the totally positive Grassmannian, and have since proven to be fundamental objects in the theories of cluster algebras \cite{Scott} and KP solitons \cite{Kodama-Williams} and in the physics of scattering amplitudes \cite{Arkani-Hamed-book}. See \cite{Postnikov-ICM} for a recent survey.

\begin{definition}
\label{def:plabic-trip-perm}
For each boundary vertex $b_i$ of a plabic graph $G$, one may obtain another boundary vertex $b_{\pi(i)}$ by beginning a walk on the unique edge incident to $b_i$ and turning left at each white vertex and right at each black vertex, until the boundary is reached. We call the sequence of vertices and edges in such a walk a \emph{trip strand}. The function $\pi$ is a permutation, called the \emph{trip permutation} of $G$ and denoted $\trip(G)$. The turning rules are called \emph{the rules of the road}.
\end{definition}

\begin{figure}[ht]
    \centering
    \includegraphics{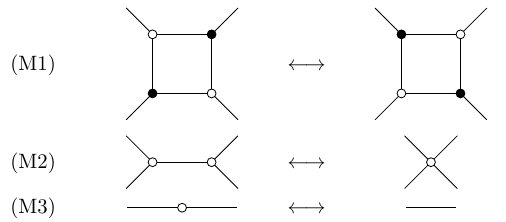}
    \caption{The moves for plabic graphs. }
    \label{fig:plabic-moves}
\end{figure}

\begin{definition}
Two plabic graphs $G$ and $G'$ are \emph{move-equivalent}, denoted $G \sim G'$ if they can be connected by a sequence of the \emph{moves} shown in \Cref{fig:plabic-moves}.
\end{definition}

It is easily checked that the moves (M1), (M2), (M3) from \Cref{fig:plabic-moves} preserve trip permutations, so if $G \sim G'$, then $\trip(G)=\trip(G')$.

In addition to moves on plabic graphs, the \emph{reduction} (R1), shown in \Cref{fig:plabic-reductions}, will be important. Unlike moves, reductions do not preserve the trip permutation.

\begin{figure}[ht]
    \centering
    \includegraphics{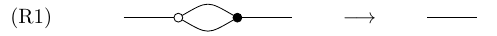}
    \caption{The reduction for plabic graphs. }
    \label{fig:plabic-reductions}
\end{figure}

A plabic graph is called \emph{leafless} if it contains no leaves (i.e., vertices of degree $1$), other than perhaps leaves incident to the boundary. An \emph{isolated component} is a connected component that is not connected to the boundary.

\begin{definition}
\label{def:plabic-reduced}
A leafless plabic graph $G$ with no isolated components is called \emph{reduced} if no $G' \sim G$ admits a parallel edge reduction (R1).
\end{definition}

Two trip strands $\ell \neq \ell'$ in a plabic graph $G$ are said to have a \emph{bad double crossing} if they traverse some edge $e$ in opposite directions and subsequently (in the forward direction on both strands) traverse another edge $e'$ in opposite directions. A trip strand has an \emph{essential self-intersection} if it traverses some edge in both directions. Finally, $G$ is said to have a \emph{round trip} if following the rules of the road (\Cref{def:plabic-trip-perm}) starting from some edge not incident to the boundary produces an infinite walk.

\begin{theorem}[Postnikov \cite{Postnikov-arxiv}]
\label{thm:plabic-reduced-iff-no-bad-crossings}
A leafless plabic graph $G$ with no isolated components is reduced if and only if all of the following conditions hold:
\begin{itemize}
    \item[(1)] $G$ has no round trips,
    \item[(2)] $G$ has no trip strands with essential self-intersections,
    \item[(3)] $G$ has no pair of trip strands with a bad double crossing, and
    \item[(4)] if $\trip(G)(i)=i,$ then $G$ has a leaf attached to the boundary vertex $b_i$.
\end{itemize}
\end{theorem}

The \emph{decorated trip permutation} ${\trip'}(G)$ of a reduced plabic graph records the extra data of the color of the leaf attached to $b_i$ for all fixed points $\trip(G)(i)=i$.

\begin{theorem}[Postnikov \cite{Postnikov-arxiv}]
\label{thm:plabic-trip-iff-move}
Two reduced plabic graphs $G$ and $G'$ are move-equivalent if and only if $\trip'(G)=\trip'(G')$.
\end{theorem}

\section{Hourglass plabic graphs and \symm six-vertex configurations} 
\label{sec:hourglass-and-six-vertex}

 We now specialize $U_q(\fsl_r)$ to the case $r=4$ for \Cref{sec:hourglass-and-six-vertex,sec:backwards-map,sec:forwards-map,sec:bijection-and-basis,sec:reduction-rules,sec:cc-asm-pp}; see \Cref{sec:two_column} for a discussion of how our methods can be adapted to unify the known web bases in other cases.

\subsection{Hourglass plabic graphs}
An \emph{hourglass graph} $G$ is an underlying planar embedded graph $\widehat{G}$, together with a positive integer multiplicity $m(e)$ for each edge $e$. The hourglass graph $G$ is drawn in the plane by replacing each edge $e$ of $\widehat{G}$ with $m(e)>1$ with $m(e)$ strands, twisted so that the clockwise orders of these strands around the two incident vertices are the same. For $m(e) \geq 2$, we call this twisted edge an \emph{$m(e)$-hourglass}, and call an edge with $m(e)=1$ a \emph{simple edge}. Note that we do not consider a simple edge to be an hourglass. 

The \emph{degree} $\deg(v)$ of a vertex $v \in G$ is the number of edges incident to $v$, counted with multiplicity, while its \emph{simple degree} $\widehat{\deg}(v)$ is its degree in the underlying graph $\widehat{G}$.

\begin{definition}
\label{def:hourglass-plabic-graph}
An \emph{hourglass plabic graph} is a bipartite hourglass graph $G$, with a fixed proper black-white vertex coloring, embedded in a disk, with all internal vertices of degree 4, and all boundary vertices of simple degree one, labeled clockwise as $b_1,b_2,\ldots, b_n$. See \Cref{fig:hourglass-example-with-trips} and \Cref{fig:growth-example} (lower right) for examples. We consider $G$ up to planar isotopy fixing the boundary circle. We write $\col(v)=1$ for a black vertex $v$ and $\col(v)=-1$ for a white vertex.
\end{definition}

\begin{remark}
In later sections, we will view hourglass plabic graphs as $U_q(\fsl_4)$-webs, with hourglass multiplicities becoming the web edge multiplicities. However in this section we focus on developing the underlying combinatorics.
\end{remark}

\begin{definition}
\label{def:oscillating-version-of-G}
The \emph{type} $\underline{c}$ of an hourglass plabic graph $G$ is the sequence $(c_1,\ldots,c_n)$ where $c_i=\col(b_i)\deg(b_i)$ for $i=1,\ldots,n$. We say the type is \emph{oscillating}\footnote{So named because graphs of this type will be seen to correspond to \emph{(generalized) oscillating tableaux}.} if $|c_i|=1$ for all $i$. A type denoted $\underline{o}$ is assumed throughout to be oscillating. For general $G$, the \emph{oscillization} of $G$, $\osc(G)$, is the hourglass plabic graph of oscillating type obtained from $G$ by replacing each boundary vertex $b$ of degree $d>1$ (connected by a $d$-hourglass to an internal vertex $v$) with $d$ boundary vertices of the same color, each connected to $v$ by a simple edge. (If $b$ is instead connected directly to another boundary vertex $b'$ by an edge $e$, first apply the un-contraction move of \Cref{fig:contraction-hourglass-moves} to $e$ to produce internal vertices $v$ and $v'$ adjacent to $b$ and $b'$, respectively, and then proceed as above.) We call this set of $d$ simple edges in $\osc(G)$ the \emph{claw} associated to $b$.
\end{definition}

Note that, for any hourglass plabic graph $G$, the underlying graph $\widehat{G}$ is a bipartite plabic graph (\Cref{def:plabic-graph}) using the vertex colors from $G$. Unlike plabic graphs, for which a single trip permutation is traditionally considered, we will associate a $3$-tuple of trip permutations to each hourglass plabic graph. 

\begin{definition}\label{def:trip_perms}
Let $G$ be an hourglass plabic graph of oscillating type with boundary vertices $b_1,\ldots, b_n$. For $1 \leq a \leq 3$, the \emph{$a$-th trip permutation} $\trip_a(G)$ is the permutation of $[n]$ obtained as follows: for each $i$, begin at $b_{i}$ and walk along the edges of $G$, taking the $a$-th leftmost turn at each white vertex, and $a$-th rightmost turn at each black vertex, until arriving at a boundary vertex $b_{j}$. Then $\trip_a(G)(i)=j$. The walk taken is called the \emph{$\trip_a$-strand}. See \Cref{fig:hourglass-example-with-trips} for an example. Note that $\trip_1(G)^{-1}=\trip_3(G)$ and $\trip_2(G)$ is an involution. We write $\trip_{\bullet}(G)$ for the tuple of these trip permutations. For general $G$, we define $\trip_{\bullet}(G)=\trip_{\bullet}(\osc(G))$. We will sometimes view $\trip_a(G)$ as a permutation of the boundary vertices $b_1,\ldots,b_n$ themselves, rather than their indices $[n]$, and also sometimes omit the graph $G$ from the notation when it is clear from context.
\end{definition}
    
\begin{figure}[hbtp]
    \centering
    \includegraphics[scale=1.2]{figures/hourglass-with-trips.pdf} \quad
    \includegraphics[scale=1.2]{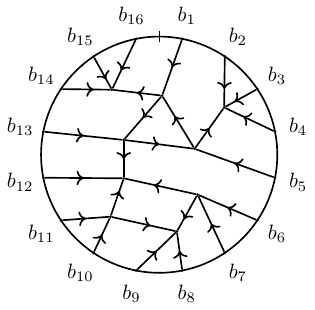}
    \caption{Left, an hourglass plabic graph $G$, with purple ($\textcolor{amethyst}{\blacksquare}$) $\trip_1$-, orange ($\textcolor{amber}{\blacksquare}$) $\trip_2$-, and green ($\textcolor{dark}{\blacksquare}$) $\trip_3$-strands drawn, showing that $\trip_i(G)(1)=5, 10,$ and $13$ for $i=1,2,$ and $3$, respectively.  Right, the corresponding six-vertex configuration, as constructed in \Cref{def:6-vertex-hourglass-bijection}.}     
    \label{fig:hourglass-example-with-trips}
\end{figure}

\begin{proposition}\label{prop:trip1_agrees}
    Suppose $G$ is an hourglass plabic graph of oscillating type. Then the $\trip_1$-strands of $G$ agree with the $\trip$-strands of $\widehat{G}$, so in particular $\trip_1(G)=\trip(\widehat{G})$.
\end{proposition}
\begin{proof}
    Observe that the sequence of vertices visited by a $\trip_1$-strand is not affected by the multiplicities of the edges it follows.
\end{proof}

As for plabic graphs, hourglass plabic graphs admit certain \emph{moves}. These are the \emph{benzene move} and \emph{square moves} shown in \Cref{fig:main-hourglass-moves} and the \emph{contraction moves} shown in \Cref{fig:contraction-hourglass-moves}. None of the vertices appearing in the moves is allowed to be a boundary vertex; hence, moves do not change the type of the hourglass plabic graph.

\begin{definition}
\label{def:hourglass-move-equivalence}
Two hourglass plabic graphs $G$ and $G'$ are \emph{move-equivalent}, written $G \sim G'$, if some sequence of benzene, square, and contraction moves transforms $G$ into $G'$. An hourglass plabic graph is \emph{contracted} if contraction moves have been applied to convert all subgraphs from the left side of \Cref{fig:contraction-hourglass-moves} to the corresponding graph from the right side. We write $\cg(\underline{c})$ for the set of contracted hourglass plabic graphs of type $\underline{c}$. Note that a contracted hourglass plabic graph of oscillating type only contains simple and $2$-hourglass edges.
\end{definition}

 For example, the move equivalence class of the contracted hourglass plabic graph $G$ from \Cref{fig:hourglass-example-with-trips} contains three other graphs, constructed by applying the benzene move, the square move, or both. 

\begin{figure}[hbtp]
    \centering
    \includegraphics[scale=1.1]{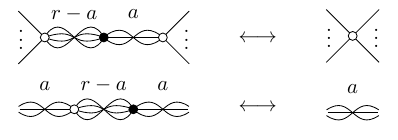}
    \caption{The contraction moves for hourglass plabic graphs. The parameter $a$ runs over $0,\ldots,r$; we have $r=4$ except in \Cref{sec:two_column}. The color reversals of these moves are also allowed.}
    \label{fig:contraction-hourglass-moves}
\end{figure}

\begin{proposition}
\label{prop:trip-of-underlying-plabic}
Let $G \sim G'$ be move-equivalent hourglass plabic graphs. Then:
\begin{itemize}
\item[(a)] $\trip_{\bullet}(G)=\trip_{\bullet}(G')$, and
\item[(b)] the underlying plabic graphs $\widehat{G}$ and $\widehat{G'}$ are move-equivalent.
\end{itemize}
\end{proposition}
\begin{proof}
Part (a) is easily checked by observing that the benzene, square, and contraction moves all locally preserve $\trip_{\bullet}$. For part (b), notice that each of the moves for hourglass plabic graphs is a composite of the plabic graph moves from \Cref{fig:plabic-moves}. For example, after passing to $\widehat{G}$, the hourglass square moves correspond to expanding the corners of the square with (M2) if necessary, applying the plabic square move (M1), and then contracting the remaining corners with another application of (M2). 
\end{proof}

\begin{remark}\label{remark:benzene}
As noted in the proof of \Cref{prop:trip-of-underlying-plabic}, the hourglass square moves and contraction moves are lifts to $G$ of (composites of) usual plabic moves on $\widehat{G}$. The benzene move, on the other hand, is trivial on the level of $\widehat{G}$.

Compositions of benzene moves on an hourglass plabic graph can be seen as a transformation on a \emph{dimer cover} (perfect matching) of the subgraph of vertices of simple degree $3$, where the hourglass edges correspond to the included edges of the dimer cover. For hexagonal lattices, this transformation is well-known to correspond to adding and removing boxes from an associated \emph{plane partition} (see e.g.~\cite{WhatIsDimer}). This is discussed further in \Cref{sec:plane-partitions}.
\end{remark}

\subsection{Fully reduced hourglass plabic graphs}

The following definition is central to this work, and should be compared to the notion of reduced plabic graph (\Cref{def:plabic-reduced}).

\begin{definition}
\label{def:fully-reduced}
An hourglass plabic graph $G$ with no isolated components is called \emph{fully reduced} if no $G' \sim G$ has a 4-cycle containing an hourglass (a \emph{forbidden 4-cycle}). We write $\rg(\underline{c})$ for the set of fully reduced hourglass plabic graphs of type $\underline{c}$. We write $\crg(\underline{c})$ for the set of these which are contracted (as in \Cref{def:hourglass-move-equivalence}).
\end{definition}

\begin{remark}
\Cref{thm:6-vertex-hourglass-correspondence} and \Cref{cor:fully-reduced-iff-monotonic} provide alternative characterizations of fully reduced graphs which avoid the (\emph{a priori} infinite) move-equivalence class exploration in \Cref{def:fully-reduced}.
\end{remark}

\begin{remark}
\Cref{def:fully-reduced} implies, in particular, that a fully reduced hourglass plabic graph $G$ does not have two or more distinct edges connecting a pair of vertices $u,v$; otherwise one of the edges could be expanded using a contraction move, creating a forbidden 4-cycle.
\end{remark}

The following proposition partially justifies the use of ``reduced" in ``fully reduced". Note that the converse  is not true: $G$ being fully reduced is a stronger condition than $\widehat{G}$ being reduced.

\begin{proposition}
\label{prop:underlying-plabic-is-reduced}
Let $G$ be a fully reduced hourglass plabic graph. Then the underlying plabic graph $\widehat{G}$ is reduced.
\end{proposition}

\Cref{prop:underlying-plabic-is-reduced} is key to the proof of
\Cref{thm:hourglass-trips-determine-move-equivalence}, which is an analogue of \Cref{thm:plabic-trip-iff-move} for hourglass plabic graphs, and is a fundamental ingredient in the construction of our web basis.

\begin{theorem}
\label{thm:hourglass-trips-determine-move-equivalence}
Two fully reduced hourglass plabic graphs $G_1$ and $G_2$ of the same type are move-equivalent if and only if $\trip_{\bullet}(G_1)=\trip_{\bullet}(G_2)$.
\end{theorem}

We now transform contracted fully reduced hourglass plabic graphs to another form which will be useful in the proofs of \Cref{prop:underlying-plabic-is-reduced} and \Cref{thm:hourglass-trips-determine-move-equivalence}. These proofs appear in \Cref{sec:move-equiv-and-trips}. 

\subsection{Symmetrized six-vertex configurations}
\label{sec:six-vertex-correspondence}
In this subsection, we define certain six-vertex configurations on $4$-valent graphs embedded in a disk and show these configurations are in bijection with contracted hourglass plabic graphs. The \emph{well-oriented} configurations (\Cref{def:symm6v-well-oriented}) correspond under this bijection to the fully reduced hourglass plabic graphs (\Cref{thm:6-vertex-hourglass-correspondence}). In the later sections, we convert freely between these objects, as convenient. 

\begin{definition}
\label{def:symm6v}Let $G$ be a planar graph embedded in a disk, with all internal vertices of degree 4 and all boundary vertices of degree $1$ labeled clockwise as $b_1,\ldots,b_n$. A \emph{\symm six-vertex configuration} $D$ on $G$ is a directed graph with underlying undirected graph $G$ such that the orientation of edges around each vertex is any rotation of:
  \[
  \includegraphics[scale=1]{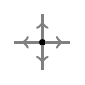} \qquad
  \includegraphics[scale=1]{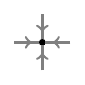} \qquad
  \includegraphics[scale=1]{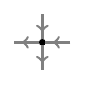}
     \]
These are \emph{sources}, \emph{sinks}, and \emph{transmitting vertices}, respectively; note that there are six possibilities for the orientation around a single vertex. We write $\ssv(G)$ for the set of \symm six-vertex configurations on $G$.
\end{definition}

\begin{remark}\label{remark:symm6vconnections}
    Symmetrized six-vertex configurations can be seen as a symmetrized version of the well-studied \emph{six-vertex model} (see e.g.\ \cite{BaxterBook}),   
a model equivalent to ours by reversing the direction of all vertical edges in the allowed vertex configurations. 
On a square grid, six-vertex configurations with domain wall boundary conditions (in our convention, corresponding to all boundary edges oriented inward) are well-known to correspond to \emph{alternating sign matrices}, matrices whose entries are in $\{0,\pm1\}$, whose rows and columns sum to $1$, and whose nonzero entries alternate in sign along each row and column (see e.g.\ \cite{Kuperberg_ASM_pf,ProppManyFaces}). On our graphs, sinks correspond to $1$, sources to $-1$, and transmitting vertices to $0$. See \Cref{sec:asm} for further discussion.
\end{remark}

Like (hourglass) plabic graphs, \symm six-vertex configurations have a notion of trip permutation.
\begin{definition} 
Let $D \in \ssv(G)$. Define $\trip_2(D)$ as the permutation of $[n]$ obtained as follows: for each $i$, begin at $b_{i}$ and walk along the edges of $G$, going straight across each vertex to the opposite edge, until arriving at a boundary vertex $b_{j}$. Then $\trip_2(D)(i)=j$ and each walk constructed by this rule is called a $\trip_2$-strand.  

We define $\trip_1(D)$ as the permutation of $[n]$ obtained as follows: for each $i$, begin at $b_{i}$ and walk along the edges of $D$, turning as indicated in \Cref{fig:six-vertex-trips} until arriving at a boundary vertex $b_{j}$. Then $\trip_1(D)(i)=j$ and each such walk is called a $\trip_1$-strand. Define $\trip_3(D)$ via reversal of the $\trip_1$-strands. Note that $\trip_1(D)^{-1}=\trip_3(D)$ and 
$\trip_2(D)$ is an involution.
\end{definition}

\begin{figure}[ht]
    \centering
    \includegraphics[scale=1.25]{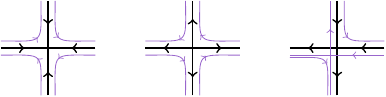}
    \caption{The behavior of $\trip_1$-strands in the symmetrized six-vertex model.}
    \label{fig:six-vertex-trips}
\end{figure} 

As with (hourglass) plabic graphs, we consider an equivalence relation on \symm six-vertex configurations generated by \emph{moves}. These are the \emph{Yang--Baxter} and \emph{ASM} moves shown in \Cref{fig:main-6v-moves}. 

\begin{remark}\label{remark:ASM_Yang_Baxter}
The moves of \Cref{fig:main-6v-moves} are so named because of their correlation to studied moves on six-vertex configurations and alternating sign matrices. The Yang--Baxter move corresponds to the \emph{star-triangle relation} associated to the \emph{Yang--Baxter equation} \cite{BaxterBook}, while the ASM move may be seen as a \emph{fiber toggle} in the alternating sign matrix tetrahedral poset~\cite{RazStrogRow} or as traveling along an edge of the \emph{alternating sign matrix polytope} \cite{StrikerASMPoly}. Appropriate compositions of these toggles produce the well-studied \emph{gyration} action of~\cite{Wieland2000} on \emph{fully-packed loops}, objects in bijection with alternating sign matrices. Fully-packed loops on generalized domains (such as ours) were considered by Cantini and Sportiello in their proof of the \emph{Razumov{\textendash}Stroganov conjecture} \cite{CantiniSportiello2} and relate to \emph{chained alternating sign matrices} \cite{ChainedASM}.
\end{remark}

\begin{definition}
Two \symm six-vertex configurations $D$ and $D'$ are \emph{move-equivalent}, written $D \sim D'$, if some sequence of Yang--Baxter and ASM moves transforms $D$ into $D'$.
\end{definition}

\begin{figure}[hbtp]
    \centering
    \includegraphics{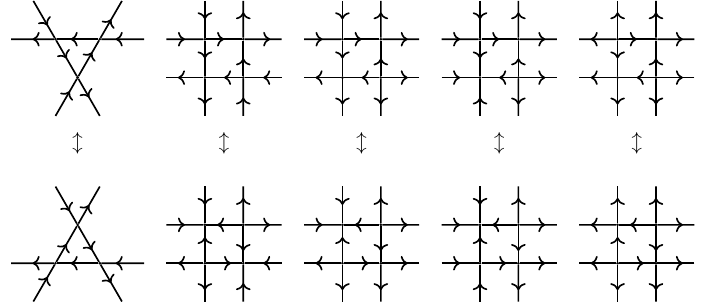}
    \caption{The Yang--Baxter move (leftmost) and the ASM moves for \symm six-vertex configurations. The edge reversals of these ASM moves are also allowed.}
    \label{fig:main-6v-moves}
\end{figure}

\begin{definition}
\label{def:symm6v-well-oriented}
    A \symm six-vertex configuration $D$ with no isolated components is \emph{well oriented}  if for every $D'$ with $D\sim D'$, the underlying undirected graph of $D'$ is simple (no loops and no $2$-cycles) and every $3$-cycle of $D'$ is (cyclically) oriented. We write $\wssv(G)$ for the set of well oriented configurations $D$ from $\ssv(G)$.
\end{definition} 

Consider the following properties that $D \in \ssv(G)$ may have:
    \begin{itemize}
         \item[(P1)] The $\trip_2$-strands of $D$ do not double-cross or self-intersect.
         \item[(P2)] The segments between the three intersection points of three pairwise crossing $\trip_2$-strands are oriented (\emph{big triangles} are oriented). 
         \item[(P3)] Among any four $\trip_2$-strands, there is a pair of strands which do not cross.
    \end{itemize}

\Cref{thm:6-vertex-hourglass-correspondence} shows that well-oriented \symm six-vertex configurations are in bijection with contracted fully reduced hourglass plabic graphs. The conditions of \Cref{def:symm6v-well-oriented} parallel those for full-reducedness. However, \Cref{prop:trip2_6V} gives an effective way to determine whether a \symm six-vertex configuration is well oriented by checking (P1) and (P2).

\begin{lemma}
\label{lem:p1-p2-implies-p3}
If $D \in \ssv(G)$ satisfies (P1) and (P2), then it satisfies (P3).
\end{lemma}

\begin{proof}  
If four $\trip_2$-strands in $D$ are pairwise crossing, then by (P1) these strands form $\binom{4}{3}$ big triangles, but there is no way to consistently direct the edges of these triangles satisfying (P2).
\end{proof}

\begin{proposition}
\label{prop:trip2_6V}
     A configuration $D \in \ssv(G)$ is well oriented if and only if (P1) and (P2) hold.
\end{proposition}
\begin{proof}
Suppose $D \in \ssv(G)$ satisfies (P1) and (P2) and $D' \sim D$. Yang--Baxter and ASM moves do not change the multiplicity of intersection of any pair of $\trip_2$-strands. Since $\trip_2$-strands in $D$ do not double-cross or self-intersect, the same is true for $D'$, and in particular $D'$ has no loops or 2-cycles. We now argue that $D'$ has oriented 3-cycles. We may assume without loss of generality that $D'$ differs from $D$ by a single move. An ASM move applied to a face $F$ of $D$ cannot create an unoriented 3-cycle, for such a cycle would need to contain an edge bordering $F$, in which case $D$ would contain an unoriented big triangle ($\triangle \ell_1 \ell_2 \ell_4$ on the left side of \Cref{fig:moves-dont-create-bad-3-cycle}), violating (P2). A Yang--Baxter move applied to $D$ cannot create an unoriented 3-cycle either, for the new 3-cycles in $D'$ (see the right side of \Cref{fig:moves-dont-create-bad-3-cycle}) are forced to be oriented by property (P2) of $D$. Thus $D'$ has oriented 3-cycles and $D$ is well oriented.

\begin{figure}[ht]
    \centering
    \includegraphics[scale=0.9]{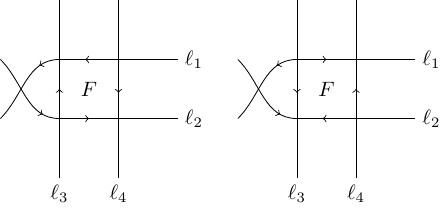}
    \includegraphics[scale=0.9]{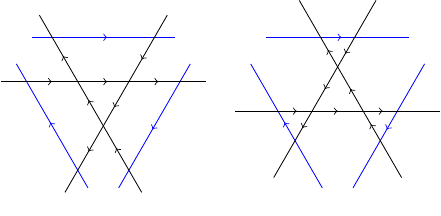} 
    \caption{Left: a 3-cycle adjacent to the face $F$ in $D$ and in $D'$ from the proof of \Cref{prop:trip2_6V}. Right: the 3-cycles potentially created by a Yang--Baxter move applied to $D$.}
    \label{fig:moves-dont-create-bad-3-cycle}
\end{figure}

Now suppose that $D$ is well oriented; we show that (P1) and (P2) must hold. We begin by proving (P2) in the case that $D$ satisfies (P1). Under this assumption, if $\ell_1,\ell_2,\ell_3$ are three pairwise crossing $\trip_2$-strands, let $T=\triangle \ell_1\ell_2\ell_3$ denote the big triangle they form.  We prove that the boundary of $T$ is oriented by induction on the number of vertices inside $T$, inclusive of the boundary. As a base case, if $T$ is a $3$-cycle, it must be oriented, since $D$ is well oriented. Otherwise, some strand must cross through $T$. Let $v$ be any vertex on the boundary of $T$, but not a corner of $T$. The strand passing into $T$ at $v$ forms a smaller big triangle $T'$ with the boundary of $T$. By induction $T'$ is oriented, so $v$ is a transmitting vertex. Thus the edge directions along any side of $T$ all agree (otherwise some $v$ along the side would be a source or sink).  Since $T'$ is oriented, the orientations of the sides it crosses, say $\ell_1$ and $\ell_2$, must agree. If $T$ is not oriented, it then follows that the remaining side $\ell_3$ has a single edge in $T$. In this case, all strands crossing through $T$ must cross $\ell_1$ and $\ell_2$ and not cross each other, forming a nested set of triangles sharing vertex $\ell_1 \cap \ell_2$ and sharing a common orientation. By a sequence of Yang--Baxter moves at $\ell_1 \cap \ell_2$, we may remove these strands from $T$. Since $D$ is well oriented, the resulting $D'$ has oriented 3-cycles; in particular, $\ell_3$ is oriented compatibly with $\ell_1$ and $\ell_2$. Thus $T$ is oriented.
    
Finally, we prove (P1). Suppose $D$ has a double-crossing or self-intersection in its $\trip_2$-strands and let $R$ denote the region of the graph this bounds. Suppose also that $R$ contains the minimal number of vertices among all such regions in all move-equivalent configurations. In particular, there are no double-crossings or self-intersections inside $R$. If $R$ were bounded by a self-intersection, any other $\trip_2$-strand crossing into $R$ would create a smaller double-crossing, so we may assume that $R$ is bounded by a double-crossing between $\trip_2$-strands $\ell_1$ and $\ell_2$. Therefore, viewing the interior of $R$ as a well-oriented \symm six-vertex configuration $D''$ with $\ell_1$ and $\ell_2$ forming the boundary circle, (P1) is satisfied. By the above argument, so is (P2), and by \Cref{lem:p1-p2-implies-p3}, $D''$ satisfies (P3). Since $D$ is well oriented, there must be a pair of strands intersecting inside $R$. Choose such a pair $a$ and $b$ so that the number of vertices in $\triangle a b \ell_1$ is minimized (see \Cref{fig:eye}). All strands cutting through this triangle must cross both $a$ and $b$, and they may not cross each other by (P3). Thus a sequence of Yang--Baxter moves may be applied at $a \cap b$ to remove these strands from the triangle. Finally, a Yang--Baxter move may be applied to $\triangle ab\ell_1$ to obtain a smaller region bounded by a double-crossing, contradicting the minimality of $R$. This completes the proof of (P1).
\end{proof}

\begin{figure}[th]
    \centering
    \includegraphics{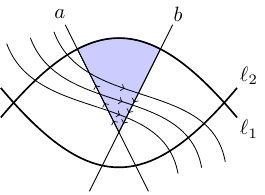}
    \caption{The diagram described in the final paragraph of the proof of \Cref{prop:trip2_6V}.}
    \label{fig:eye}
\end{figure}

The numbers $\underline{o}=(o_1,\ldots,o_n)$ where $o_i=1$ (resp.\ $o_i=-1$) if the edge incident to $b_i$ is directed inwards (resp.\ outwards) form the \emph{boundary conditions} of a \symm six-vertex configuration. We write $\ssv(\underline{o})$ (and $\wssv(\underline{o})$) for the (well-oriented) \symm six-vertex configurations with boundary conditions $\underline{o}$. 

We now define the correspondence $\varphi$ between contracted hourglass plabic graphs and \symm six-vertex configurations. 

\begin{definition}
\label{def:6-vertex-hourglass-bijection}
Given $G \in \cg(\underline{o})$, $\varphi(G)$ is the \symm six-vertex configuration formed from $G$ by orienting all simple edges from black vertex to white, removing the vertex coloring, and contracting each 2-hourglass and incident vertices to a single $4$-valent vertex. \Cref{fig:hourglass-example-with-trips} shows an example of an hourglass plabic graph and its image under $\varphi$.
\end{definition}

\begin{proposition}
\label{prop:welldefvarphi}
For $G \in \cg(\underline{o})$, we have $\varphi(G) \in  \ssv(\underline{o})$. 
\end{proposition}
\begin{proof}
  Let $D =\varphi(G)$. When a 2-hourglass in $G$ is collapsed to a vertex in $D$, the resulting vertex is still $4$-valent. Every white vertex in $G$ incident to no hourglass becomes a sink in $D$. Every black vertex in $G$ incident to no hourglass becomes a source in $D$. Each contracted hourglass results in a transmitting vertex. The map $\varphi$ converts boundary colors to edge orientations, so $\varphi(G)$ has boundary conditions $\underline{o}$. 
\end{proof}

The following notion will be used in \Cref{sec:nicelabelings}.
\begin{definition}\label{def:ssv-proper-labeling}
For $D \in \ssv(\underline{o})$, a \emph{proper labeling} is a labeling of the edges by elements of $\{1,2,3,4\}$ such that
the labels adjacent to each source or sink are distinct and each transmitting vertex has its two incoming edges labeled $a \neq b$, and its two outgoing edges labeled $a$ and $b$.
\end{definition}

\begin{theorem}
\label{thm:6-vertex-hourglass-correspondence}
The map $\varphi : \cg(\underline{o}) \to \ssv(\underline{o})$ is a bijection intertwining benzene moves with Yang--Baxter moves, intertwining square moves with ASM moves, and preserving trip permutations and proper labelings. Moreover, $G$ is fully reduced if and only if $\varphi(G)$ is well oriented. 
\end{theorem}
\begin{proof}
The map $\varphi$ has an inverse which colors sinks white, sources black, and expands each transmitting vertex to an hourglass from a white to a black vertex in the direction of transmission. The output of $\varphi^{-1}$ is contracted, since this procedure does not produce 3- or 4-hourglasses or adjacent pairs of 2-hourglasses.

The Yang--Baxter move of \Cref{fig:main-6v-moves} is the image under $\varphi$ of the benzene move of \Cref{fig:main-hourglass-moves}. The ASM moves of \Cref{fig:main-6v-moves} are the images under $\varphi$ of the corresponding square moves of \Cref{fig:main-hourglass-moves}. The preservation of trip permutations and proper labelings is by construction. The behavior in \Cref{fig:six-vertex-trips} was chosen to reflect the rules of the road after applying $\varphi$, while the labelings of \Cref{def:ssv-proper-labeling} correspond under $\varphi^{-1}$ to proper labelings in the usual sense of hourglass plabic graphs (by labeling hourglass edges with the complement in $\{1,2,3,4\}$ of the labels that appear on adjacent edges).

Suppose $G\in \crg(\underline{o})$. Then $G$ has no isolated components and no $G'\sim G$ has a $4$-cycle containing an hourglass. An unoriented $3$-cycle, a $2$-cycle, or a loop corresponds under $\varphi^{-1}$ to a $4$-cycle containing an hourglass. Thus $D$ has no isolated components and no $D' \sim D=\varphi(G)$ contains an unoriented $3$-cycle, a $2$-cycle, or a loop. By definition, this means $D$ is well oriented. The converse works similarly.
\end{proof}

\subsection{Matching diagrams}
In this subsection, we establish \Cref{prop:pm-diagrams-connected-under-Yang--Baxter}, which will be used in the proof of \Cref{thm:hourglass-trips-determine-move-equivalence}.

A \emph{matching diagram} $M$ is the underlying graph of some well-oriented \symm six-vertex configuration $D$ (with edge directions forgotten). The \emph{matching} of $M$ is $\trip_2(D)$, viewed as a matching of the boundary vertices $b_1,\ldots,b_n$. By \Cref{lem:p1-p2-implies-p3} and \Cref{prop:trip2_6V}, among any four strands in a matching diagram $M$, there is a pair which does not cross (that is, $M$ has no \emph{4-crossings}). We may apply Yang--Baxter moves to matching diagrams as to \symm six-vertex configurations, but without remembering the edge orientations.

The following is closely related to the Tits--Matsumoto theorem (see, e.g., \cite[Thm.~3.3.1]{Bjorner.Brenti}) on reduced words of permutations; see \cite[\S7]{Curtis.Ingerman.Morrow} for a related result for circular planar graphs.

\begin{proposition}
\label{prop:pm-diagrams-connected-under-Yang--Baxter}
Any matching diagrams for the same matching are connected by Yang--Baxter moves.
\end{proposition}
\begin{proof}
We will show that all matching diagrams with a fixed matching $m$ are move-equivalent by constructing a canonical matching diagram $\tilde{M}$ with matching $m$ to which all diagrams are connected.

We denote a strand in a matching diagram with endpoints $b_i$ and $b_j$ with $i<j$ by $\ell_i$. The region of the disk bounded by $\ell_i$ and the boundary segments connecting $b_i,b_{i+1},\ldots,b_j$ is called the \emph{inside} of $\ell_i$ and the remainder of the disk, bounded by $\ell_i$ and the boundary segments connecting $b_j, b_{j+1},\ldots,b_n, b_1,\ldots,b_i$ is the \emph{outside}. 

The matching diagram $\tilde{M}$ is uniquely characterized by the following property (the \emph{abc-property}): if $a<b<c$ and strands $\ell_a, \ell_b, \ell_c$ all cross each other, then $\ell_b$ and $\ell_c$ cross on the outside of $\ell_a$. The abc-property is simultaneously achievable for all $a<b<c$ since $\tilde{M}$ may be drawn according to the following recipe: Cut the boundary of the disk between $b_n$ and $b_1$ and unfold the boundary onto a line; then, for each $i=1,2,\ldots$ in order, draw $\ell_i$ with an initial vertical segment, during which it crosses exactly those previously drawn strands $\ell_j$ such that $\ell_i,\ell_j$ are required to cross in $m$; finish drawing $\ell_i$ with a horizontal segment followed by a vertical segment, during which it does not cross any previously drawn strands. See \Cref{fig:matching-diagram-recipe} for an example. Furthermore the abc-property uniquely specifies the matching diagram $\tilde{M}$, since it determines the relative positions of the crossings in every 3-crossing $\ell_a,\ell_b,\ell_c$.

Suppose that $M$ is a matching diagram with matching $m$ having some number $d>0$ of triples $a<b<c$ which violate the abc-property. Among these violations, choose the one with $a$ maximal, $b$ minimal with respect to $a$, and $c$ maximal with respect to $a,b$. By assumption, $\ell_b$ and $\ell_c$ cross on the inside of $\ell_a$. If no strands enter the region bounded by $\ell_a,\ell_b,\ell_c$, then we may immediately apply a Yang--Baxter move to eliminate this violation. Otherwise there is some strand entering this region which crosses exactly two of $\ell_a,\ell_b,\ell_c$, since a 4-crossing is impossible by \Cref{prop:trip2_6V}; see \Cref{fig:matching-diagram-recipe}. But in any case, the new strand $\ell_{a'},\ell_{b'},$ or $\ell_{c'}$ creates a new violation $\{\ell_{a'},\ell_b,\ell_c\}, \{\ell_{a},\ell_{b'},\ell_c\},$ or $\{\ell_{a},\ell_b,\ell_{c'}\}$ of the abc-property which contradicts the minimality and maximality assumptions imposed on $a,b,c$. Thus we can reduce the number of violations of the abc-property using the Yang--Baxter move, and we see by induction that $M \sim \tilde{M}$.
\end{proof}

\begin{figure}[th]
    \centering
    \includegraphics[scale=0.8]{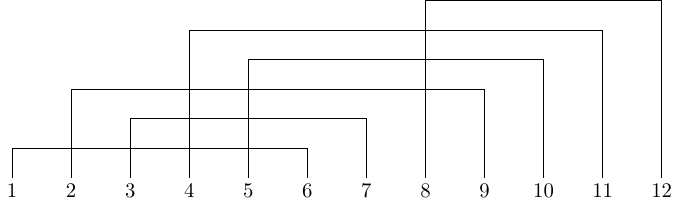}\quad
    \includegraphics[scale=0.65]{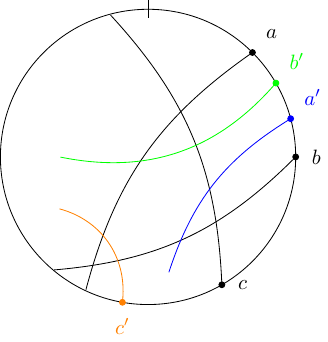}
    \caption{Left: the drawing of $\tilde{M}$ for $m=(1 \: 6) (2 \: 9) (3 \: 7) (4\: 11) (5 \: 10) (8 \: 12)$. Right: the violation of the abc-property discussed in the proof of \Cref{prop:pm-diagrams-connected-under-Yang--Baxter}.}
    \label{fig:matching-diagram-recipe}
\end{figure}

\subsection{Well-oriented configurations and monotonicity}

We now establish a key $\trip_\bullet$-theoretic property, called \emph{monotonicity}, of well-oriented configurations and fully reduced hourglass plabic graphs.
Our goal is to show that monotonicity is equivalent to being well-oriented and fully reduced. 

\begin{definition}
\label{def:six-vertex-monotonicity}
A symmetrized six-vertex configuration $D$ is \emph{monotonic} if $\trip_2$-strands do not revisit vertices or double cross, and if for every $\trip_1$-strand $\ell_1$, passing through vertices $U_1$, and every $\trip_2$-strand $\ell_2$, passing through vertices $U_2$, the vertices in the intersection $U_1 \cap U_2$ are consecutive along both $\ell_1$ and $\ell_2$. Note that we could equivalently impose this condition for $\trip_3$-strands instead of $\trip_1$-strands.
\end{definition}

\begin{theorem}
\label{thm:well-oriented-implies-monotonic}
Let $D$ be a well-oriented symmetrized six-vertex configuration, then $D$ is monotonic.
\end{theorem}

We use the following lemma in the proof of \Cref{thm:well-oriented-implies-monotonic}.
\begin{lemma}
\label{lem:left-right-alternate}
Let $\ell$ be a $\trip_1$-strand in a symmetrized six-vertex configuration $D$. Then the following hold.
\begin{itemize}
    \item[(a)] Suppose $\ell$ passes through edges $e = uv$ and $e' = vw$. Then $\ell$ turns left at $v$ if $e,e'$ are oriented toward $v$, turns right if $e,e'$ are oriented away from $v$, and goes straight through $v$ otherwise.
    \item[(b)] The turns of $\ell$ alternate between left and right.
\end{itemize}
\end{lemma}
\begin{proof}
A simple inspection of the turning rules for $\trip_1$ yields part (a). For part (b), note that if $\ell$ has just turned right at $v$, then the orientation of $\ell$ agrees with the orientation of the edge it is traversing, and this will continue to be the case as $\ell$ passes straight through any vertices. While this condition is satisfied, $\ell$ cannot turn right by part (a). Thus $\ell$ must take a left turn between any two right turns, and, symmetrically, must take a right turn between any two left turns.
\end{proof}

\begin{proof}[Proof of \Cref{thm:well-oriented-implies-monotonic}]
Let $D$ be a well-oriented symmetrized six-vertex configuration. By \Cref{prop:trip2_6V}, $\trip_2$-strands do not revisit vertices or double cross. Suppose that $D$ is not monotonic, with some $\trip_1$-strand $\ell_1$ intersecting a $\trip_2$-strand $\ell_2$, separating from it after some vertex $v$, and then re-intersecting at some vertex $v'$ (see \Cref{fig:monotonic-proof}).

Let $v_1,\ldots,v_k$ be the vertices between $v,v'$ along $\ell_1$ at which $\ell_1$ turns, and write $v=v_0, v'=v_{k+1}$ for convenience. By \Cref{lem:left-right-alternate}, these turns alternate between left and right, with inward pointing edges at left turns and outward pointing edges at right turns, and with orientations not changing along the intervening segments on which $\ell_1$ goes straight. 

If $k=1$, then $\ell_2$ and the $\trip_2$-strands $s,s'$, passing through $\{v,v_1\}$ and $\{v_1,v'\}$, respectively, form an unoriented triangle, contradicting the fact that $D$ is well oriented.

Thus we have $k \geq 2$. The polygon with vertices $v,v',v_1,\ldots,v_k$ has two $\trip_2$-strands entering the interior of the polygon at $v_2$; these strands $s$ and $s'$ must recross the boundary of the polygon elsewhere. Let $s$ be the strand passing through $v_1$ and $v_2$. 

The strand $s$ may not exit the polygon through any of the segments $\overline{v v_1}, \overline{v_1 v_2}, \overline{v_2 v_3},$ or $\overline{vv'}$, lest it create an unoriented triangle or double crossing of $\trip_2$-strands. Thus $s$ exits the polygon at some vertex $u$, further along $\ell_1$. But then, replacing $\ell_2$ by $s$, $v$ by $v_2$, and $v'$ by $u$, we have a smaller instance of $\trip_1$- and $\trip_2$-strands intersecting, separating, and re-intersecting; by induction on $k$ this is impossible.
\end{proof}

\begin{figure}[th]
    \centering
    \includegraphics[scale=0.85]{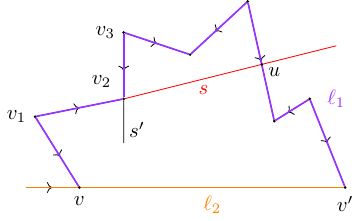}
    \caption{The strands discussed in the proof of \Cref{thm:well-oriented-implies-monotonic}.}
    \label{fig:monotonic-proof}
\end{figure}

\begin{definition}[cf.\ \Cref{def:six-vertex-monotonicity}]
\label{def:hourglass-monotonicity}
An hourglass plabic graph $G$ is \emph{monotonic} if $\trip_2$-strands do not revisit vertices or double cross, and if for every $\trip_1$-strand $\ell_1$, passing through vertices $U_1$, and every $\trip_2$-strand $\ell_2$, passing through vertices $U_2$, the vertices in the intersection $U_1 \cap U_2$ are consecutive along both $\ell_1$ and $\ell_2$. Note that we could equivalently impose this condition for $\trip_3$-strands instead of $\trip_1$-strands.
\end{definition}

\Cref{prop:monotonic-6-vertex-implies-monotonic-hourglass} is immediate from the definition of $\varphi$ (\Cref{def:6-vertex-hourglass-bijection}). 

\begin{proposition}
\label{prop:monotonic-6-vertex-implies-monotonic-hourglass}
A symmetrized six-vertex configuration $D$ is monotonic if and only if the contracted hourglass plabic graph $\varphi^{-1}(D)$ is.
\end{proposition}

\begin{proposition}
\label{prop:associated-oscillating-is-still-fully-reduced}
An hourglass plabic graph $G$ is fully reduced if and only if $\osc(G)$ is fully reduced.
\end{proposition}
\begin{proof}
A sequence of moves on $\osc(G)$ produces a forbidden 4-cycle if and only if the same sequence of moves applied to $G$ does.
\end{proof}

\begin{corollary}
\label{cor:fully-reduced-iff-monotonic}
An hourglass plabic graph $G$ is fully reduced if and only if it is monotonic.
\end{corollary}
\begin{proof}
First, suppose $G$ is fully reduced. We may assume without loss of generality that $G$ is contracted, since contraction moves do not affect monotonicity.

Suppose moreover that $G$ is of oscillating type. Then by \Cref{thm:6-vertex-hourglass-correspondence} , $D=\varphi(G)$ is well oriented, and thus by \Cref{thm:well-oriented-implies-monotonic} $D$ is monotonic. By \Cref{prop:monotonic-6-vertex-implies-monotonic-hourglass}, we conclude that $G$ is monotonic.

If $G$ is of general type, \Cref{prop:associated-oscillating-is-still-fully-reduced} implies that $\osc(G)$ is fully reduced, and so by the above argument $\osc(G)$ is monotonic. It is clear then by construction that $G$ must also be monotonic.

For the converse, \Cref{fig:forbidden-4-cycles} demonstrates that each of the forbidden 4-cycles fails to be monotonic. Furthermore, the moves for hourglass plabic graphs preserve monotonicity. Thus monotonic graphs are fully reduced.
\end{proof}

\begin{figure}[ht]
    \centering
    \includegraphics[scale=1.2]{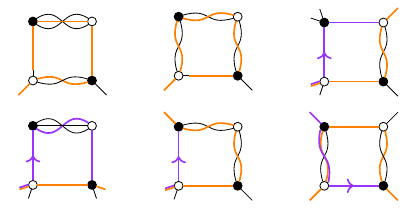}
    \caption{The 4-cycles forbidden in a fully reduced hourglass plabic graph. The $\trip_2$-strands, drawn in orange ($\textcolor{amber}{\blacksquare}$), and $\trip_1$-strands, drawn in purple ($\textcolor{amethyst}{\blacksquare}$), demonstrate the failure of monotonicity in each case.}
    \label{fig:forbidden-4-cycles}
\end{figure}

\subsection{Move equivalence and trip permutations}
\label{sec:move-equiv-and-trips}

We now give some final $\trip_\bullet$-theoretic lemmas and the proofs of \Cref{prop:underlying-plabic-is-reduced} and \Cref{thm:hourglass-trips-determine-move-equivalence}.

\begin{lemma}
\label{lem:trip1-does-not-revisit-vertex}
Let $G \in \rg(\underline{c})$, let $\ell$ be a $\trip_1$-strand in $G$, and let $v$ be an internal vertex. Then $\ell$ visits $v$ at most once.
\end{lemma}
\begin{proof}
By \Cref{cor:fully-reduced-iff-monotonic}, $G$ is monotonic. Suppose $\ell$ visits $v$ twice, enclosing a polygon with vertices $v=v_0,v_1,\ldots,v_k$. Since $G$ is bipartite, this polygon has a vertex, say $v_i$, at which some $\trip_2$-strand $\ell_2$ enters the interior of the polygon. Now, $\ell_2$ must also exit the polygon at some vertex $v_j$. Since $G$ is fully reduced, $\ell_2$ may not have a self-intersection, so $v_i \neq v_j$. The only other possibility consistent with the monotonicity of the pair $\ell_1,\ell_2$ is $v_j=v_{i+1}$. But then $G$ must contain two separate edges between $v_i$ and $v_{i+1}$ (the edges traversed by $\ell_1$ and $\ell_2$), impossible in a fully reduced graph.
\end{proof}

\begin{lemma}
\label{lem:how-trip1-trip2-diverge}
Let $G\in \crg(\underline{o})$, and let $\ell_1$ and $\ell_2$ be $\trip_1$- and $\trip_2$-strands, respectively. Let $v_1,\ldots,v_k$ be the vertices shared by $\ell_1$ and $\ell_2$, ordered according to the direction of $\ell_1$. For $i=1,\ldots,k-1$, let $e_i$ be the edge between $v_i$ and $v_{i+1}$. Then:
\begin{itemize}
    \item[(a)] The multiplicities $m(e_i)$ alternate between $1$ and $2$.
    \item[(b)] If $m(e_1)=2$ then $k=2$.
    \item[(c)] $k$ is even, so $\col(v_k) \neq \col(v_1)$; thus if $v_1$ and $v_k$ are not boundary vertices, $\ell_1$ ends on the opposite side of $\ell_2$ from which it started.
    \item[(d)] If $m(e_i)=1$, then $\col(v_i)=\col(v_1)$.
\end{itemize}
\end{lemma}
\begin{proof}
Since $G$ is contracted and of oscillating type, it has no edges of multiplicity greater than two, and no pair of consecutive 2-hourglasses. The case $m(e_i)=m(e_{i+1})=1$ is impossible, since $\ell_1$ and $\ell_2$ could not share both edges; this establishes (a).

If $m(e_1)=2$, then $\ell_1$ and $\ell_2$ must enter $v_1$ on different simple edges, share $e_1$, and then leave $v_2$ on different simple edges, so $k=2$.

A simple case check shows that if $\ell_1$ and $\ell_2$ share at least one vertex, then they share an edge, so $k=0$ or $k \geq 2$. If $k>2$, then by parts (a) and (b), we know $m(e_1)=m(e_3)=\cdots=1$ and $m(e_2)=m(e_4)=\cdots=2$. If $m(e_{k-1})=2$, then letting $\ell_3$ be the $\trip_3$-strand obtained by reversing $\ell_1$, the same arguments as above would require $k=2$. Thus we must have $m(e_{k-1})=1$ and $k$ is even, proving (c). Finally, if $m(e_i)=1$, then $k>2$, so $m(e_1)=1$; then (d) follows from (a) and the bipartiteness of $G$. 
\end{proof}

\begin{proof}[Proof of \Cref{prop:underlying-plabic-is-reduced}]
Let $G \in \rg(\underline{c})$ have underlying plabic graph $\widehat{G}$. We wish to show that $\widehat{G}$ is reduced, which we will do by verifying that the conditions of \Cref{thm:plabic-reduced-iff-no-bad-crossings} are satisfied.

First, $\widehat{G}$ is leafless, since any vertex $v$ with $\widehat{\deg}(v)=1$ must be incident in $G$ to a 4-hourglass. Since all vertices of $G$ have degree 4, the endpoints of $e$ would form an isolated component in $G$, violating full reducedness. This also implies that $\widehat{G}$ has no isolated components. 

Assume now that $G$ is of oscillating type. \Cref{prop:trip1_agrees} and \Cref{lem:trip1-does-not-revisit-vertex} ensure that $\widehat{G}$ satisfies conditions (1),(2), and (4) of \Cref{thm:plabic-reduced-iff-no-bad-crossings}, so it remains to check that $\widehat{G}$ has no bad double crossings.

Suppose that two $\trip_1$-strands $\ell_1$ and $\ell_1'$ cross; then in $G$, we have one of the local configurations shown in \Cref{fig:no-bad-double-crossing}. In each case, $\trip_2$-strands $\ell_2$ and $\ell_2'$ are indicated. By \Cref{lem:how-trip1-trip2-diverge}, the $\trip_1$-strands diverge from the $\trip_2$-strands as indicated in the figure. In particular, after the crossing, $\ell_2$ and $\ell_2'$ lie between $\ell_1$ and $\ell_1'$. Since $G$ is fully reduced, it is monotonic by \Cref{cor:fully-reduced-iff-monotonic}. Since each of $\ell_1,\ell_1'$ has already shared vertices with both $\ell_2$ and $\ell_2'$, neither can retouch these strands, and thus $\ell_1,\ell_1'$ cannot recross in the forward direction.   

If $G$ is not of oscillating type, the above argument implies that $\widehat{\osc(G)}$ is reduced. It then follows that $\widehat{G}$ is likewise reduced.
\end{proof}

\begin{figure}[ht]
    \centering
    \includegraphics[scale=1.5]{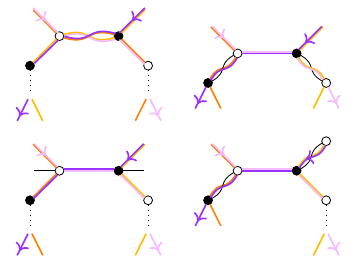}
    \caption{The possibilities for the local configuration of a pair of crossing $\trip_1$-strands $\ell_1$ (light purple $\textcolor{lavender}{\blacksquare}$) and $\ell_1'$ (dark purple $\textcolor{amethyst}{\blacksquare}$). In each case a pair $\ell_2$ (orange $\textcolor{amber}{\blacksquare}$) and $\ell_2'$ (yellow $\textcolor{clay}{\blacksquare}$) of $\trip_2$-strands prevents $\ell_1$ and $\ell_1'$ from crossing again, by monotonicity.}
    \label{fig:no-bad-double-crossing}
\end{figure}

\begin{lemma}
\label{lem:underlying-plabic-and-matching-diagram-determine-graph}
Let $G_1,G_2 \in \rg(\underline{o})$, and suppose $\widehat{G_1}=\widehat{G_2}=P$ and that the symmetrized six-vertex configurations $\varphi(G_1)$ and $\varphi(G_2)$ have the same underlying matching diagram $M$. Then $G_1=G_2$.
\end{lemma}
\begin{proof}
We show, by induction on the number of vertices, how to uniquely recover $G$ from $P$ and $M$.

If $M$ contains some $\trip_2$-strand that does not intersect any other $\trip_2$-strand, then both $M$ and $P$ may be split in two pieces, taking that strand as part of the boundary of the pieces, giving smaller pairs of plabic graphs and matching diagrams which, by induction, may each be uniquely lifted to reconstruct $G$. Thus we may assume that $M$ contains no such strand.

In this case, there exists some consecutive boundary vertices $b$ and $b'$ which are both adjacent in $M$ to the same internal vertex $v$. The possible neighborhoods of $b, b'$ in $P$ are shown in \Cref{fig:H-M-boundary-deformation}. In the latter two cases, the edge $e$ must come from a 2-hourglass in $G$, allowing us to uniquely reconstruct this neighborhood of $G$. After doing so, we can replace the boundary segment surrounding $b,b'$ with the dashed line shown in the figure, to obtain a smaller matching diagram $M_1$ and plabic graph $P_1$. By induction, that information allows us to uniquely reconstruct the remainder of $G$.
\end{proof}

\begin{figure}[ht]
    \centering
    \includegraphics[width=0.8\linewidth]{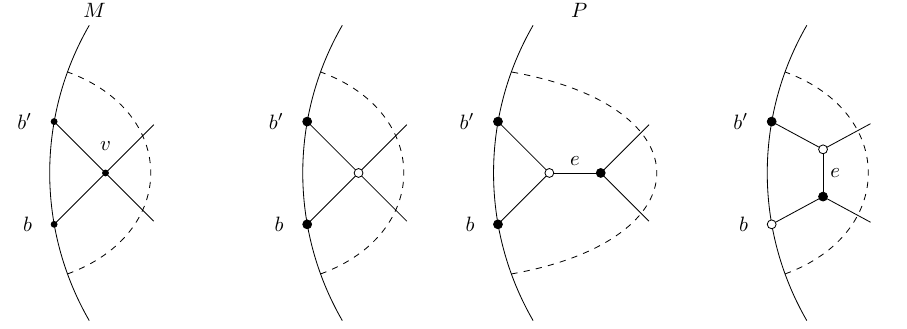}
    \caption{Consecutive boundary vertices with a common neighbor in a matching diagram $M$ (left) and the three possible corresponding configurations in the plabic graph $P$ (right). See the proof of \Cref{lem:underlying-plabic-and-matching-diagram-determine-graph}.}
    \label{fig:H-M-boundary-deformation}
\end{figure}

\begin{proof}[Proof of \Cref{thm:hourglass-trips-determine-move-equivalence}]
Let $G_1,G_2 \in \rg(\underline{c})$. If $G_1 \sim G_2$, then $\trip_{\bullet}(G_1)=\trip_{\bullet}(G_2)$ by \Cref{prop:trip-of-underlying-plabic}(a).

Conversely, suppose $\trip_{\bullet}(G_1)=\trip_{\bullet}(G_2)$. Since any hourglass plabic graph is move-equivalent to a unique contracted graph, we may also assume without loss of generality that $G_1,G_2$ are contracted.

Suppose first that $G_1,G_2$ are of oscillating type. 
By \Cref{prop:underlying-plabic-is-reduced}, the underlying plabic graphs $P_1=\widehat{G_1}$ and $P_2=\widehat{G_2}$ are reduced and by \Cref{prop:trip1_agrees}, we have  $\trip(P_1)=\trip_1(G_1)=\trip_1(G_2)=\trip(P_2)$. Thus, by \Cref{thm:plabic-trip-iff-move}, $P_1$ and $P_2$ are equivalent under the moves from \Cref{fig:plabic-moves}. Consider a sequence of moves transforming $P_1$ into $P_2$ that minimizes the number $s$ of times that the square move (M1) is applied, and also such that the only moves occurring before the first square move are those necessary to make the square move applicable. Let $F$ be the face on which the first square move acts. Since $G_1$ is bipartite, so is $P_1$, and since $G_1$ is contracted with all internal vertices of degree $4$, $P_1$ has all internal vertices of degree $3$ or $4$. Since $P_1$ is reduced, it has no $2$-gonal faces. From this we conclude that $F$ is already a bipartite square face of $P_1$, so the only moves occurring before the square move at $F$ are the expansions of any degree-$4$ corners of $F$ into two degree-$3$ vertices. Let $P_1^\dagger$ be the plabic graph obtained by applying these expansion moves, then the square move at $F$, and then contracting with their unique neighbors outside of $F$ all those corners of $F$ that were not expanded. Then $P_1^\dagger=\widehat{G_1^\dagger}$, where $G_1^\dagger$ is obtained from $G_1$ by applying the appropriate square move from \Cref{fig:main-hourglass-moves}. Furthermore, $P_1^\dagger$ can be transformed into $P_2$ by $s-1$ square moves. Hence, by induction on $s$, there exists a contracted fully reduced hourglass plabic graph $G_1'$ with $G_1 \sim G_1'$ and $\widehat{G_1'}=P_2$.

It remains to show that $G_1'$ may be transformed into $G_2$ by a sequence of benzene moves. Let $\varphi$ be the bijection from \Cref{def:6-vertex-hourglass-bijection} and define $D_1'=\varphi(G_1')$ and $D_2=\varphi(G_2)$. The symmetrized six-vertex configurations $D_1',D_2$ are well oriented by \Cref{thm:6-vertex-hourglass-correspondence} and satisfy $\trip_2(D_1')=\trip_2(D_2)$. Thus, by \Cref{prop:pm-diagrams-connected-under-Yang--Baxter}, the matching diagram $M_1'$ for $D_1'$ can be transformed into the matching diagram $M_2$ for $D_2$ by a sequence of Yang--Baxter moves. Consider the hourglass plabic graph $G_1''$ obtained by transporting these Yang--Baxter moves to benzene moves on $G_1'$ via $\varphi^{-1}$. Since benzene moves do not change the underlying plabic graph, we have $\widehat{G_1''}=\widehat{G_1'}=P_2=\widehat{G_2}$, and by construction the matching diagram for $\varphi(G_1'')$ is equal to the matching diagram $M_2$ for $\varphi(G_2)$. By \Cref{lem:underlying-plabic-and-matching-diagram-determine-graph} we conclude that in fact $G_1''=G_2$, as desired.

If $G_1,G_2$ are not of oscillating type, we oscillize to $\osc(G_1),\osc(G_2)$, which are fully reduced by \Cref{prop:associated-oscillating-is-still-fully-reduced}. Indeed, these still satisfy $\trip_{\bullet}(\osc(G_1))=\trip_{\bullet}(\osc(G_2))$, so by the previous argument, $\osc(G_1) \sim \osc(G_2)$. Since benzene and square moves cannot be applied to boundary faces, the sequence of moves yielding this move equivalence can also be applied to yield $G_1 \sim G_2$.
\end{proof}

The following useful fact is immediate from the preceding proof.

\begin{corollary}
\label{cor:benzene-then-square}
Suppose that $G, G' \in \crg(\underline{c})$ are move-equivalent. Then they may be connected by a sequence of benzene moves followed by a sequence of square moves, or vice versa.
\end{corollary}

\section{Fluctuating tableaux and separation labelings}
\label{sec:backwards-map}

In this section, we recall the notion of \emph{fluctuating tableaux} of $r$ rows from \cite{fluctuating-paper}. For $r = 4$, we establish a map (shown later, in \Cref{thm:main-bijection}, to be a bijection) between move-equivalence classes of fully reduced hourglass plabic graphs and rectangular fluctuating tableaux. See \Cref{sec:two_column} for related constructions outside the $r=4$ setting.

\subsection{Fluctuating tableaux and promotion permutations}
\label{sec:background-fluctuating}
In this subsection, we recall needed material from the companion paper \cite{fluctuating-paper}.

A \emph{generalized partition} with $r$ rows is a tuple $\lambda = (\lambda_1,\ldots, \lambda_r) \in \mathbb{Z}^r$ where $\lambda_1 \geq \cdots \geq \lambda_r$. We visualize generalized partitions as \emph{diagrams}, which are semi-infinite collections of \emph{cells}. 

We write $\mathbf{e}_i$ for the $i$-th standard basis vector of $\mathbb{Z}^r$. If $S \in \mathcal{A}_r$ is a positive subset of $\{\pm 1,\ldots,\pm r\}$, we define $\mathbf{e}_S = \sum_{i \in S} \mathbf{e}_i$, while if $S$ is a negative subset, we define $\mathbf{e}_S = - \sum_{i \in S} \mathbf{e}_{-i}$. We say two $r$-row generalized partitions $\lambda, \mu$ \emph{differ by a skew column} if $\lambda=\mu + \mathbf{e}_S$ for some $S \in \mathcal{A}_r$. For $c \geq 0$, we furthermore write $\mu \too{c} \lambda$ if $\lambda$ is obtained from $\mu$ by adding a skew column of $c$ boxes and $\mu \too{-c} \lambda$ if $\lambda$ is obtained from $\mu$ by removing a skew column of $c$ boxes. 

\begin{definition}
An $r$-row \emph{fluctuating tableau} of \emph{length} $n$ is a sequence
\[ T = (\mathbf{0}=\lambda^0 \too{c_1} \lambda^1 \too{c_2} \cdots \too{c_n} \lambda^n) \]
of $r$-row generalized partitions such that $\lambda^{i-1}$ and $\lambda^i$ differ by a skew column obtained by adding $c_i$ or removing $-c_i$ cells for all $1 \leq i \leq n$. The partition  $\lambda^n$ is called the \emph{shape} of $T$. The sequence $\underline{c} = (c_1, \ldots, c_n) \in \{0, \pm 1, \ldots, \pm r\}^n$ is the \emph{type} of $T$. Let $\mathrm{FT}(r,\lambda, \underline{c})$ be the set of fluctuating tableaux with $r$ rows, shape $\lambda$, and type $\underline{c}$. We will drop some parameters from $\mathrm{FT}(r, \lambda, \underline{c})$ as convenient. In particular, unless otherwise specified, we assume $r=4$ in the remainder of the paper.

An $r$-row fluctuating tableau is \emph{rectangular} if its shape $\lambda$ satisfies $\lambda_1=\cdots=\lambda_r$. In this case $\lambda$ is called a \emph{generalized rectangle}. We write $\rft(\underline{c})$ for the set of rectangular fluctuating tableaux of type $\underline{c}$, where no $c_j = 0$. We say a fluctuating tableau is of \emph{oscillating type}  
  if each $c_j\in \{\pm 1\}$.
\end{definition}

\begin{remark}\label{rem:drawing_tabs}
Building on work of Patrias \cite{Patrias} in the oscillating case, we visualize fluctuating tableaux by writing $i$ in the added cells of $\lambda^i - \lambda^{i-1}$ or $\overline{i}$ in the removed cells of $\lambda^{i-1} - \lambda^i$; see \Cref{fig:ftexample} for an example.
Note that each box is labeled by an alternating string of barred and unbarred numbers. Indeed, given any sequence of $n$ symbols, we can visualize a fluctuating tableau by writing the $i$-th symbol unbarred in the added cells of $\lambda^i - \lambda^{i-1}$ or the $i$-th symbol barred in the removed cells of $\lambda^{i-1} - \lambda^i$; this extra flexibility will be useful to us in defining a jeu de taquin using the symbol $\bullet$.
\end{remark}

Fluctuating tableaux interest us because of the following fact.

\begin{proposition}[{see \cite[Prop.~2.10]{fluctuating-paper}}]
\label{prop:fluc-tab-give-dimension}
For any type $\underline{c}$, we have
\begin{equation}\label{eq:fluc-tab-give-dimension}
  |\rft(\underline{c})| = \dim_{\mathbb{C}} \Inv_{\SL_r}(\bigwedge\nolimits^{\underline{c}} V) = \dim_{\mathbb{C}(q)} \Inv_{U_q(\fsl_r)}(\bigwedge_q\nolimits^{\underline{c}} V_q).
\end{equation}
\end{proposition}

The \emph{lattice word} associated to a fluctuating tableau $T \in \ft(\underline{c})$ is the word $L = L(T) = w_1 \cdots w_n$ on 
$\mathcal{A}_r$, 
  where $\lambda^i = \lambda^{i-1} + \mathbf{e}_{w_i}$.
 We may recover $T$ from $L(T)$, so we sometimes identify $T$ and $L(T)$. 

A word $w$ is a lattice word of a fluctuating tableau if and only if for every prefix $w_1 \cdots w_k$ and every $1 \leq a \leq b \leq r$,
\begin{equation}\label{eq:lattice_inequalities}
  (\mathbf{e}_{w_1} + \cdots + \mathbf{e}_{w_k})_a \geq (\mathbf{e}_{w_1} + \cdots + \mathbf{e}_{w_k})_b,
\end{equation}
where $\mathbf{e}_{\overline{S}} = -\mathbf{e}_S$. More concretely, in each prefix we require the number of $a$'s minus the number of $\overline{a}$'s to be weakly greater than the number of $b$'s minus the number of $\overline{b}$'s. We write $\slack_{a}(w_1\cdots w_k)$ for the difference $(\mathbf{e}_{w_1} + \cdots + \mathbf{e}_{w_k})_a - (\mathbf{e}_{w_1} + \cdots + \mathbf{e}_{w_k})_{a+1}$. A fluctuating tableau is rectangular if and only if equality holds when $k=n$, in which case we call $L$ \emph{balanced}.

\begin{definition}\label{def:varpietc}
  Consider a word $L = w_1 \cdots w_n$ in the alphabet $\mathcal{A}_r$. Define
  \begin{enumerate}[(i)]
    \item $\tau(L) = \overline{w_n} \cdots \overline{w_1}$;
  \item $\varpi(L) = \varpi(w_1) \dots \varpi(w_n)$, where $\varpi(w)$ replaces each element $i$ in $w$ with $-\sgn(i)(r-|i|+1)$; and
    \item $\varepsilon(L) = \varepsilon(w_n) \dots \varepsilon(w_1)$, where $\varepsilon(w)$ replaces each element $i$ in $w$ with $\sgn(i)(r-|i|+1)$.
  \end{enumerate}
\end{definition}

    Consider $\mu \too{c} \lambda$. Let $\lambda = \mu + \mathbf{e}_{\{i_1 < \cdots < i_c\}}$. 
The \emph{oscillization} of $\mu \too{c} \lambda$ is the sequence
  \[ \osc(\mu \too{c} \lambda) \coloneqq \mu \to \mu^1 \to \cdots \to \mu^{|c| - 1} \to \lambda, \]
where
\[
\mu^j =  \mu^{j-1} + \mathbf{e}_{i_j} 
\]
with $\mu^0 \coloneqq \mu$ and $\mu^{|c|} = \lambda$.
Given a fluctuating tableau $T = (\mathbf{0}=\lambda^0 \too{c_1} \lambda^1 \too{c_2} \cdots \too{c_n} \lambda^n)$, let $\osc(T)$ be the concatenation of $\osc(\lambda^{i-1} \to \lambda^i)$ for $1 \leq i \leq n$. A fluctuating tableau is in the image of oscillization if and only if it is of oscillating type.
Oscillization may be described easily in terms of lattice words. To compute $L(\osc(T))$, we simply erase the braces from $L(T)$, where elements of each letter are written in increasing order, keeping the same initial and final shape. 

\ytableausetup{boxsize=0.85cm}
\begin{figure}[htbp]
\begin{center}
\begin{tikzpicture}
\node (tab) at (0,0) {$
T=\begin{ytableau}
  \none & 1 & 3\,\overline{7} \\
  \none & 1 & 4\,\overline{5} \\
  \none & 3\,\overline{5}\,6 \\
  \overline{2}\,3 & 6 \\
\end{ytableau}$};
\draw[ultra thick,black] ([xshift=1.70cm,yshift=.14cm]tab.south west)--([xshift=1.70cm,yshift=-.13cm]tab.north west);
\node [right = 7 em of tab] (tabo) {$
\osc(T)=\begin{ytableau}
  \none & 1 & 4\,\overline{12} \\
  \none & 2 & 7\,\overline{9} \\
  \none & 5\,\overline{8}\,10 \\
  \overline{3}\,6 & 11 \\
\end{ytableau}$};
\draw[ultra thick,black] ([xshift=2.46cm,yshift=.14cm]tabo.south west)--([xshift=2.46cm,yshift=-.13cm]tabo.north west);
\end{tikzpicture}
\end{center}
\caption{An example visual depiction of a rectangular fluctuating tableau $T$ with $n = 7$ and $\underline{c} = (2,\overline{1},2,1,\overline{2},2,\overline{1})$, together with its oscillization.}
\label{fig:ftexample}
\end{figure}

We now define promotion on fluctuating tableaux; see \cite[\S4.2,~\S4.4]{fluctuating-paper} for detailed examples and further discussion. We use the visual representation of fluctuating tableaux from \Cref{rem:drawing_tabs}.

\begin{definition}
      Fix $1 \leq i \leq n-1$. Let $T \in \ft(\underline{c})$ be a fluctuating tableau whose diagram is labeled by $1 < 2 < \cdots < i-1 < \bullet < i < \cdots$ and their negatives. 
 If $i$ and $\bullet$ appear in $T$ with opposite signs, then for each row $R$ which does not contain any of $i, \bullet, \overline{i},$ or $\overline{\bullet}$, we identify a cell $\bbb_R$ in $R$ and call it \emph{open}. Let $j$ be the largest absolute value of an entry in $R$ less than $i$, if any exist.
    \begin{itemize}
      \item[$\ast$] If the $\bullet$'s in $T$ are unbarred, let $\bbb_R$ be the cell immediately right of the cell containing $j$, or the cell containing $\overline{j}$, or if $j$ does not exist then let $\bbb_R$ be first cell of row $R$.
      \item[$\ast$] If the $\bullet$'s in $T$ are barred, let $\bbb_R$ be the cell containing $j$, or the cell immediately left of the cell containing $\overline{j}$, or otherwise the first negative cell of $R$.
    \end{itemize}

  The \emph{jeu de taquin slides} for $\JDT_{i-1}$ are the following. See \cite[Figure~4.2]{fluctuating-paper} for schematic diagrams.
  \begin{enumerate}[(a)]
  \item $\bullet$'s first move right by swapping places with $i$'s, then move down as far as possible by swapping places with $i$'s.
  \item $\overline{\bullet}$'s first move left by swapping places with $\overline{i}$'s, then move up as far as possible by swapping places with $\overline{i}$'s.
  \item all $\bullet\overline{i}$ pairs move down as far as possible into open cells, and then they move left one column.
  \item all $\overline{\bullet}i$ pairs move up as far as possible into open cells, and then move right one column.
  \end{enumerate}
  All other entries are left unchanged.     This results in a fluctuating tableau whose diagram is labeled by $1 < 2 < \cdots < i < \bullet < i+1 < \cdots$ and their negatives.
\end{definition}

\begin{definition}
  Let $T$ be a length $n$ fluctuating tableau.
  \begin{enumerate}[(i)]
    \item The \emph{promotion} $\promotion(T)$ is the result of replacing $\pm 1$'s with $\pm \bullet$'s, sequentially applying jeu de taquin slides $\JDT_1, \dots, \JDT_{n-1}$, replacing $\pm \bullet$'s with $\pm (n+1)$'s, and subtracting $1$ from each entry's absolute value.
    \item The \emph{evacuation} $\evacuation(T)$ is the result of first replacing $\pm 1$'s with $\pm \bullet_n$'s, sliding them past $\pm n$'s, replacing $\pm 2$'s with $\pm \bullet_{n-1}$'s, sliding them past $\pm n$'s, etc., and finally replacing $\pm \bullet_i$'s with $\pm i$'s.
    \item The \emph{dual evacuation} $\devacuation(T)$ is the result of first replacing $\pm n$'s with $\pm \bullet_1$'s, sliding backward past $\pm 1$'s, replacing $\pm (n-1)$'s with $\pm \bullet_2$'s, sliding backward past $\pm 1$'s, etc., and finally replacing $\pm \bullet_i$'s with $\pm i$'s.
  \end{enumerate}
\end{definition}

Promotion of fluctuating tableaux can be reduced to the oscillating case. Going forward, we often make this reduction.

\begin{lemma}[see {\cite[Lem.~4.16]{fluctuating-paper}}]
\label{lem:prom-of-osc}
    Let $T \in \ft(\underline{c})$ be a fluctuating tableau. Then
    \[
    \osc(\promotion(T)) = \promotion^{|c_1|} ( \osc(T)).
    \]
\end{lemma}

\ytableausetup{boxsize=1.05cm}
\begin{figure}
    \centering
\[
\begin{tikzpicture}
\node (begin) at (0,0) {$S = $};
\node [right = 0.0cm of begin] (0)  {\begin{ytableau}
  \none & 1 & 4\,\overline{12} \\
  \none & 2 & 7\,\overline{9} \\
  \none & 5\,\overline{8}\,10 \\
  \overline{3}\,6 & 11 \\
\end{ytableau}};
 \node [right = 0.5cm of 0] (1) {\begin{ytableau}
  \none & \bullet & 4\,\overline{12} \\
  \none & 2 & 7\,\overline{9} \\
  \none & 5\,\overline{8}\,10 \\
  \overline{3}\,6 & 11 \\
\end{ytableau}};
 \node [right = 0.5cm of 1] (2) {\begin{ytableau}
  \none & 2 & 4\,\overline{12} \\
  \none & \bullet & 7\,\overline{9} \\
  \none & 5\,\overline{8}\,10 \\
  \overline{3}\,6 & 11 \\
\end{ytableau}};
 \node [right = 0.5cm of 2] (3) {\begin{ytableau}
  \none & 2 & 4\,\overline{12} \\
  \none & \bullet & 7\,\overline{9} \\
  \none & 5\,\overline{8}\,10 \\
  \overline{3}\,6 & 11 \\
\end{ytableau}};
 \node [below = 0.5cm of 3] (4) {\begin{ytableau}
  \none & 2 & 4\,\overline{12} \\
  \none & \bullet & 7\,\overline{9} \\
  \none & 5\,\overline{8}\,10 \\
  \overline{3}\,6 & 11 \\
\end{ytableau}};
 \node [left = 0.5cm of 4] (5) {\begin{ytableau}
  \none & 2 & 4\,\overline{12} \\
  \none & 5 & 7\,\overline{9} \\
  \none & \bullet \,\overline{8}\,10 \\
  \overline{3}\,6 & 11 \\
\end{ytableau}};
 \node [left = 0.5cm of 5] (6)  {\begin{ytableau}
  \none & 2 & 4\,\overline{12} \\
  \none & 5 & 7\,\overline{9} \\
  \none & \bullet \,\overline{8}\,10 \\
  \overline{3}\,6 & 11 \\
\end{ytableau}};
 \node [left = 0.5cm of 6] (7) {\begin{ytableau}
  \none & 2 & 4\,\overline{12} \\
  \none & 5 & 7\,\overline{9} \\
  \none & \bullet \,\overline{8}\,10 \\
  \overline{3}\,6 & 11 \\
\end{ytableau}};
 \node [below = 0.5cm of 7] (8) {\begin{ytableau}
  \none & 2 & 4\,\overline{12} \\
  \none & 5 & 7\,\overline{9} \\
  \none & 10 \\
  \overline{3}\,6 \, \overline{8} \, \bullet & 11 \\
\end{ytableau}};
 \node [right = 0.5cm of 8] (9) {\begin{ytableau}
  \none & 2 & 4\,\overline{12} \\
  \none & 5 & 7\,\overline{9} \\
  \none & 10 \\
  \overline{3}\,6 \, \overline{8} \, \bullet & 11 \\
\end{ytableau}};
 \node [right = 0.5cm of 9] (10) {\begin{ytableau}
  \none & 2 & 4\,\overline{12} \\
  \none & 5 & 7\,\overline{9} \\
  \none & 10 \\
  \overline{3}\,6 \, \overline{8} \, \bullet & 11 \\
\end{ytableau}};
 \node [right = 0.5cm of 10] (11) {\begin{ytableau}
  \none & 2 & 4\,\overline{12} \\
  \none & 5 & 7\,\overline{9} \\
  \none & 10 \\
  \overline{3}\,6 \, \overline{8} \, 11 & \bullet \\
\end{ytableau}};
 \node [below = 0.5cm of 11] (12) {\begin{ytableau}
  \none & 2 & 4\,\overline{12} \\
  \none & 5 & 7\,\overline{9} \\
  \none & 10 \\
  \overline{3}\,6 \, \overline{8} \, 11 & \bullet \\
\end{ytableau}};
 \node [left = 0.5cm of 12] (13) {\begin{ytableau}
  \none & 2 & 4\,\overline{12} \\
  \none & 5 & 7\,\overline{9} \\
  \none & 10 \\
  \overline{3}\,6 \, \overline{8} \, 11 & 13 \\
\end{ytableau}};
 \node [left = 0.5cm of 13] (14) {\begin{ytableau}
  \none & 1 & 3\,\overline{11} \\
  \none & 4 & 6\,\overline{8} \\
  \none & 9 \\
  \overline{2}\,5 \, \overline{7} \, 10 & 12 \\
\end{ytableau}};
  \node [left = 0.5cm of 14] (end) {$\promotion(S)=$};
 \draw[pil] (0) -- (1) node[midway,above] {}; 
 \draw[pil] (1) -- (2) node[midway,above] {$\JDT_1$};
 \draw[pil] (2) -- (3) node[midway,above] {$\JDT_2$};
 \draw[pil] (3) -- (4) node[midway,right] {$\JDT_3$};
 \draw[pil] (4) -- (5) node[midway,above] {$\JDT_4$};
 \draw[pil] (5) -- (6) node[midway,above] {$\JDT_5$};
 \draw[pil] (6) -- (7) node[midway,above] {$\JDT_6$};
 \draw[pil] (7) -- (8) node[midway,left] {$\JDT_7$};
 \draw[pil] (8) -- (9) node[midway,above] {$\JDT_8$};
 \draw[pil] (9) -- (10) node[midway,above] {$\JDT_9$};
 \draw[pil] (10) -- (11) node[midway,above] {$\JDT_{10}$};
 \draw[pil] (11) -- (12) node[midway,right] {$\JDT_{11}$};
 \draw[pil] (12) -- (13) node[midway,above] {};
 \draw[pil] (13) -- (14) node[midway,above] {$-1$};
 \foreach \i in {0,...,14}
 {
 \draw[ultra thick,black] ([xshift=1.2cm,yshift=.14cm]\i.south west)--([xshift=1.2cm,yshift=-.13cm]\i.north west);
 }
\end{tikzpicture}
\]
    \caption{Let $S = \osc(T)$ be as in \Cref{fig:ftexample}.  We compute $\promotion(S)$ by a sequence of jeu de taquin slides. Notice in particular the type (c) move that occurs with the application of $\JDT_7$. For the full promotion orbit of $S$, see \cite[Fig.~6.1]{fluctuating-paper}.}
    \label{fig:jdt}
\end{figure}
\ytableausetup{boxsize=0.8cm}

\begin{definition}[{\cite[Def.~8.21]{fluctuating-paper}}]\label{def:balance}
  Let $w_1 \ldots w_{n} \in \rft(\underline{o})$ be the lattice word of a rectangular oscillating tableau. An \emph{$i$-balance point} is an index $j$ such that $\slack_i(w_1 \ldots w_j)=0$. 
\end{definition}
See \Cref{ex:balance_point} for an example of this definition and the following two propositions.

\begin{proposition}[{\cite[Prop.~8.22]{fluctuating-paper}}]\label{prop:first_balance}
Let $w \coloneqq w_1 \ldots w_{n} \in \rft(\underline{o})$. For $i=1,\ldots,r-1$, let $j_i \geq 2$ be the unique value such that $\JDT_{j_i - 1}$ moves a $\bullet$ across the boundary between rows $i$ and $i+1$ or moves a $\overline{\bullet}$ across the boundary between rows $r+1-i$ and $r-i$; set $j_0=j_r=1$. If $w_1$ is positive, $j_i$ is the first $i$-balance point of $w$ weakly after $j_{i-1}$.  In particular, the sequence $2 \leq j_1 \leq \cdots \leq j_{r-1} \leq n$ weakly increases. If $w_1$ is negative, $j_i$ is the first $i$-balance point of $w$ weakly after $j_{i+1}$.  In particular, the sequence $n \geq j_1 \geq \cdots \geq j_{r-1} \geq 2$ weakly decreases.
\end{proposition}

We refer to the following as the ``first balance point characterization'' of promotion. It is a generalization of a characterization that is implicitly used in the literature (see e.g.~\cite{Patrias,Petersen-Pylyavskyy-Rhoades,Tymoczko2012}).

\begin{proposition}[see {\cite[Prop.~8.22]{fluctuating-paper}}]
\label{prop:promotion-from-balance-points}
Let $w \coloneqq w_1 \ldots w_{n} \in \rft(\underline{o})$ and let $j_i$ be as in \Cref{prop:first_balance}. For $j=1,\ldots,n$, let $\theta_j$ act on the letter $w_j$ by $r \mapsto r-1 \mapsto \cdots \mapsto 1 \mapsto r$ and $\bar{1} \mapsto \bar{2} \mapsto \cdots \mapsto \bar{r} \mapsto \bar{1}$. Then
\[
\promotion(w)=\shuffle \circ \left(\prod_{i=0}^{r-1} \theta_{j_i} \right)(w),
\]
where $\shuffle$ cyclically rotates the word to the left.
\end{proposition}

\begin{example}
\label{ex:balance_point}
The word of the tableau $S$ from \Cref{fig:jdt} is $w=12\bar{4}1342\bar{3}\bar{2}34\bar{1}$.
The first $1$-balance point is $2$. The first $2$-balance point weakly after the first $1$-balance point is $5$.  The first $3$-balance point weakly after the first $2$-balance point is $8$. Note in the computation of promotion in \Cref{fig:jdt}, $\JDT_1$ moves a $\bullet$ between rows $1$ and $2$, so $j_1=2$; $\JDT_4$ moves a $\bullet$ between rows $2$ and $3$, so $j_2=5$; and $\JDT_7$ moves a $\bullet$ between rows $3$ and $4$, so $j_3=8$. This illustrates \Cref{prop:first_balance}. 

Using \Cref{prop:promotion-from-balance-points}, we compute $\promotion(S)$ from $w$. First, decrement the numbers in positions $1$, $2$, $5$, and $8$, obtaining $\left(\prod_{i=0}^{r-1} \theta_{j_i} \right)(w)41\bar{4}1242\bar{4}\bar{2}34\bar{1}$. 
Then cyclically rotate the word to the left, obtaining $\shuffle\circ\left( \prod_{i=0}^{r-1} \theta_{j_i} \right)(w)=1\bar{4}1242\bar{4}\bar{2}34\bar{1}4$, which is the lattice word of $\promotion(S)$, as seen in \Cref{fig:jdt}.
\end{example}

We define \emph{promotion permutations} as follows. See \cite[$\mathsection 6$]{fluctuating-paper} for a more general construction and further details.

\begin{definition}\label{def:prom-perms}
    Let $T \in \ft(r,\underline{c})$ be a fluctuating tableau and $1\leq i\leq r-1$. Then the \emph{$i$-th promotion permutation} $\prom_i(T)(b) \equiv |a|+b-1 \pmod n$ if and only if $a$ is the unique value that crosses the boundary between rows $i$ and $i+1$ in the application of promotion to  $\promotion^{b-1}(\osc(T))$.
    We write $\prom_{\bullet}(T)$ for the tuple of these promotion permutations.
\end{definition}

\begin{definition}\label{def:prom-mats}
    Given $T \in \ft(r,\underline{c})$ and $1\leq i\leq r-1$, the \textit{$i$-th promotion matrix} is the matrix $\PM^i(T)$ where the $(b, d)$-th entry is $1$ if $\prom_i(T)(b) = d$ and $0$ otherwise. Let $\PM(T) = (\PM^1(T), \ldots, \PM^{r-1}(T))$. We may also think of $\PM(T)$ as a single matrix whose entries are tuples.
\end{definition}

\begin{example}
    The promotion permutations (in cycle notation) of the tableau $S = \osc(T)$ from \Cref{fig:ftexample} are
    \begin{align*}
        \prom_1(S)& = (1\ 2\ 7\ 10\ 11\ 4\ 5\ 6\ 3\ 12\ 9\ 8), \\
\prom_2(S) &= (1\ 5)(2\ 10)(3\ 9)(4\ 6)(7\ 12)(8\ 11), \text{ and} \\
\prom_3(S) &= (1\ 8\ 9\ 12\ 3\ 6\ 5\ 4\ 11\ 10\ 7\ 2).
    \end{align*}
    The corresponding promotion matrices are
    \[
    \PM^1(T) =
\left(\begin{smallmatrix}
\cdot & 1 & \cdot & \cdot & \cdot & \cdot & \cdot & \cdot & \cdot & \cdot & \cdot & \cdot \\
\cdot & \cdot & \cdot & \cdot & \cdot & \cdot & 1 & \cdot & \cdot & \cdot & \cdot & \cdot \\
\cdot & \cdot & \cdot & \cdot & \cdot & \cdot & \cdot & \cdot & \cdot & \cdot & \cdot & 1 \\
\cdot & \cdot & \cdot & \cdot & 1 & \cdot & \cdot & \cdot & \cdot & \cdot & \cdot & \cdot \\
\cdot & \cdot & \cdot & \cdot & \cdot & 1 & \cdot & \cdot & \cdot & \cdot & \cdot & \cdot \\
\cdot & \cdot & 1 & \cdot & \cdot & \cdot & \cdot & \cdot & \cdot & \cdot & \cdot & \cdot \\
\cdot & \cdot & \cdot & \cdot & \cdot & \cdot & \cdot & \cdot & \cdot & 1 & \cdot & \cdot \\
1 & \cdot & \cdot & \cdot & \cdot & \cdot & \cdot & \cdot & \cdot & \cdot & \cdot & \cdot \\
\cdot & \cdot & \cdot & \cdot & \cdot & \cdot & \cdot & 1 & \cdot & \cdot & \cdot & \cdot \\
\cdot & \cdot & \cdot & \cdot & \cdot & \cdot & \cdot & \cdot & \cdot & \cdot & 1 & \cdot \\
\cdot & \cdot & \cdot & 1 & \cdot & \cdot & \cdot & \cdot & \cdot & \cdot & \cdot & \cdot \\
\cdot & \cdot & \cdot & \cdot & \cdot & \cdot & \cdot & \cdot & 1 & \cdot & \cdot & \cdot
\end{smallmatrix}\right)
    \PM^2(T) = 
\left(\begin{smallmatrix}
\cdot & \cdot & \cdot & \cdot & 1 & \cdot & \cdot & \cdot & \cdot & \cdot & \cdot & \cdot \\
\cdot & \cdot & \cdot & \cdot & \cdot & \cdot & \cdot & \cdot & \cdot & 1 & \cdot & \cdot \\
\cdot & \cdot & \cdot & \cdot & \cdot & \cdot & \cdot & \cdot & 1 & \cdot & \cdot & \cdot \\
\cdot & \cdot & \cdot & \cdot & \cdot & 1 & \cdot & \cdot & \cdot & \cdot & \cdot & \cdot \\
1 & \cdot & \cdot & \cdot & \cdot & \cdot & \cdot & \cdot & \cdot & \cdot & \cdot & \cdot \\
\cdot & \cdot & \cdot & 1 & \cdot & \cdot & \cdot & \cdot & \cdot & \cdot & \cdot & \cdot \\
\cdot & \cdot & \cdot & \cdot & \cdot & \cdot & \cdot & \cdot & \cdot & \cdot & \cdot & 1 \\
\cdot & \cdot & \cdot & \cdot & \cdot & \cdot & \cdot & \cdot & \cdot & \cdot & 1 & \cdot \\
\cdot & \cdot & 1 & \cdot & \cdot & \cdot & \cdot & \cdot & \cdot & \cdot & \cdot & \cdot \\
\cdot & 1 & \cdot & \cdot & \cdot & \cdot & \cdot & \cdot & \cdot & \cdot & \cdot & \cdot \\
\cdot & \cdot & \cdot & \cdot & \cdot & \cdot & \cdot & 1 & \cdot & \cdot & \cdot & \cdot \\
\cdot & \cdot & \cdot & \cdot & \cdot & \cdot & 1 & \cdot & \cdot & \cdot & \cdot & \cdot
\end{smallmatrix}\right)
    \PM^3(T) =
\left(\begin{smallmatrix}
\cdot & \cdot & \cdot & \cdot & \cdot & \cdot & \cdot & 1 & \cdot & \cdot & \cdot & \cdot \\
1 & \cdot & \cdot & \cdot & \cdot & \cdot & \cdot & \cdot & \cdot & \cdot & \cdot & \cdot \\
\cdot & \cdot & \cdot & \cdot & \cdot & 1 & \cdot & \cdot & \cdot & \cdot & \cdot & \cdot \\
\cdot & \cdot & \cdot & \cdot & \cdot & \cdot & \cdot & \cdot & \cdot & \cdot & 1 & \cdot \\
\cdot & \cdot & \cdot & 1 & \cdot & \cdot & \cdot & \cdot & \cdot & \cdot & \cdot & \cdot \\
\cdot & \cdot & \cdot & \cdot & 1 & \cdot & \cdot & \cdot & \cdot & \cdot & \cdot & \cdot \\
\cdot & 1 & \cdot & \cdot & \cdot & \cdot & \cdot & \cdot & \cdot & \cdot & \cdot & \cdot \\
\cdot & \cdot & \cdot & \cdot & \cdot & \cdot & \cdot & \cdot & 1 & \cdot & \cdot & \cdot \\
\cdot & \cdot & \cdot & \cdot & \cdot & \cdot & \cdot & \cdot & \cdot & \cdot & \cdot & 1 \\
\cdot & \cdot & \cdot & \cdot & \cdot & \cdot & 1 & \cdot & \cdot & \cdot & \cdot & \cdot \\
\cdot & \cdot & \cdot & \cdot & \cdot & \cdot & \cdot & \cdot & \cdot & 1 & \cdot & \cdot \\
\cdot & \cdot & 1 & \cdot & \cdot & \cdot & \cdot & \cdot & \cdot & \cdot & \cdot & \cdot
\end{smallmatrix}\right)
    \]
\end{example}

Let $\sigma = (1\,2\,\cdots\,N)$ be the long cycle and let $w_0 = (1, N)(2, N-1)\cdots$ be the longest element in the symmetric group $\mathfrak{S}_N$. The following collects the key results we need on promotion permutations.
\begin{theorem}[{\cite[Thm.~6.7]{fluctuating-paper}}]\label{thm:prom_perms}
Let $T\in\ft(r,\underline{c})$, where $|c_1| + \cdots + |c_n| = N$. Then for all $1 \leq i \leq r-1$:
 \begin{enumerate}[(i)]
   \item $\prom_i(T)$ is a permutation in $\mathfrak{S}_N$,
   \item $\prom_i(T) = \prom_{r-i}(T)^{-1}$,
   \item $\prom_i(\promotion(T)) = \sigma^{-|c_1|} \prom_i(T) \sigma^{|c_1|}$,
   \item $\prom_i(\evacuation(T)) = w_0 \prom_{r-i}(T) w_0$.
 \end{enumerate}
\end{theorem}

\begin{remark}
    There is a typographical error in the statement of \cite[Thm.~6.7(iv)]{fluctuating-paper}, which is corrected in the statement above: part (iv) should have $w_0\prom_{r-i}(T)w_0$ on the right-hand side, rather than $w_0\prom_{i}(T)w_0$.
\end{remark}

We will need the following additional property of promotion permutations. It follows easily from the sliding definition. It may also be seen crystal-theoretically as in \cite[Prop.~8.22]{fluctuating-paper}.

\begin{proposition}\label{prop:prom-cyclically-monotonic}
    Let $T \in \rft(\underline{o})$. Then
    \begin{align*}
      0 < (\prom_1(T)(i) - i)\pmod{n} \leq \cdots \leq (\prom_{r-1}(T)(i) - i)\pmod{n} \qquad&\text{if }o_i = 1, \\
      0 < (i - \prom_1(T)(i))\pmod{n} \leq \cdots \leq (i - \prom_{r-1}(T)(i))\pmod{n} \qquad&\text{if }o_i = -1,
    \end{align*}
    where $a \pmod{n} \in [0, n-1]$.
\end{proposition}

\subsection{The separation labeling}\label{sec:sep_labeling}
To connect to fluctuating tableaux, we now define an intrinsic labeling of contracted fully reduced hourglass plabic graphs.

Recall from \Cref{def:proper-labeling} the notion of a (proper) labeling, which we now apply to hourglass plabic graphs rather than webs. A labeling assigns to each $m$-hourglass an $m$-subset of $\{1,2,3,4\}$.

\begin{definition}
\label{def:separation-labeling}
Let $G \in \crg(\underline{o})$. The \emph{base face} of $G$ is the face $F_0$ incident to the boundary segment between $b_n$ and $b_1$. We define the \emph{separation labeling} $\sep$ of $G$ as follows:
\begin{itemize}
    \item For $e$ a simple edge of $G$, let $F(e)$ be the face incident to $e$ which is on the right when traversing $e$ from the black vertex to the white vertex, and let $\ell_1,\ell_2,\ell_3$ be, respectively, the $\trip_1$-,$\trip_2$-, and $\trip_3$-strands traversing $e$ in this direction. If exactly $a$ of $\ell_1,\ell_2,\ell_3$ separate $F(e)$ from $F_0$, then set $\sep(e)=\{a+1\}$. We often omit braces when writing singleton sets. 
    \item For an  hourglass $e$, let $v$ be either endpoint, and define
\begin{equation}\label{eq:sep-label-for-hourglass}
\sep(e) = \{1,2,3,4\} \setminus \bigcup_{\substack{e' \ni v \\ e' \neq e}} \sep(e').
\end{equation}
\end{itemize}
See \Cref{fig:sep-label-equivariant-proof2} for an example.

If instead $G \in \crg(\underline{c})$, we define $\sep$ of $G$ by first computing $\sep$ of $\osc(G)$ (see \Cref{def:oscillating-version-of-G}), then labeling each simple edge and each internal hourglass of $G$ by the label on the corresponding edge of $\osc(G)$, and finally labeling each boundary hourglass by the set of labels appearing on the corresponding claw of $\osc(G)$.
\end{definition}

\begin{figure}[htbp]
    \centering
    \includegraphics[scale=1]{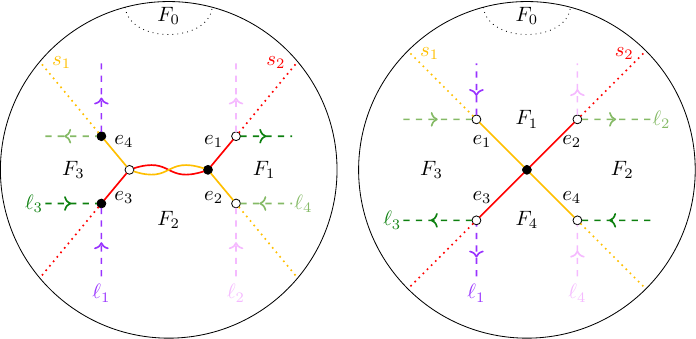}
    \caption{The $\trip_1/\trip_3$-strands $\ell_i$ and $\trip_2$-strands $s_i$ passing through the simple edges surrounding an hourglass (left) or a simple vertex (right). See the proof of \Cref{prop:sep-label-well-defined-and-proper}.}
    \label{fig:sep-label-well-defined}
\end{figure}

\begin{proposition}
\label{prop:sep-label-well-defined-and-proper}
Let $G \in \crg(\underline{c})$. 
Then the separation labeling is well defined and proper.
\end{proposition}
\begin{proof}
The separation labeling is clearly well defined on simple edges, so it suffices to show that \eqref{eq:sep-label-for-hourglass} does not depend on the choice of $v$ in the second bullet point when $e$ has two internal endpoints.

Consider the four simple edges $e_1,e_2,e_3,e_4$ adjacent to the two endpoints of an hourglass. See the left side of \Cref{fig:sep-label-well-defined} for notation we will use; in particular, $s_1$ and $s_2$ are $\trip_2$-strands and the $\ell_i$ are $\trip_1$- or $\trip_3$-strands. To compute $\sep(e_1)$, we must determine which of $s_2,\ell_2,\ell_3$ separate the face $F_1$ from the base face $F_0$. Since $G$ is fully reduced, and hence monotonic by \Cref{cor:fully-reduced-iff-monotonic}, the $\trip_2$-strands $s_1,s_2$ do not recross, nor do any $\trip_1$-strands retouch either $\trip_2$-strand. Thus we see that $s_2,\ell_3$ separate $F_1$ from $F_0$, and that $\ell_2$ separates the two faces if and only if it reaches the boundary to the right of $F_0$ in \Cref{fig:sep-label-well-defined}. Similarly, computing $\sep(e_2)$, we see that $s_1,\ell_4$ separate $F_2$ from $F_0$, and that $\ell_2$ does so if and only if it ends to the left of $F_0$. Thus we conclude that $\{\sep(e_1),\sep(e_2)\}=\{3,4\}$. A similar analysis shows $\{\sep(e_3),\sep(e_4)\}=\{3,4\}$, so that the hourglass may be consistently labeled by $\{1,2\}$; moreover, the resulting labeling is proper around these vertices. The cases differing from \Cref{fig:sep-label-well-defined} by a color reversal or the relative location of $F_0$ are analogous.

It remains to verify that $\sep$ gives a proper labeling around each simple vertex; see the right side of \Cref{fig:sep-label-well-defined}. It is easy to see that $F_2$ and $F_0$ are separated by $s_2,\ell_2$ and that $F_4$ and $F_0$ are separated by $s_1,\ell_3$. Furthermore, exactly one of $F_2,F_4$ is separated from $F_0$ by $\ell_4$. Similarly, exactly one of $F_1,F_3$ is separated from $F_0$ by $\ell_1$. Thus $\{\sep(e_1),\ldots,\sep(e_4)\}=\{1,2,3,4\}$ and $\sep$ is a proper labeling.
\end{proof}

 \subsection{Separation words, promotion, and evacuation} We now show that the separation labeling of a contracted fully reduced hourglass plabic graph gives rise to a rectangular fluctuating tableau and that promotion and evacuation of the tableau correspond to rotation and reflection of the graph.

\begin{definition}
\label{def:boundary-word}
Given $G \in \crg(\underline{c})$, with boundary vertices $b_1,\ldots,b_n$, incident to edges $e_1,\ldots,e_n$, respectively, the \emph{separation word} $\bw(G)=w_1 \ldots w_n$ is the word of type $\underline{c}$ given by setting $w_i=\col(b_i)\sep(e_i)$ for $i=1,\ldots,n$. That is, $\bw(G)=\partial(\sep)$.
\end{definition}

For $\pi \in \mathfrak{S}_n$, let $\aexc(\pi)\coloneqq \{i : \pi^{-1}(i)>i\}$ denote the set of \emph{antiexcedances} of $\pi$. 

\begin{proposition}
\label{prop:sep-boundary-label-from-trips}
Let $G \in \crg(\underline{o})$. Then:
\[
\sep(e_i) = \begin{cases} 1+\left| \{ a : i \not \in \aexc(\trip_a(G)) \} \right|, &\text{if $b_i$ is black;} \\ 1+\left| \{ a : i \in \aexc(\trip_a(G)) \} \right|, &\text{if $b_i$ is white.} \end{cases}
\]
\end{proposition}
\begin{proof}
If $b_i$ is black, then the $\trip_{4-a}(G)(i)$-strand separates $F(e_i)$ and $F_0$ when $\trip_{4-a}(G)(i)<i$. This is equivalent to $i \not \in \aexc(\trip_{a}(G))$. The other case is similar.
\end{proof}

\begin{lemma}
\label{lem:trip1-2-3-layout}
Let $G \in \crg(\underline{o})$. If $b_i$ is black (resp. white), then $b_i, \trip_1(b_i), \trip_2(b_i),$ and $\trip_3(b_i)$ appear in clockwise (resp.\ counterclockwise) order. Some of these vertices may coincide.
\end{lemma}
\begin{proof}
This claim follows by using \Cref{lem:how-trip1-trip2-diverge} to track the colors of the vertices at which the other strands diverge from the $\trip_2(b_i)$-strand.
\end{proof}

\begin{definition}
\label{def:rot-and-refl}
Given an hourglass plabic graph $G$, let $\rot(G)$ be the graph obtained by rotating $G$ one step counterclockwise with respect to the labeling $b_1,\ldots,b_n$ of the boundary, so that what was $b_1$ becomes $b_n$. Let $\refl(G)$ denote the graph obtained by reflecting $G$ with respect to the diameter of the boundary circle passing between $b_n$ and $b_1$.
\end{definition}

Recall that we often identify a fluctuating tableau $T$ with its lattice word $L(T)$. In particular, we write $\promotion(L(T))$ for $L(\promotion(T))$, and likewise for $\evacuation$.

\begin{theorem}
\label{thm:sep-label-promotion-equivariant}
Let $G \in \crg(\underline{c})$. Then:
\begin{enumerate}[(i)]
    \item \label{item:sep-is-lattice}$\bw(G)$ is the lattice word of a rectangular fluctuating tableau $\mathcal{T}(G) \in \rft(\underline{c})$,  
    \item \label{item:prom-is-rot} $\promotion(\bw(G))=\bw(\rot(G))$, and 
    \item \label{item:evac-is-refl}$\evacuation(\bw(G))=\bw(\refl(G))$.
\end{enumerate}
\end{theorem}
\begin{proof}
We may assume without loss of generality that $\underline{c}$ is oscillating, since $\bw(\osc(G))$ being a lattice word implies the same for $\bw(G)$ and since $\promotion^{|c_1|}(\osc(T))=\osc(\promotion(T))$ for $T \in \ft(\underline{c})$ by \Cref{lem:prom-of-osc} while we clearly have $\rot^{|c_1|}(\osc(G))=\osc(\rot(G))$. By \Cref{prop:sep-label-well-defined-and-proper}, $\sep$ is a proper labeling, so $\bw(G)$ is balanced. 

The claim (\ref{item:evac-is-refl}) is straightforward: evacuation \cite[Thm.~6.5]{fluctuating-paper} and reflection both have the effect of applying the involution $\varepsilon$ to $\bw(G)$. It remains to show that $\bw(G)=w_1\ldots w_n$ is a lattice word satisfying (\ref{item:prom-is-rot}). For $a=1,2,3$, let $\ell_a$ denote the $\trip_a$-strand starting at $b_1$, and let $b_{i_a}$ be its endpoint. If there are no internal vertices on $\ell_2$, then (\ref{item:sep-is-lattice}) reduces to a check of the claim for each connected component of $G$ and (\ref{item:prom-is-rot}) is clear. Otherwise, we proceed by induction on the number of internal vertices in $G$. Suppose $\col(b_1)=1$; otherwise we may reverse all colors, which has the effect of applying the involution $\varpi$ to $\bw(G)$, an operation which preserves the lattice word property and commutes (by \cite[Lem.~3.13]{two-column}) with promotion. 

Since $\trip_a$-strands rotate along with the graph, it is clear by \Cref{prop:sep-boundary-label-from-trips} that $\bw(\rot(G))_{j-1} = \bw(G)_j$ (with indices taken modulo $n$) unless $j \in \{1 \eqqcolon i_0,i_1,i_2,i_3\}$. Thus, by the first balance point characterization of promotion in \Cref{prop:promotion-from-balance-points}, for (\ref{item:prom-is-rot}) we need to show: 
\begin{equation*}
\text{$i_a$ is the smallest index weakly greater than $i_{a-1}$ such that $\slack_a(w_1 \ldots w_{i_a})=0$.} \tag{$\dagger$}
\end{equation*}

We now claim that we can reduce to the case where $\ell_1$ and $\ell_2$ share only their first edge and $i_1 \neq i_2$. Suppose that $\ell_1$ and $\ell_2$ share also their second edge. Then by \Cref{lem:how-trip1-trip2-diverge}, they also share their third edge, and these three edges are exactly as shown in \Cref{fig:sep-label-equivariant-proof1}. We consider a smaller hourglass plabic graph $C$ whose boundary contour $\bs{c}$ excises the first two edges and intersects the third, as shown in the figure. The boundary vertices of $C$ lie on the cut edges of its contour and their colors are determined by bipartiteness. These boundary vertices are numbered so that the marked starting point of the contour is taken as the base face of the graph. As $C$ has fewer internal vertices than $G$, it rotations and reflections satisfy (\ref{item:sep-is-lattice}) and (\ref{item:prom-is-rot}) by induction.

We wish to compare $\bw(C)$ with $\bw(G)$. All boundary edges $e$ of $C$ or $G$ have $\sep_C(e)=\sep_G(e)$ except for the endpoints of the four $\trip$-strands passing through the second or last boundary edges of $C$, as these are the only strands which end on opposite sides of the base faces of the two graphs. If these endpoints are distinct, with colors as in \Cref{fig:sep-label-equivariant-proof1}, then \Cref{lem:trip1-2-3-layout} and monotonicity imply that these separation labels are as indicated in the figure. (Note that if some endpoints coincide, then their colors are necessarily as indicated.) Thus we have
\begin{align*}
    \bw(C) &= 11 \circ a_1 \circ \bar{1} \circ a_2 \circ \bar{2} \circ a_3 \circ \bar{3} \circ a_4 \circ 3 \circ a_5 \circ 2 \circ a_6 \circ 1 \circ a_7 \circ \bar{1}, \\
    \bw(G) &= 1 \circ a_1 \circ \bar{2} \circ a_2 \circ \bar{3} \circ a_3 \circ \bar{4} \circ a_4 \circ 4 \circ a_5 \circ 3 \circ a_6 \circ 2 \circ a_7,
\end{align*}
where $\circ$ denotes concatenation, and where the $a_k$ are the intervening unchanged subwords. Since $C$ satisfies (\ref{item:sep-is-lattice}) and $\rot(C)$ satisfies (\ref{item:prom-is-rot}), there is no $1$-balance point in $1 \circ a_1$. Thus we may swap the $a_1$ and $\bar{1}$ in $\bw(C)$ and still have a lattice word $1 1 \bar{1} \circ a_1 \circ a_2 \circ \cdots$. We can always delete a consecutive $1 \bar{1}$ and remain a lattice word; similarly, we may freely insert $\bar{4} 4$ anywhere, and choose to do so following $a_4$. 
We can likewise commute $a_2$ with the following $\bar{2}$, since there is no $2$-balance point in $1 \circ a_1 \circ a_2$ inside the subword $a_2$. Continuing in this way from left to right in the word, and also from right to left at the other end of the word (using the inductive hypotheses for $\refl(C)$), we produce $\bw(G)$. Thus $\bw(G)$ is a lattice word, giving (\ref{item:sep-is-lattice}) for $G$. These operations have not changed the balance points relevant to ($\dagger$), with the exception of the first $3$-balance point after $b_{i_2}$, which moves as desired to $b_{i_3}=\trip_3(G)(b_1)$. If instead the endpoints are not distinct, some steps in the transformation of $\bw(C)$ into $\bw(G)$ can be omitted. If some endpoints have colors differing from \Cref{fig:sep-label-equivariant-proof1}, the analysis is similar and omitted. In any case, we have (\ref{item:prom-is-rot}) for $G$. 

If, after having performed this reduction, $\ell_1$ and $\ell_2$ still share more than their first edge, we can repeat until either they diverge at the first vertex or we reach a disconnected graph, which can be further reduced. This proves the claim.   

\begin{figure}[ht]
    \centering
    \includegraphics[scale=1.25]{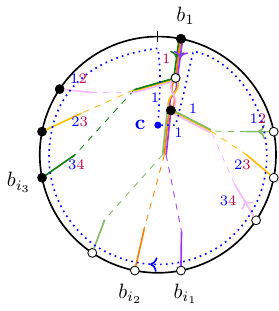}
    \includegraphics[scale=1.25]{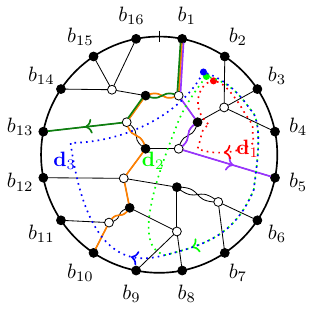}
    \caption{Left: The contour $\bs{c}$, its separation labels (in blue), and the separation labels for $G$ (in red). The two green trips may cross, as may the pink and purple trips, but this does not materially affect the argument.  Right: an example of the contours $\bs{d_1}, \bs{d_2},$ and $\bs{d_3}$.}
    \label{fig:sep-label-equivariant-proof1}
\end{figure}

\medskip

Having made the above reductions, we now prove $(\dagger)$ for $G$. As shown in \Cref{fig:sep-label-equivariant-proof1}, consider hourglass plabic graphs $D_a$ for $a=1,2,3$ cut out by contours $\bs{d}_a$ starting in the boundary face between $b_1$ and $b_2$ and encircling all internal vertices of $G$ that lie to the left of $\ell_a$. Since $\ell_3$ contains an internal vertex, the $D_a$ have strictly fewer internal vertices than $G$, so we may assume by induction they satisfy (\ref{item:sep-is-lattice}) and (\ref{item:prom-is-rot}).

\medskip
\noindent
{\sf Case 1 ($a=1$):} 
Let $\bw(D_1)=v_1\ldots v_m$; we wish to compare $w_2 \ldots w_{i_1-1}$ with $v_1 \ldots v_{i_1-2}$. 

First, we claim that for $j>i_1-2$ we have $v_j \in \{3,4\}$. These are obtained from the separation labeling of the boundary edges $e'$ of $D_1$ that are incident to $\ell_1$. The corresponding boundary vertices $b'$ of $D_1$ are black, as otherwise $\ell_1$ would have turned left to follow $e'$. Furthermore, since $G$ (and so $D_1$) is fully reduced, the $\trip_2$-strand passing through $e'$ must end at one of the genuine boundary vertices $b_2,\ldots,b_{i_1-1}$, which have a smaller index in $D_1$. Thus, by \Cref{prop:sep-boundary-label-from-trips} and \Cref{lem:trip1-2-3-layout}, we have $\sep_{D_1}(e') \geq 3$. This implies that, since $\bw(D_1)$ is balanced, $\slack_1(v_1\ldots v_{i_1-2})=0$.

Now, let $b_j$ be among $b_2,\ldots,b_{i_1-1}$ with incident edge $e_j$. Since $G$ is fully reduced and therefore monotonic by \Cref{cor:fully-reduced-iff-monotonic}, the $\trip_2$-strand through $b_j$ may not exit and reenter the subgraph $D_1$. This implies $\sep_{D_1}(e_j) \in \{1,2\}$ if and only if $\sep_G(e_j) \in \{1,2\}$. 

Moreover, we claim that if $\sep_{G}(e_j) =2$, then $\sep_{D_1}(e_j) = 2$. To see this, first assume that $b_j$ is black and consider the $\trip_3$-strand through $b_j$. If we had $\sep_{D_1}(e_j) = 1$, then this $\trip_3$-strand must cross the $\ell_1$. But then $\sep_{G}(e_j) =2$ implies that it must cross again, forming a bad double crossing in the underlying plabic graph $\widehat{G}$, contradicting that $\widehat{G}$ is reduced by \Cref{prop:underlying-plabic-is-reduced}.  If instead $b_j$ is a white vertex, then every strand through $b_j$ that separates in $G$ also separates in $D_1$. This establishes the claim.
An analogous argument, reversing the role of black and white vertices, shows that if $\sep_{G}(e_j) =1$, then $\sep_{D_1}(e_j) = 1$.

Together with the fact that $\slack_1(v_1\ldots v_{i_1-2})=0$, these results imply that $\slack_1(w_2 \ldots w_{i_1-1})=\slack_1(w_1 \ldots w_{i_1})=0$. If $\bw(G)$ had balanced $1$'s
and $2$'s for some shorter prefix, this would contradict the inductive hypothesis that $\bw(D_1)$ is a lattice word.

\medskip
\noindent
{\sf Case 2 ($a=2$):}
Let $\bw(D_2)=v_1\ldots v_m$; we wish to compare $w_2 \ldots w_{i_2-1}$ with $v_1 \ldots v_{i_2-2}$.

First, we claim that for $j>i_2-2$, we have $v_j \in \{\bar{1},4\}$. These values are obtained from the separation labeling of the boundary edges $e'$ of $D_2$ that are incident to $\ell_2$. Since $G$ is fully reduced and therefore monotonic by \Cref{cor:fully-reduced-iff-monotonic}, no $\trip_{\alpha}$-strand can double cross $\ell_2$ for any $\alpha$. Hence, every $\trip_{\alpha}$-strand passing through $e'$ must end at one of the genuine boundary vertices $b_2,\ldots,b_{i_2-1}$, which have a smaller index in $D_2$. Thus, by \Cref{prop:sep-boundary-label-from-trips} and \Cref{lem:trip1-2-3-layout}, we have $\sep_{D_2}(e') = 1$ if $b'$ is white and $\sep_{D_2}(e') = 4$ if $b'$ is black. This implies that, since $\bw(D_2)$ is balanced, we have $\slack_2(v_1\ldots v_{i_2-2})=0$. Moreover, for $1 \leq k \leq i_2-2$, we have $v_k  = w_{k+1}$, except for $k=i_1-2$ when $v_{i_1-2} = 1$ (or $\bar{4}$) and $w_{i_1-1} = 2$ (or $\bar{3}$).
Thus, $\slack_2(w_2 \ldots w_{i_2-1})=1$ and therefore $\slack_2(w_1 \ldots w_{i_2})=0$. If we had $\slack_2(w_1 \dots w_k)=0$ for some $i_1 \leq k < i_2$, this would contradict the inductive hypothesis that $\bw(D_2)$ is a lattice word.

\medskip
\noindent
{\sf Case 3 ($a=3$):}
Let $\bw(D_3)=v_1\ldots v_m$; we wish to compare $w_2 \ldots w_{i_3-1}$ with $v_1 \ldots v_{i_3-2}$. 

First, we claim that for $j>i_3-2$, we have $v_j \in \{\overline{1},\overline{2},\{\overline{1},\overline{2}\}\}$. These are obtained from the separation labeling of the boundary edges $e'$ of $D_3$ that are incident to $\ell_3$. The corresponding boundary vertices $b'$ of $D_3$ are white, as otherwise $\ell_3$ would have turned right to follow $e'$. If $e'$ is an hourglass, we find that the adjacent simple edges along $\ell_3$ have separation labels $4, 3$, so by \Cref{prop:sep-label-well-defined-and-proper}, $\sep_{D_3}(e') = \{1, 2\}$. We now assume $e'$ is a simple edge. Since $G$ (and so $D_3$) is fully reduced and therefore monotonic by \Cref{cor:fully-reduced-iff-monotonic}, the $\trip_2$-strand passing through $e'$ must end at one of the genuine boundary vertices $b_2,\ldots,b_{i_3-1}$, which have a smaller index in $D_3$. Thus, by \Cref{prop:sep-boundary-label-from-trips} and \Cref{lem:trip1-2-3-layout}, we have $\sep_{D_3}(e') \leq 2$. This implies that, since $\bw(D_3)$ is balanced, we must have $\slack_3(v_1\ldots v_{i_3-2}) = 0$.

Let $b_j$ be among $b_2,\ldots,b_{i_3-1}$ with incident edge $e_j$. Since $G$ is monotonic, the $\trip_2$-strand through $b_j$ may not exit and reenter the subgraph $D_3$. This implies that $\sep_{D_3}(e_j) \in \{3,4\}$ if and only if $\sep_G(e_j) \in \{3,4\}$, except that $\sep_G(e_{i_2})=3$ while $\sep_{D_3}(e_{i_2})=2$. Moreover, we claim for $j \neq i_2$ that if $\sep_{G}(e_j) = 3$, then $\sep_{D_3}(e_j) = 3$. 

Suppose first that $b_j$ is black. The $\trip_k$-strand through $b_j$ is separating in $D_3$ if and only if the strand ends at a smaller-indexed vertex of $D_3$. If the $\trip_k$-strand through $b_j$ lands at a smaller-indexed vertex of $D_3$, then certainly it lands at a smaller-indexed vertex in $G$. Hence, we have $\sep_{D_3}(e_j) \leq \sep_G(e_j)$. In particular, for $j \neq i_2$, if $\sep_G(e_j)=3$, then $\sep_{D_3}(e_j) = 3$. Now, suppose instead that $b_j$ is white. Suppose that $\sep_{G}(e_j) = 3$. Then the unique strand through $b_j$ that fails to separate in $G$ is the $\trip_1$-strand. If instead we had $\sep_{D_3}(e_j) = 4$, then the $\trip_1$-strand through $b_j$ must be separating in $D_3$. Thus, in $G$, the $\trip_1$-strand through $b_j$ must cross $\ell_3$ to produce a separation in $D_3$, but then cross back to land at a smaller-indexed vertex in $G$; this yields a bad double crossing in the underlying plabic graph $\widehat{G}$, contradicting the assumption that $\widehat{G}$ is reduced. This completes the proof of the claim.

We claim also that if $\sep_{G}(e_j) = 4$, then $\sep_{D_3}(e_j) = 4$. First suppose that $b_j$ is white. Any separating trip strand from $b_j$ which is separating in $G$ is also separating in $D_3$, hence $\sep_{G}(e_j) = 4$ implies $\sep_{D_3}(e_j) = 4$ in this case.
Now suppose that $b_j$ is black and $\sep_{G}(e_j) = 4$. If we had $\sep_{D_3}(e_j) = 3$, then it must be that the $\trip_1$-strand through $b_j$ crosses $\ell_3$, so as not to be separating for $D_3$, but then crosses $\ell_3$ again to land at a smaller-indexed vertex in $G$; this yields a bad double crossing in $\widehat{G}$, which is a contradiction. This proves the second claim.

Together with the fact that $\slack_3(v_1\ldots v_{i_3-2}) = 0$, these results imply that $\slack_3(w_2 \ldots w_{i_3-1}) = 1$, and therefore also $\slack_3(w_1 \ldots w_{i_3}) = 0$. If $\bw(G)$ had some shorter prefix $w_1 \dots w_k$ with $i_2 \leq k < i_3$ and $\slack_3(w_1 \dots w_k) = 0$, this would contradict the inductive hypothesis that $\bw(D_3)$ is lattice. 

\medskip

This completes the proof of ($\dagger$) for $G$. All that remains is to prove that $\bw(G)$ is lattice. Consider hourglass plabic graphs $C_a$ for $a=1,2,3$ cut out by contours $\bs{c}_a$, drawn as follows (see \Cref{fig:sep-label-equivariant-proof2}):
\begin{enumerate}
    \item $\bs{c}_1$ starts in the boundary face between $b_1$ and $b_2$ and includes all internal vertices to the left of $\ell_2$ (with respect to its orientation);
    \item $\bs{c}_2$ starts in the base face and includes the internal vertices of $G$ which lie on $\ell_2$;
    \item $\bs{c}_3$ starts in the base face and includes the internal vertices of $G$ which lie to the right of $\ell_2$.
\end{enumerate} 
Since $\ell_2$ contains an internal vertex, the graphs $C_a$ have strictly fewer internal vertices than $G$, with the possible exception of $C_2$. The exceptional case, where all internal vertices lie on $\ell_2$, is treated separately in \Cref{lem:backbone-is-lattice-word}. Thus we may assume by induction that each $C_a$ satisfies (\ref{item:sep-is-lattice}) and (\ref{item:prom-is-rot}).

Write $\bw'(C_i)$ for the boundary word obtained by using the separation edge labels of $G$, rather than the intrinsic separation labels used to compute $\bw(C_i)$. Since $C_2$ and $C_3$ have the same base face as $G$, it follows as in previous cases that $\bw'(C_2)=\bw(C_2)$ and $\bw'(C_3)=\bw(C_3)$, the key point being that $\trip_a$-strands do not reenter any of the contours after having left. Write $\bw'(C_1)=u'_1\ldots u'_m$ and $\bw(C_1)=u_1\ldots u_m$. In the same way, we see that $u'_j=u_j$ except when $j$ is $i_1-1$ or $m$. By \Cref{prop:sep-boundary-label-from-trips} and \Cref{lem:trip1-2-3-layout}, we have $u'_{i_1-1} \in \{2,\bar{1}\}$, $u_{i_1-1}\in \{1,\bar{2}\}$, $u'_m=\bar{2}$, and $u_m=\bar{1}$. 

\begin{figure}[ht]
    \centering
    \includegraphics[scale=1.25]{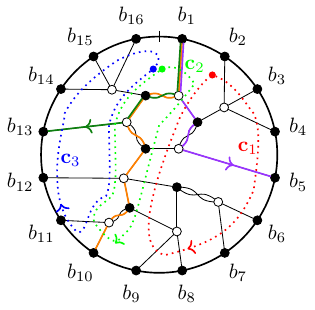}
    \includegraphics[scale=1.25]{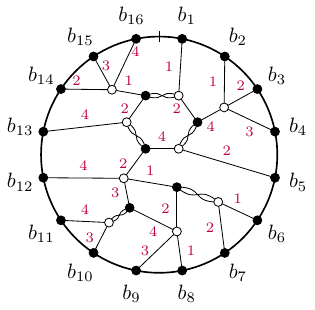}
    \caption{Left: the contours $\bs{c}_a$ discussed in the proof of \Cref{thm:sep-label-promotion-equivariant}. Right: the separation labeling of the same graph, where the labels on $2$-hourglass edges are omitted for visual clarity.}
    \label{fig:sep-label-equivariant-proof2}
\end{figure}

We claim that $1 \circ \bw'(C_1) \circ \bar{1}$ is a lattice word. By Case 1, we have that $\slack_1(1u'_1 \ldots u'_{i_1-1})=\slack_1(w_1\ldots w_{i_1})=0$ and thus $\slack_1(u_1\ldots u_{i_1-2})=0$ as well. Since $\bw(C_1)$ is a lattice word by induction, so is $\tau(\bw(C_1))=\bar{u}_m \ldots \bar{u}_1$. Furthermore, it is clear from the definitions that this is the separation word of the fully reduced graph $\tau(C_1)$ obtained by reversing the colors of $\refl(C_1)$ (note that $\tau(C_1)$ still has a black first boundary vertex). It is also immediate that $\trip_1(\tau(C_1))(1)=m-i_1+2$ (using the indexing intrinsic to $\tau(C_1)$). By (\ref{item:prom-is-rot}) applied inductively to $\tau(C_1)$, we therefore have that $\bar{u}_m \ldots \bar{u}_{i_1-1}$ has balanced $1$'s and $2$'s and is the shortest prefix of $\bw(\tau(C_1))$ with this property. Thus for $j>i_1-1$ we have $\slack_1(\bar{u}_m \ldots \bar{u}_j)>0$; since $\bw(\tau(C_1))$ is balanced, we likewise have $\slack_1(u_1 \ldots u_{j-1})>0$. This implies that $\slack_1(1u'_1\ldots u'_j)=\slack_1(u_1 \ldots u_{j-1})-1 \geq 0$, so $1 \circ \bw'(C_1)$ is a lattice word. By balancedness of $\bw'(C_1)$,  $1 \circ \bw'(C_1) \circ \bar{1}$ is a lattice word as well, proving the claim.

Consider the word $z=1 \circ \bw'(C_1) \circ \bar{1} \circ \bw(C_2) \circ \bw(C_3)$. In the example of \Cref{fig:sep-label-equivariant-proof2}, we have
\[z=1(123212134\bar{1}4\bar{2})\bar{1}(12\bar{4}1\bar{4}3444\bar{1})(1\bar{4}\bar{4}\bar{4}444234).\] Since concatenation preserves lattice words, $z$ is also lattice. By construction, we may re-parenthesize $z$ as 
\[z=112321213(4\bar{1}4\bar{2}\bar{1}12\bar{4}1\bar{4})3(444\bar{1}1\bar{4}\bar{4}\bar{4})444234,\] grouping together the reversed and sign-reversed substrings coming from the overlaps (with reversed orientations) of $\bs{c}_1$ and $\bs{c}_2$ and of $\bs{c}_2$ and $\bs{c}_3$. These grouped substrings are clearly balanced. We may remove any consecutive balanced substrings from a lattice word and obtain another lattice word, in this case $z'=1123212133444234$. By construction, we have $z'=\bw(G)$, so $\bw(G)$ is a lattice word.
\end{proof}

\begin{lemma}
\label{lem:backbone-is-lattice-word}
Let $G \in \crg(\underline{o})$. In the notation of the proof of \Cref{thm:sep-label-promotion-equivariant}, suppose that all internal vertices of $G$ lie on $\ell_2$ and that $\ell_1$ and $\ell_2$ share only a single edge. Then $G$ satisfies the conclusions of \Cref{thm:sep-label-promotion-equivariant} (\ref{item:sep-is-lattice}) and (\ref{item:prom-is-rot}).
\end{lemma}
\begin{proof}
Suppose that $\col(b_1)=1$; otherwise apply $\varpi$. Let $\bw(G)=w_1\ldots w_n$. By \Cref{lem:how-trip1-trip2-diverge}, \Cref{prop:sep-boundary-label-from-trips}, and \Cref{lem:trip1-2-3-layout} we have $w_{i_1}=2$, $w_{i_2} \in \{3,\bar{2}\}$ and $w_{i_3}=4$. Moreover, it is clear from the hypotheses that $i_1=2$. The remainder of the separation word is also easy to compute. For $2<j<i_2$, we have $w_j \in \{1, \bar{4}\}$ and for $i_2<j<n$ we have $w_j \in \{\bar{1},4\}$.

If $i_2 \neq i_3$, then $\bw(G)=12 \circ u \circ w_{i_2} \circ u'$, where $u$ consists of $1$'s and $\bar{4}$'s and $u'$ consists of $\bar{1}$'s and $4$'s. Any balanced word of this form is a lattice word, giving (\ref{item:sep-is-lattice}). Furthermore, the balance point conditions are also clearly satisfied (noting that $b_{i_3}$ must be the last black boundary vertex, so $w_{i_3}$ is the last $4$ in $\bw(G)$); this gives (\ref{item:prom-is-rot}).

If instead $i_2=i_3$, then $b_{i_2}$ is white by \Cref{lem:how-trip1-trip2-diverge}. Furthermore, there are no white vertices among $b_j$ for $2 < j < i_2$ nor black vertices among $b_j$ for $j>i_2$. Thus $\bw(G)=121^k\bar{2}\bar{1}^{k+1}$ for $k=\frac{1}{2}(n-4)$. This word is clearly a lattice word, and satisfies the balance point conditions since $\slack_2(w_1\ldots w_{i_2})=\slack_3(w_1\ldots w_{i_2})=0$.
\end{proof}

In \Cref{sec:hourglass-and-six-vertex,sec:backwards-map}, we have developed the theory of fully reduced hourglass plabic graphs in order to define the map $\mathcal{T}:\crg(\underline{c})/{\sim} \to \rft(\underline{c})$ and show that it intertwines rotation with promotion (\Cref{thm:sep-label-promotion-equivariant}). Our web basis $\mathscr{B}_q^{\underline{c}}$ for $\Inv_{U_q(\fsl_4)}(\bigwedge_q\nolimits^{\underline{c}} V_q)$ will consist of a special representative $W$ from each equivalence class in $\crg(\underline{c})/{\sim}$. In \Cref{sec:forwards-map}, we develop \emph{growth rules} in order to show that the invariants $[W]_q$ are linearly independent and to define a map $\mathcal{G}:\rft(\underline{c}) \to \crg(\underline{c})/{\sim}$. In \Cref{sec:bijection-and-basis}, we show that $\mathcal{T}$ and $\mathcal{G}$ are mutually inverse bijections and therefore (by \Cref{prop:fluc-tab-give-dimension}) that $\mathscr{B}_q^{\underline{c}}$ is the desired basis.

\section{Growth rules from crystals}
\label{sec:forwards-map}

In this section, we describe a set of \emph{growth rules} for producing fully reduced hourglass plabic graphs from rectangular fluctuating tableaux. This algorithm is necessary for our main results in \Cref{sec:bijection-and-basis}. \Cref{sec:bypassing_growth_rules} gives alternate proofs of weaker versions of some of these results, bypassing the growth rules developed here; for discussion of what is lost in the weakening, see \Cref{remark:what_we_lose}. The reader who is satisfied with these weaker results may skip the current section on a first reading.

While there are $14$ growth rules in the $r=3$ case in \cite{Khovanov-Kuperberg}, our $r=4$ algorithm involves $88$ \emph{short} growth rules of length $2$ or $3$ (falling into $10$ families), with $2$ of those families extending to \emph{long} growth rules of arbitrary length; see \Cref{fig:growth-rules}. Our proof primarily involves carefully tracking trip and promotion permutations simultaneously by using a novel crystal-theoretic algorithm, which we introduce; see \Cref{thm:good_degen}. Our approach is fundamentally inductive and verifies that many desirable properties are preserved when applying an additional growth rule.

\subsection{The growth algorithm}
We now give an algorithm which takes as input the lattice word of a rectangular fluctuating tableau with $r=4$ and constructs a corresponding fully reduced hourglass plabic graph; see \Cref{thm:growth_algorithm}. The growth algorithm broadly proceeds as follows; see \Cref{fig:growth-example} and \Cref{ex:growth_figure}. We begin with a balanced lattice word $w$ with dangling strands labeled by each corresponding letter. The algorithm proceeds by selecting certain subwords $w_{j+1} \cdots w_{j+p}$ of $w$ and creating either a new oriented $\sfX$ or a $\cup$ (an \emph{end cap}) by, respectively, crossing or joining two adjacent strands; all other strands are extended further down without other modification. This process results in a new dangling lattice word $w'$ obtained from $w$ by substituting a subword $w_{j+1}' \cdots w_{j+q}'$ for $w_{j+1} \cdots w_{j+p}$. This is iterated reading the remaining labels of the dangling strands from left to right until there are none left.

\begin{algorithm}[\textsc{Growth algorithm}]\label{alg:growth}
  Let $T \in \rft(\underline{c})$ and $w=L(T)$ with $r=4$.
  \begin{enumerate}[(G1)]
    \item Make a horizontal line of boundary vertices and downward dangling strands for each letter in $\osc(w)$. Orient strands down for positive letters and up for negative letters.
    \item Apply the growth rules from \Cref{fig:growth-rules} by crossing dangling strands or applying end caps until there are no dangling strands.
    \item Convert the resulting \symm six-vertex configuration to an hourglass plabic graph and de-oscillize boundary vertices by combining into hourglasses.
  \end{enumerate}
  \textsc{Return} the resulting hourglass plabic graph $\mathcal{G}(T)$.
\end{algorithm}

Furthermore, \Cref{alg:growth} produces a proper labeling of its output. More precisely, for $G=\mathcal{G}(T)$, the \emph{growth labeling} $\Gamma(G)$ is the proper labeling of $G$ obtained by applying the growth rules and ignoring bars on the labels. Note that $\partial(\Gamma(G)) = w$.

\begin{figure}[htbp]
    \centering
    \includegraphics[width=0.95\linewidth]{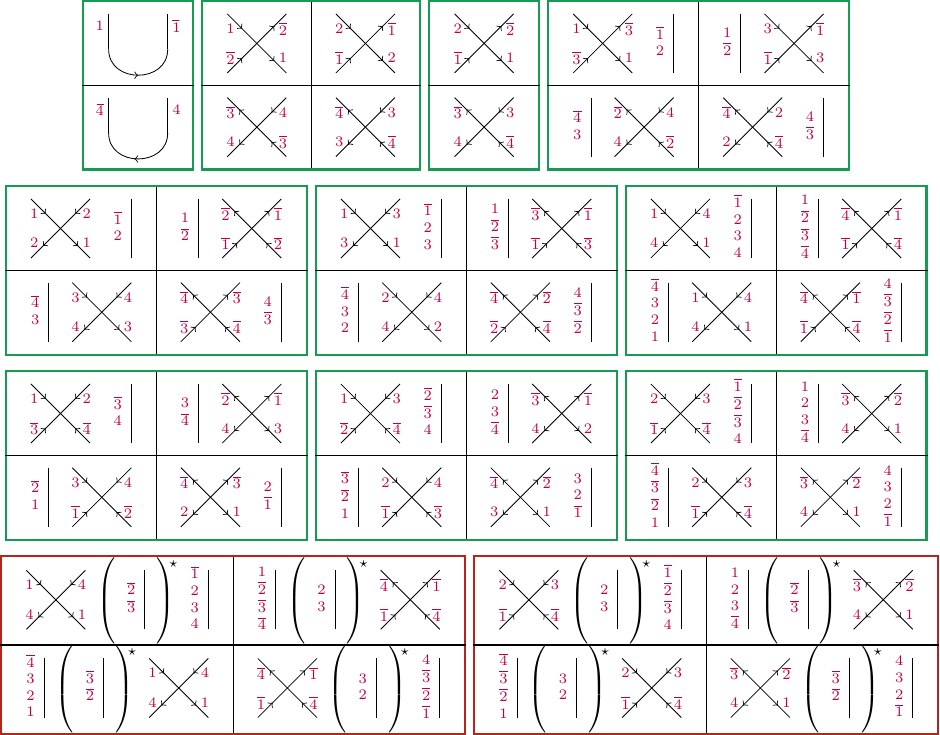}
    \caption{The growth rules of \Cref{alg:growth}. The $88$ short rules of length $2$ or $3$ are boxed in green ($\textcolor{ForestGreen}{\blacksquare}$) and fall into $10$ families. Vertical lines with multiple labels indicate that a \emph{witness} with one of these labels must be present. Two additional families of long rules extend two of the short families and are boxed in red ($\textcolor{BrickRed}{\blacksquare}$). Each of these long rules includes a witness in parentheses with a $\star$ beside it; this means multiple such witnesses (or none) may appear. Note we often refer to these growth rules as ``top labels'' $\rightarrow$ ``bottom labels''; for instance, the top left rule is denoted $1\overline{1}\rightarrow\varnothing$ while the rule to its right is denoted $1\overline{2}\rightarrow\overline{2}1$.
}
    \label{fig:growth-rules}
\end{figure}

\begin{remark}
Before stating the key properties of \Cref{alg:growth} precisely, we note the following.
\begin{itemize}
    \item[---] Step (G2) of \Cref{alg:growth} involves arbitrary choices. We will show in \Cref{thm:growth_algorithm} that any sequence of choices eventually terminates in some fully reduced hourglass plabic graph.
    \item[---] In contrast to the growth rules in \cite{Khovanov-Kuperberg} for the case $r=3$, the output is \textit{not} independent of the choices made. For instance, beginning with $12132434$ and first applying either $132 \to 312$ or $324 \to 342$ results in two hourglass plabic graphs related by a square move.
    \item[---] All choices produce hourglass plabic graphs in the same move-equivalence class, by \Cref{thm:hourglass-trips-determine-move-equivalence} and \Cref{thm:growth_algorithm}(i)--(ii) below.
    \item[---] The rules presented here do not in general produce all elements of all move-equivalence classes. For instance, only one orientation of the basic benzene graph corresponding to $1\overline{4}2\overline{2}4\overline{1}$ is produced.
    \item[---] We will later see in \Cref{thm:main-bijection} that every move-equivalence class of $\crg(\underline{c})$ has an element produced by the growth algorithm.
    \item[---] Furthermore, we will also see that every element $G$ of each move-equivalence class has a unique $\tlex$-minimal (\Cref{def:grevlex}) proper labeling with boundary $w$, which is the growth labeling when $G$ is the output of \Cref{alg:growth}.
    \item[---] Westbury \cite{Westbury} gives a deterministic growth algorithm to produce web bases for all $\mathfrak{gl}_r$. When $r=4$ for $w=11223344$, Westbury's algorithm agrees with ours after translation to hourglass plabic graphs and produces a basic square. For $w=12132434$, the promotion of $11223344$, Westbury's algorithm produces a web with two squares glued along an hourglass edge, which is not fully reduced.
\end{itemize}
\end{remark}

\begin{example}\label{ex:growth_figure}
  Let $w = 11\overline{4}21\overline{3}\{\overline{4},\overline{3},\overline{2},\overline{1}\}3\overline{1}\overline{1}\{\overline{2},\overline{1}\}4$. Begin with dangling strands labeled by $\osc(w)$. First apply growth rule $\overline{3}\overline{2}\overline{1} \to 41\overline{1}$, then apply $\overline{4}4 \to \varnothing$ and $1\overline{1} \to \varnothing$, etc. See \Cref{fig:growth-example} for the resulting \symm six-vertex configuration and the final fully reduced hourglass plabic graph $G$, with its growth labeling. 
\end{example}

\begin{figure}[htbp]
    \centering
    \includegraphics[width=.89\linewidth]{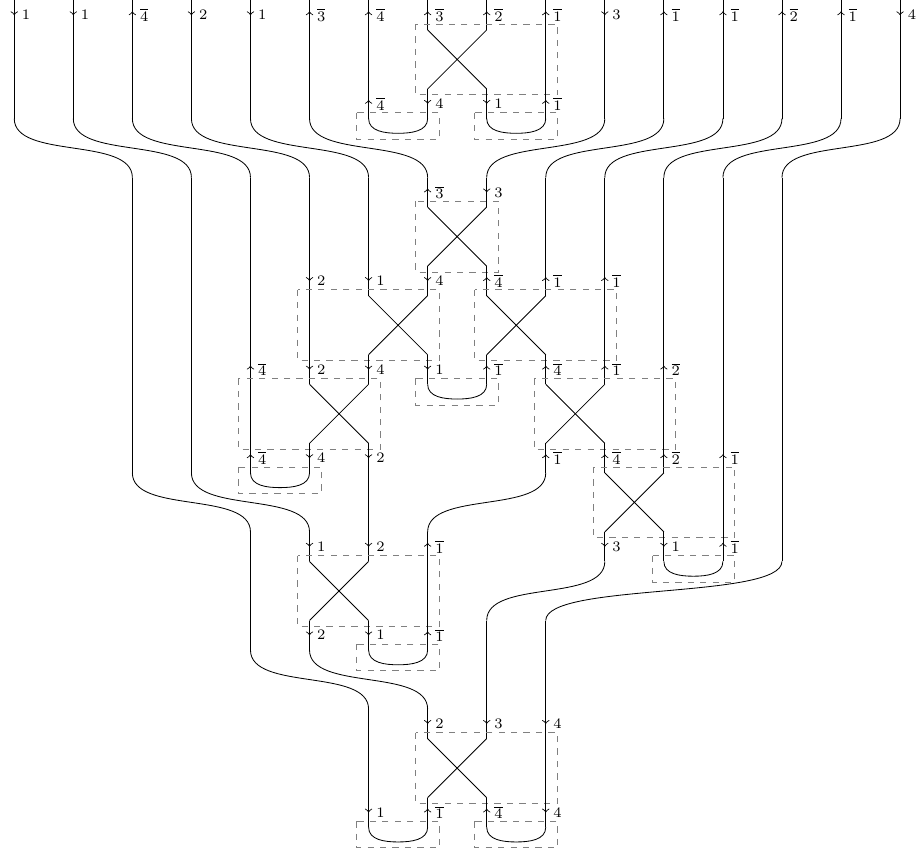}
    
    \includegraphics[width=0.4\textwidth]{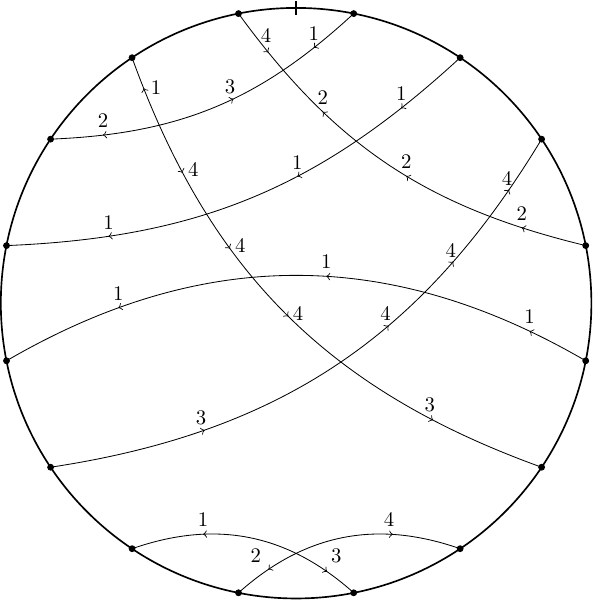} \quad \quad 
    \includegraphics[width=0.4\textwidth]{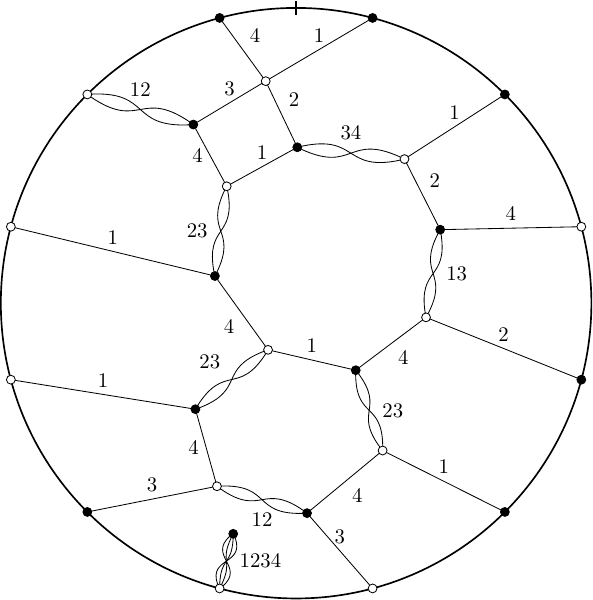}
    \caption{An example of the growth algorithm, \Cref{alg:growth}. The top diagram illustrates the process of applying the growth rules of \Cref{fig:growth-rules}; the lower left diagram is a resulting symmetrized six-vertex configuration; and the lower right diagram is the resulting fully reduced hourglass plabic graph $G$, together with its growth labeling. The $|$ on each circle marks the location of the base face. See \Cref{ex:growth_figure} for discussion.}
    \label{fig:growth-example}
\end{figure}

The following is the main result of this section. 

\begin{theorem}\label{thm:growth_algorithm}
  Applying the growth algorithm (\Cref{alg:growth}) to $T \in \rft(\underline{c})$ always terminates in finitely many steps at an hourglass plabic graph $G = \mathcal{G}(T) \in \crg(\underline{c})$. Moreover,
  \begin{enumerate}[(i)]
    \item $G$ is fully reduced,
    \item $\trip_{\bullet}(G) = \prom_{\bullet}(T)$,
    \item $\Gamma(G)$ is the unique proper labeling of $G$ with boundary $w=L(T)$,
    \item all other proper labelings $\phi$ of $G$ have $\partial(\phi) >_{\tlex} w$, and
    \item $i \in \Des(\osc(w))$ (see \Cref{def:Des}) if and only if boundary vertices $b_i$ and $b_{i+1}$ of $\osc(G)$ are connected to the same vertex.
  \end{enumerate}
\end{theorem}

The proof of \Cref{thm:growth_algorithm} occupies the remainder of \Cref{sec:forwards-map}. The proof may be skipped by a reader willing to accept \Cref{thm:growth_algorithm} as a black box; however, we believe the crystal-theoretic techniques developed here may be of independent interest. We next give a high-level summary of the argument, to help motivate the constructions and elucidate the main ideas. The reader may find it useful to return to this discussion after reading the later subsections. 

The growth rules are related to a modification of Knuth equivalence that we call \emph{$\SL_r$-Knuth equivalence}. Recall that the elementary Knuth equivalences are
\begin{equation}
  txy \sim_K tyx\text{ if }t \in [x, y) \qquad\text{and}\qquad xyt \sim_K yxt\text{ if }t \in (x, y],
\end{equation}
e.g.,~$334 \sim_K 343$ and $122 \sim_K 212$. We call such $t$ a \textit{witness}. Observe that every balanced lattice word coming from an $r \times c$ rectangular standard tableau is Knuth equivalent to $(123\ldots r)^c$.

Now, we define $\SL_r$-Knuth equivalence on words in $\pm [r]$ by additionally imposing $i_1 i_2 \cdots i_k \sim_K \overline{j_1} \cdots \overline{j_{r-k}}$ whenever $\{i_1 < \cdots < i_k\} \sqcup \{j_1 > \cdots > j_{r-k}\} = [r]$. Consequently, we find that $12\ldots r \sim_K 1\overline{1} \sim_K \varnothing \sim_K \overline{r}r$, and indeed every balanced oscillating lattice word is $\SL_r$-Knuth equivalent to $\varnothing$.

The growth rules in \Cref{fig:growth-rules} in fact all come from pairs of $\SL_r$-Knuth equivalent words. $\SL_r$-Knuth equivalence respects much of the structure of promotion permutations. More precisely, if $v \sim_K v'$, then the promotion matrices $\PM(uvw)$ and $\PM(uv'w)$ from \Cref{def:prom-mats} agree after deleting the rows and columns corresponding to $v$ and $v'$.

A guiding principle behind the growth algorithm is to inductively create an hourglass plabic graph with $\trip_\bullet = \prom_\bullet$. This requires controlling the changes in $\PM(uvw)$ and $\PM(uv'w)$ even at rows and columns corresponding to $v$ and $v'$. In \Cref{sec:plumbing,sec:crystal-algorithm}, we introduce \textit{appliances} and other crystal-theoretic machinery to track all possible changes for a given pair $v, v'$ as $u$ and $w$ vary. In certain cases, these appliances always differ by the same fragment of an hourglass plabic graph, which we call a \textit{(realized) plumbing} (\Cref{def:plumbing}). For example, when $r=4$ we will show that one may always obtain $\PM(u122w)$ from $\PM(u212w)$ by attaching the plumbing $\mathsf{XI}$ which is part of the growth rule $122 \to 212$ in \Cref{fig:growth-rules}. This is an example of a \textit{good degeneration} (\Cref{def:good_degens}).  In \Cref{sec:long-good-degenerations}, we show that all growth rules are in fact good degenerations. A key feature of good degenerations is the appearance of extensive \textit{witnesses}; for example, in the growth rule $14\overline{2}^k4 \to 41\overline{2}^k4$, the witnesses are $\overline{2}^k4$. Single witnesses appear already in Knuth equivalence, but we require further control of certain crystal configurations, which is provided by the use of multiple witnesses.

However, not all possible good degenerations are appropriate for our purposes. Some good degenerations arise from \textit{virtual} plumbings, which do not have a particular hourglass plabic graph realization. An example which is important to some of our proofs is \Cref{lem:long_good_degens}(iv). More problematically, some realizations do not have good properties.
For example, one may show that $\overline{3}2 \leftrightarrow 2\overline{3}$ with an $\mathsf{H}$ plumbing is a good degeneration in either direction; see \Cref{lem:pair_degens}(i). Applying this repeatedly would result in an infinite loop. Conceptually, doing so results in bad double crossings of $\trip_1$-strands. Indeed, one may realize this $\mathsf{H}$ as in \Cref{fig:bad-good-degens} (left) involving a non-reduced square. However, if we allowed this in the growth rules, the patterns of barred and unbarred letters would result in hourglass plabic graphs which are not bipartite. Even more problematically, consider the realized plumbing $\mathsf{L}$ in \Cref{fig:bad-good-degens} (right). It may be used in the identity good degeneration $a\overline{b} \to a\overline{b}$ for $a, b$ arbitrary, resulting in infinite loops and bad double crossings of $\trip_1$- and $\trip_2$-strands.

\begin{figure}[htbp]
    \centering
    \includegraphics[width=0.35\textwidth]{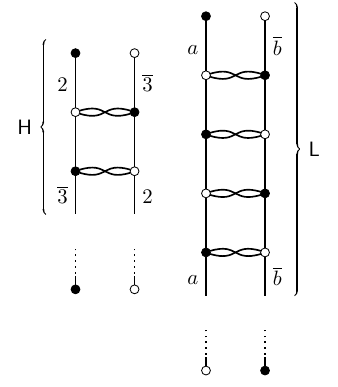}
    \caption{(\textsc{Left}) A problematic realization of the $\mathsf{H}$ plumbing used in \Cref{lem:pair_degens}(i). (\textsc{Right}) A problematic realization $\mathsf{L}$ of the identity plumbing.}
    \label{fig:bad-good-degens}
\end{figure}

The specific growth rules in \Cref{fig:growth-rules} have been carefully chosen to avoid infinite loops and in fact to monotonically increase the lattice word in a certain total order. In \Cref{sec:nicelabelings,sec:completeness}, we show that they always terminate at $\varnothing$. A key argument, \Cref{lem:growth_lex}, shows that the specific labelings used in the growth rules inductively preserve the property of having a unique $\tlex$-minimal boundary word. Consequently, \Cref{thm:nice-unitriangularity} shows that the invariant associated to the output of the growth algorithm has $\mathbb{Z}[q,q^{-1}]$-unit leading coefficient. To extend this result to all hourglass plabic graphs in the move-equivalence class, in \Cref{sec:nicelabelings} we introduce the machinery of \textit{nice labelings} of \textit{linearized diagrams} where we perform a sort of surgery on such graphs.

The main remaining piece to complete the verification of \Cref{thm:growth_algorithm} is to verify that the growth rules preserve full reducedness. This is carried out in \Cref{sec:growth-reduced} by a  careful analysis of the possible configurations of appliances and plumbings arising from the growth rules. All of the preceding arguments are combined in \Cref{sec:growth-theorem-proof} to prove \Cref{thm:growth_algorithm}.

\begin{remark}
    The growth rules in \Cref{fig:growth-rules} were the result of a lengthy iterative process of trial and error. Extensive computer experimentation was involved. We first developed certain heuristics guided by the $\trip_\bullet = \prom_\bullet$ property to compute likely valid properly labeled symmetrized six-vertex configurations corresponding to an input balanced lattice word. The growth rules in \Cref{fig:growth-rules} were largely empirically obtained by examining thousands of pairs of webs which differed only locally by small modifications. For example, the general rule $14 \to 41$ was originally included, though $14$ and $41$ are not $\SL_r$-Knuth equivalent. It turns out this rule is valid given the presence of any of a very wide variety of witnesses, some of which are listed in the long growth rules. Indeed, the crystal of $14$ is a subcrystal of the crystal for $41$.
    
    The crystal-theoretic machinery was also developed in steps as part of the iterative trial and error, first proving the $\trip_\bullet = \prom_\bullet$ property and then extending to the various other necessary properties. Using a computer implementation of the crystal-theoretic good degeneration algorithm allowed us to find many virtual or realized good degenerations and focus on finding enough rules to ensure the algorithm was able to handle any balanced lattice word. The long growth rules were then uncovered as the key missing piece to complete the argument, which required understanding the finite state machines as in \Cref{fig:long-degens-DFA-412p3p4}.
\end{remark}

\begin{remark}\label{rem:complete-redundant}
    One may ask whether the growth rules in \Cref{fig:growth-rules} are complete or redundant. They are complete in the sense that they reduce arbitrary balanced lattice words for $r=4$ to $\varnothing$. They are incomplete in the sense that, for instance, $2233 \to 2323$ with realized plumbing $\mathsf{IXI}$ is in fact a good degeneration which preserves monotonicity. Using this additional rule for $11223344$ gives a web related by a square move to the output of the algorithm without this rule. However, our argument shows we do not need to allow this rule, and we didn't add it to avoid redundancy. It might be possible to remove some rules from \Cref{fig:growth-rules} without affecting \Cref{thm:growth_algorithm}. Indeed, while we do not need this fact, we note that $2$ of the $14$ growth rules from \cite{Khovanov-Kuperberg} may be removed while still reducing all balanced lattice words with $r=3$ to $\varnothing$. In our notation, these redundant $r=3$ rules are $2\overline{2} \to \overline{1}1$ and $\overline{2}2 \to 3\overline{3}$ with $\mathsf{H}$ plumbing.
    
    In this way, we see that our proposed list is not unique. To any proposed list of growth rules which are good degenerations preserving monotonicity, we could add the growth rules in \Cref{fig:growth-rules}. We suspect there is no finite list.  In particular, in the long growth rules $14\overline{2}^k3 \to 41\overline{2}^k3$, the final witness~$3$ plays a special role as in \Cref{lem:star-witness}. The conditions needed for us to allow a proposed rule are quite constrictive. The initial and final words of the rule must be $\SL_r$-Knuth equivalent, be good degenerations preserving $\trip_\bullet = \prom_\bullet$ using realizable plumbing, preserve monotonicity as in \Cref{sec:growth-reduced}, respect descents as in \Cref{sec:descents}, and satisfy the $\tlex$ conditions in \Cref{lem:growth_lex}. Indeed, the possible extra rule $2233 \to 2323$ mentioned above \textit{fails} \Cref{lem:growth_lex}(iii). The $\tlex$ order in \Cref{lem:growth_lex} is natural in the sense that applying crystal raising operators increases in $\tlex$ order. It is not clear if other orderings may be useful and yield a very different set of growth rules. It is an interesting question what conditions may need to be dropped to produce a finite list.
\end{remark}

\subsection{Plumbings, promotion appliances, and good degenerations}\label{sec:plumbing}

We use the notions of \textit{appliances} and \textit{plumbings} in \Cref{def:plumbing,def:appliance} to track partially defined trip and promotion permutations through each step of the growth algorithm; see \Cref{fig:plumbings-appliances} for an example. We summarize this application before giving technical details.

Given two balanced lattice words $uvw$ and $uv'w$, consider the permutation matrices $\PM^i(uvw)$ and $\PM^i(uv'w)$ encoding their $i$-th promotion permutations (see \Cref{def:prom-mats}). We are interested in substitutions $v \to v'$ where $\PM^i(uvw)$ and $\PM^i(uv'w)$ only change in the rows and columns corresponding to $v$ and $v'$. A promotion appliance will encode the portion of $\PM^i(uvw)$ or $\PM^i(uv'w)$ with rows corresponding to $v$ or $v'$. If $v \to v'$ is a substitution as in the growth rules, the two appliances differ only slightly. Attaching a plumbing to one of the appliances will implement the difference between them.

\begin{example}\label{ex:prom-appliances-1}
    Consider the growth rule $\overline{3}24 \to \overline{3}\overline{1}\overline{3}$ applied to the lattice words $w = 1\overline{4}\overline{4}\overline{3}2434\overline{2}\overline{1}$ and $w' = 1\overline{4}\overline{4}\overline{3}\overline{1}\overline{3}34\overline{2}\overline{1}$. We have $\prom_1(w) = 5\,10\,7\,3\,6\,2\,8\,4\,1\,9$ and $\prom_1(w') = 5\,10\,7\,3\,4\,2\,8\,6\,1\,9$, giving permutation matrices
    \[
    \arraycolsep=3.2pt
    \PM^1(w) = \left(\begin{array}{ccc|ccc|cccc}
\cdot & \cdot & \cdot & \cdot & 1 & \cdot & \cdot & \cdot & \cdot & \cdot \\
\cdot & \cdot & \cdot & \cdot & \cdot & \cdot & \cdot & \cdot & \cdot & 1 \\
\cdot & \cdot & \cdot & \cdot & \cdot & \cdot & 1 & \cdot & \cdot & \cdot \\\hline
\cdot & \cdot & 1 & \cdot & \cdot & \cdot & \cdot & \cdot & \cdot & \cdot \\
\cdot & \cdot & \cdot & \cdot & \cdot & 1 & \cdot & \cdot & \cdot & \cdot \\
\cdot & 1 & \cdot & \cdot & \cdot & \cdot & \cdot & \cdot & \cdot & \cdot \\\hline
\cdot & \cdot & \cdot & \cdot & \cdot & \cdot & \cdot & 1 & \cdot & \cdot \\
\cdot & \cdot & \cdot & 1 & \cdot & \cdot & \cdot & \cdot & \cdot & \cdot \\
1 & \cdot & \cdot & \cdot & \cdot & \cdot & \cdot & \cdot & \cdot & \cdot \\
\cdot & \cdot & \cdot & \cdot & \cdot & \cdot & \cdot & \cdot & 1 & \cdot
\end{array}\right)
\qquad
\PM^1(w') = \left(\begin{array}{ccc|ccc|cccc}
\cdot & \cdot & \cdot & \cdot & 1 & \cdot & \cdot & \cdot & \cdot & \cdot \\
\cdot & \cdot & \cdot & \cdot & \cdot & \cdot & \cdot & \cdot & \cdot & 1 \\
\cdot & \cdot & \cdot & \cdot & \cdot & \cdot & 1 & \cdot & \cdot & \cdot \\\hline
\cdot & \cdot & 1 & \cdot & \cdot & \cdot & \cdot & \cdot & \cdot & \cdot \\
\cdot & \cdot & \cdot & 1 & \cdot & \cdot & \cdot & \cdot & \cdot & \cdot \\
\cdot & 1 & \cdot & \cdot & \cdot & \cdot & \cdot & \cdot & \cdot & \cdot \\\hline
\cdot & \cdot & \cdot & \cdot & \cdot & \cdot & \cdot & 1 & \cdot & \cdot \\
\cdot & \cdot & \cdot & \cdot & \cdot & 1 & \cdot & \cdot & \cdot & \cdot \\
1 & \cdot & \cdot & \cdot & \cdot & \cdot & \cdot & \cdot & \cdot & \cdot \\
\cdot & \cdot & \cdot & \cdot & \cdot & \cdot & \cdot & \cdot & 1 & \cdot
\end{array}\right).
    \]
    Only portions of the matrices between the lines differ. This example is continued in \Cref{ex:prom-appliances-2}.
\end{example}

\begin{definition}
\label{def:plumbing}
  A \emph{plumbing} from set $B'$ to set $B$ is a tuple $\pi_\bullet = (\pi_1, \ldots, \pi_{r-1})$ of bijections $\pi_i \colon B' \sqcup B \to B' \sqcup B$, such that $\pi_s^{-1} = \pi_{r-s}$ for all $s \in [r-1]$. The \emph{transpose} $\pi_\bullet^\top$ of $\pi_\bullet$ is the plumbing defined by $\pi_s^\top = \pi_{r-s}$. The \emph{composite} of plumbings $\sigma_\bullet$ from $B''$ to $B'$ and $\pi_\bullet$ from $B'$ to $B$ is the plumbing $\pi_\bullet \circ \sigma_\bullet$ from $B''$ to $B$ where $(\pi_\bullet \circ \sigma_\bullet)_s \colon B'' \sqcup B \to B'' \sqcup B$ sends each $i \in B$ to the first value of $\cdots \circ \pi_s \circ \sigma_s \circ \pi_s(i)$ that is in $B'' \sqcup B$ and  sends each $i \in B''$ to the first value of $\cdots \circ \sigma_s \circ \pi_s \circ \sigma_s(i)$ that is in $B'' \sqcup B$.
\end{definition}

We visualize a plumbing $\pi_{\bullet}$  from $B'$ to $B$ as a network of connections from a line of nodes $B'$ at the top to nodes $B$ at the bottom. \Cref{fig:plumbings-appliances} shows this visualization for a single bijection $\pi_s$ in the tuple $\pi_{\bullet}$.  (One could visualize the entire plumbing on one diagram by drawing each bijection $\pi_s$ in a different color.)  We think of inputs $i$ to $\pi_s$ as inputs to the network, initially headed downward for $i \in B'$ and upward for $i \in B$. Outputs $i$ of $\pi_s$ are outputs of the network, headed upward for $i \in B'$ and downward for $i \in B$. Plumbings are composed by concatenating networks.

\begin{definition}
\label{def:appliance}
  An \emph{appliance} $g_\bullet = (g_1, \ldots, g_{r-1})$ on $A \sqcup B \sqcup C$ is a tuple of functions $g_s \colon B \to A \sqcup B \sqcup C$, such that $g_s(i) = j \Leftrightarrow g_{r-s}(j) = i$ for all $s \in [r-1]$ and $i, j \in B$. The transpose $g_\bullet^\top$ of $g_\bullet$ is the appliance defined by $g_s^\top = g_{r-s}$. A plumbing $\pi_\bullet$ from $B'$ to $B$ acts on an appliance $g_\bullet$ on $A \sqcup B \sqcup C$ on the right and results in an appliance $g_\bullet \cdot \pi_\bullet$ on $A \sqcup B' \sqcup C$ where $(g_\bullet \cdot \pi_\bullet)_s$ sends each $i \in B'$ to the first value of $\cdots \circ \pi_s \circ g_s \circ \pi_s(i)$ that is in $A \sqcup B' \sqcup C$.

  Finally, if $h \colon A \sqcup C \to A' \sqcup C'$ is any function, then $h \cdot g_\bullet$ is the appliance on $A' \sqcup B \sqcup C'$ where $(h \cdot g_\bullet)_s = \check{h} \circ g_s$ where $\check{h} \colon A \sqcup B \sqcup C \to A' \sqcup B \sqcup C'$ is the extension of $h$ that fixes $B$.
\end{definition}

 An appliance $g_\bullet$ on $A \sqcup B \sqcup C$ is visualized as a line of nodes from $A$, $B$, $C$. Inputs $i \in B$ to $g_s$ initially head downward through a network and finish by heading upward into $A$, $B$, or $C$. Plumbings act on on appliances also by concatenation. The actions satisfy $g_\bullet \cdot (\pi_\bullet \circ \sigma_\bullet) = (g_\bullet \cdot \pi_\bullet) \cdot \sigma_\bullet$ and $(h \cdot g_\bullet) \cdot \pi_\bullet = h \cdot (g_\bullet \cdot \pi_\bullet)$.

\begin{figure}[htbp]
    \centering
    \includegraphics[width=0.95\textwidth]{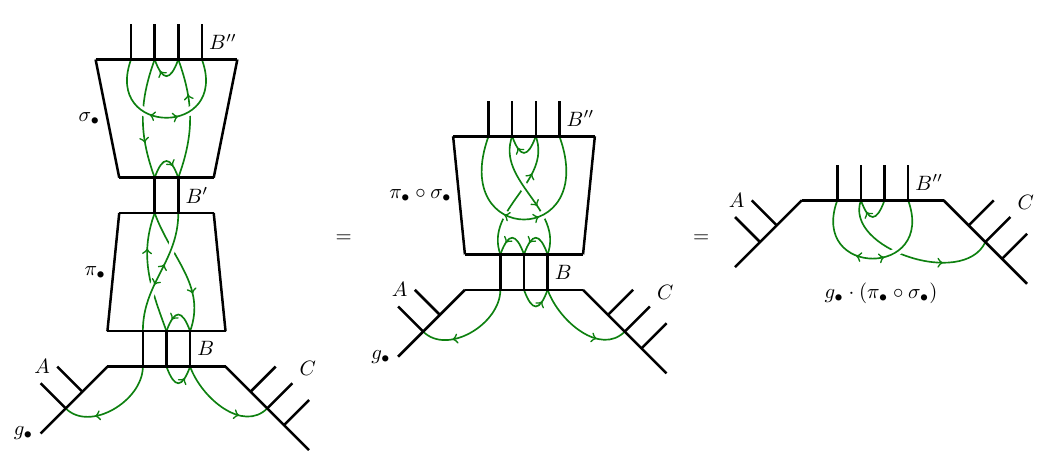}
    \caption{On the left is a bijection $\sigma_s$ of the plumbing $\sigma_\bullet$ from $B''$ to $B'$ concatenated to a bijection $\pi_s$ of the plumbing $\pi_\bullet$ from $B'$ to $B$, concatenated to a function $g_s$ of the appliance $g_\bullet$. In the middle is an equivalent network, where the two plumbings have been composed. On the right is the result of composing the plumbings with the appliance.}
    \label{fig:plumbings-appliances}
\end{figure}

A plumbing that arises from the trip permutations of a portion of an hourglass plabic graph (or a \symm six-vertex configuration) is \emph{realized}. A plumbing without such a chosen realization is called \emph{virtual}. There are natural notions of the \emph{identity} plumbing $\id$ and of the \emph{product} of two plumbings, which is given by horizontal concatenation and denoted by $\times$. For a rectangular fluctuating tableau $T$ or an hourglass plabic graph $G$, $\prom_\bullet(T)$ and $\trip_\bullet(G)$ can be thought of as appliances from $[n]$ to $\varnothing \sqcup [n] \sqcup \varnothing$.

\begin{definition}
  Suppose that $v$ is a word on $\mathcal{A}_r$ and that $uvw$ is a balanced oscillating lattice word. The \emph{promotion appliance} associated to the inclusion $v \to uvw$ is the appliance $\varrho_\bullet(u, v, w)$ on $A \sqcup B \sqcup C = \{1, \ldots, |u|\} \sqcup \{|u|+1, \ldots, |u|+|v|\} \sqcup \{|u|+|v|+1, \ldots, |u|+|v|+|w|\}$ defined by
      \[ \varrho_s(u,v,w)(i) = \prom_s(uvw)(i) \qquad \text{for }i \in B. \]
\end{definition}

The promotion appliance $\varrho_\bullet(u, v, w)$ encodes the rows of the promotion matrix $\PM(uvw)$ (see \Cref{def:prom-mats} and \cite[Def.~5.14]{fluctuating-paper}) corresponding to entries from $v$.

\begin{example}\label{ex:prom-appliances-2}
    Continuing \Cref{ex:prom-appliances-1}, the pieces of the promotion appliances for $\prom_1(w)$ and $\prom_1(w')$ are pictured below. A plumbing sending one to the other is also depicted, which in this case is realized as a fragment of a symmetrized six-vertex configuration (which we sometimes draw with edge orientations omitted, but trips drawn). 
    \[
    \includegraphics{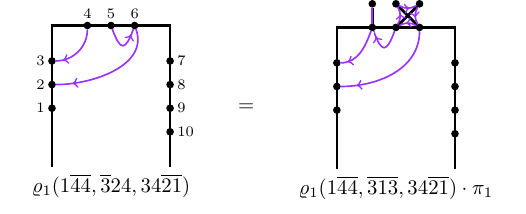}
    \]
\end{example}

The substitutions allowed in the growth algorithm are of the following sort. The growth rules in \Cref{fig:growth-rules} will be our main examples.

\begin{definition}\label{def:good_degens}
  A \emph{good degeneration} is a substitution $v_1 \cdots v_p \to v_1' \cdots v_q'$ where $v_i, v_i' \in \mathcal{A}_r$, together with a plumbing $\pi_\bullet$ from $[p]$ to $[q]$, such that the following hold for every balanced lattice word $w = w_1 \cdots w_n$:
  \begin{enumerate}[(i)]
    \item Replacing a consecutive substring $v_1 \cdots v_p$ in $w$ with $v_1' \cdots v_q'$ results in a new balanced lattice word $w' = w_1 \cdots w_j v_1' \cdots v_q' w_{j+p+1} \cdots w_n$.
    \item We have
      \begin{equation}\label{eq:good_degen.1}
        \prom_\bullet(w) = \prom_\bullet(w') \cdot \breve{\pi}_\bullet,
      \end{equation}
      where the plumbing $\breve{\pi}_\bullet = \id^j \times \pi_\bullet \times \id^{n-j-p}$ corresponds to the replacement $w \to w'$.
  \end{enumerate}
\end{definition}

Good degenerations have a natural $4$-fold symmetry. Recall the operations on words $\tau,\varpi,\epsilon$ from \Cref{def:varpietc}. These operations are applied to hourglass plabic graphs or realized plumbings as follows: $\varpi$ acts by inverting the color of all vertices; $\varepsilon$ acts by reflection through the vertical axis; and $\tau$ acts by both inversion and reflection. The following is a simple consequence of the interaction between these involutions and $\prom_\bullet(T)$ using \cite[Lem.~5.12]{fluctuating-paper}.

\begin{lemma}\label{lem:degen-symmetries}
  Suppose $v \to v'$ is a good degeneration with (realized) plumbing $\pi_\bullet$. Then $\iota(v) \to \iota(v')$ is a good degeneration with (realized) plumbing $\iota(\pi_\bullet)$ for all $\iota \in \{\tau, \varpi, \varepsilon\}$.
\end{lemma}

The $10$ families of growth rules in \Cref{fig:growth-rules} are orbits under this $4$-fold symmetry, where the left-right symmetry is $\tau$, the top-bottom symmetry is $\varepsilon$, and the diagonal symmetry is $\varpi$. The $14$ growth rules in \cite{Khovanov-Kuperberg} can similarly be grouped into $5$ families, two of size $4$ and three of size $2$. 

\subsection{A crystal-theoretic good degeneration algorithm}\label{sec:crystal-algorithm}

We now provide an algorithm for establishing that a given substitution satisfies \eqref{eq:good_degen.1}. Our proof relies on Kashiwara's theory of crystals \cite{Kashiwara-Crystals}. See \cite[\S8]{fluctuating-paper} for further background and conventions; our conventions mostly match those of the textbook \cite{Bump-Schilling}, except that we follow \cite{Kashiwara-Crystals} in using the original tensor product convention, which is opposite to that of \cite{Bump-Schilling}.

We briefly recall necessary background. Let $\mathcal{B}(\bigwedge\nolimits^{\underline{c}} V)$ be the $U_q(\fsl_r)$-crystal of words associated to the representation $\bigwedge\nolimits^{\underline{c}} V$. The highest weight elements of $\mathcal{B}(\bigwedge\nolimits^{\underline{c}} V)$ are the lattice words of type $\underline{c}$ and the balanced lattice words are the highest weight elements of weight zero. All crystals we consider will consist of unions of connected components of such crystals. We write $\mathcal{C}(v)$ for the connected crystal containing the word $v$ and $v^\uparrow$ and $v^\downarrow$ for its highest and lowest weight elements, respectively. The actions of the crystal raising and lowering operators, $e_i$ and $f_i$, can be computed for oscillating lattice words using the following well-known combinatorial rule. See \cite[Ex.~8.12]{fluctuating-paper} for an example.

\begin{definition}[The \textit{bracketing rule}]\label{def:bracketing-rule}
   To apply $e_i$ or $f_i$ to $w_1 \otimes \cdots \otimes w_n$ where each $w_j \in \pm [r]$ (and $1 \leq i \leq r-1$), do the following.
    \begin{enumerate}[(1)]
      \item Place $[$ below each $i$ and $\overline{i+1}$.
      \item Place $]$ below each $i+1$ and $\overline{i}$.
      \item Match brackets from the inside out.
      \item Now $f_i$ acts on the letter with the leftmost unmatched $[$ by replacing $i$ with $i+1$ or $\overline{i+1}$ with $\overline{i}$.
      \item Similarly, $e_i$ acts on the letter with the rightmost unmatched $]$ by replacing $i+1$ with $i$ or $\overline{i}$ with $\overline{i+1}$.
    \end{enumerate}
\end{definition}

The \textit{crystal raising algorithm} in \cite[Prop.~8.20]{fluctuating-paper} computes $\promotion(\osc(uvw))$ by cutting off the first letter, applying raising operators to reach the highest weight element, and appending the unique letter which yields the original weight. To facilitate computation, one may form a rhombus whose $i$-th row lists $\promotion^{i-1}(\osc(uvw))$, called a \textit{promotion-evacuation diagram} in \cite{fluctuating-paper}; see \Cref{ex:promotion-evacuation}. The promotion matrix $\PM(uvw)$ can be read off from the promotion-evacuation diagram by inferring where raising operators were applied and using the following fact. By \cite[Thm.~8.24]{fluctuating-paper}, $\PM^s(uvw)_{i,i+j}$ is $1$ if and only if the raising operator $e_s$ is applied at index $j$ when applying $\promotion$ to $\promotion^{i-1}(\osc(uvw))$, where $i+j$ is taken modulo $|uvw|$ (equivalently, if and only if $\prom_s(uvw)(i) = i+j$).

When computing the promotion-evacuation diagram of $uvw$, tracking the portion corresponding to $v$ results in regions I, II, III, and IV depicted in \Cref{fig:crystal-phases}. Region I ends with the first row where the subword corresponding to $u$ has been completely cut off, and the subword corresponding to $v$ has become the highest weight element $v^\uparrow$ in the crystal $\mathcal{C}(v)$ containing $v$. Dually, the bottom of region III lists $uvw$, and the subword corresponding to $v$ at the top of region III is the lowest weight element $v^\downarrow$ in $\mathcal{C}(v)$. Regions II and IV can be computed from $v^\uparrow$ and $v^\downarrow$ by evacuation and dual evacuation, respectively; see \cite[Def.~3.7]{fluctuating-paper}.

\begin{figure}[htbp]
    \centering
    \includegraphics{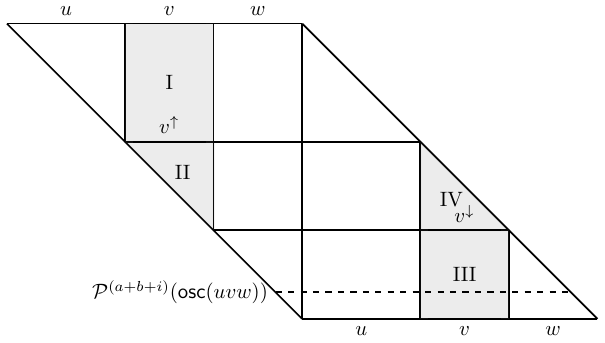}
    \caption{Phases corresponding to $v$ when repeatedly applying the crystal raising algorithm to compute $\PM(uvw)$.}
    \label{fig:crystal-phases}
\end{figure}

\begin{example}\label{ex:promotion-evacuation}
The promotion-evacuation diagram of $w = 1\overline{4}\overline{4}\overline{3}2434\overline{2}\overline{1}$ is below. Letters where a raising or lowering operator has been applied when using the crystal raising algorithm to compute promotion powers have been noted. Locations where $e_i$ or $f_i$ have been applied correspond to the locations of $1$'s in the $i$-th promotion matrix of $w$; see \Cref{ex:prom-appliances-1} for the $i=1$ case. The subword corresponding to $v = \overline{3}24$ has been highlighted in bold, together with regions I--IV from \Cref{fig:crystal-phases}.

\definecolor{greenish}{rgb}{0.05, 0.5, 0.06}
\definecolor{orangish}{rgb}{1.0,0.5,0}
\definecolor{purplish}{rgb}{.6,0.2,1}

\def\eone{\scriptstyle \textcolor{greenish}{e_1}}
\def\etwo{\scriptstyle \textcolor{orangish}{e_2}}
\def\ethree{\scriptstyle \textcolor{purplish}{e_3}}
\def\fone{\scriptstyle \textcolor{greenish}{f_1}}
\def\ftwo{\scriptstyle \textcolor{orangish}{f_2}}
\def\fthree{\scriptstyle \textcolor{purplish}{f_3}}

\def\promotionorbit{{
{"$1$","$\overline{4}$","$\overline{4}$","$\mathbf{\overline{3}}$","$\mathbf{2}$","$\mathbf{4}$","$3$","$4$","$\overline{2}$","$\overline{1}$"},
{"$\overline{4}$","$\overline{4}$","$\mathbf{\overline{3}}$","$\mathbf{1}$","$\mathbf{4}$","$3$","$4$","$\overline{4}$","$\overline{1}$","$4$"}
,{"$\overline{4}$","$\mathbf{\overline{3}}$","$\mathbf{1}$","$\mathbf{3}$","$2$","$4$","$\overline{4}$","$\overline{2}$","$4$","$\overline{1}$"}
,{"$\mathbf{\overline{4}}$","$\mathbf{1}$","$\mathbf{2}$","$1$","$4$","$\overline{4}$","$\overline{2}$","$4$","$\overline{1}$","$\overline{1}$"},
{"$\mathbf{1}$","$\mathbf{2}$","$1$","$3$","$\overline{4}$","$\overline{3}$","$4$","$\overline{1}$","$\overline{2}$","$\mathbf{\overline{1}}$"},
{"$\mathbf{1}$","$1$","$2$","$\overline{4}$","$\overline{3}$","$3$","$\overline{1}$","$\overline{2}$","$\mathbf{\overline{1}}$","$\mathbf{4}$"},
{"$1$","$2$","$\overline{4}$","$\overline{3}$","$3$","$\overline{2}$","$\overline{3}$","$\mathbf{\overline{1}}$","$\mathbf{3}$","$\mathbf{4}$"},
{"$1$","$\overline{4}$","$\overline{3}$","$3$","$\overline{3}$","$\overline{4}$","$\mathbf{\overline{1}}$","$\mathbf{3}$","$\mathbf{4}$","$4$"},{"$\overline{4}$","$\overline{3}$","$3$","$\overline{3}$","$\overline{4}$","$\mathbf{\overline{2}}$","$\mathbf{2}$","$\mathbf{4}$","$3$","$4$"},
{"$\overline{4}$","$1$","$\overline{3}$","$\overline{4}$","$\mathbf{\overline{2}}$","$\mathbf{2}$","$\mathbf{4}$","$3$","$4$","$\overline{1}$"},{"$1$","$\overline{4}$","$\overline{4}$","$\mathbf{\overline{3}}$","$\mathbf{2}$","$\mathbf{4}$","$3$","$4$","$\overline{2}$","$\overline{1}$"}
}}
\def\decorations{{{{3,"$\downarrow \eone$"},{7,"$\downarrow \ethree\circ\etwo$"},{7,""}},{{3,"$\downarrow \ethree$"},{4,"$\downarrow \etwo$"},{7,"$\downarrow \eone$"}},{{0,"$\downarrow \ethree$"},{2,"$\downarrow \etwo$"},{3,"$\downarrow \eone$"}},{{3,"$\downarrow \ethree$"},{5,"$\downarrow \etwo$"},{8,"$\downarrow \eone$"}},{{0,"$\uparrow \fone$"},{2,"$\downarrow \etwo$"},{5,"$\uparrow \fthree$"}},{{5,"$\uparrow \fone$"},{6,"$\uparrow \ftwo$"},{8,"$\uparrow \fthree$"}},{{0,"$\uparrow \fone$"},{4,"$\uparrow \ftwo$"},{5,"$\uparrow \fthree$"}},{{5,"$\uparrow \fone$"},{6,"$\uparrow \ftwo$"},{8,"$\uparrow \fthree$"}},{{0,"$\downarrow \ethree$"},{1,"$\uparrow \ftwo\circ\fone$"},{1,""}},{{1,"$\uparrow \fthree$"},{3,"$\uparrow \ftwo$"},{8,"$\uparrow \fone$"}}}}
\begin{tikzpicture}
  \foreach\k in {0,...,10}
  {
  \foreach \l in {0,...,9}
   \node at ($({(\l+\k)*0.75},{(9-\k)*0.75})$){\pgfmathparse{\promotionorbit[\k][\l]}\pgfmathresult};
  }
  \foreach\k in {0,...,9} {
  \foreach \i in {0,1,2} {
   \pgfmathsetmacro\x{\decorations[\k][\i][0]}
   \pgfmathsetmacro\d{\decorations[\k][\i][1]}
   \node[anchor=west] at ($({(\x+\k+0.71)*0.75},{(9-\k-0.5)*0.75})$){\d};
   
   }
  }
  \draw[dashed, gray] (2.5*0.75,9*0.75) -- (5.5*0.75,9*0.75) -- (5.5*0.75,6*0.75) --(2.5*0.75,6*0.75) -- (2.5*0.75,9*0.75);

  \draw[dashed, gray] (5.5*0.75,6*0.75) -- (5.5*0.75,3*0.75) -- (2.5*0.75,6*0.75);

    \draw[dashed, gray] (12.5*0.75,-1*0.75) -- (15.5*0.75,-1*0.75) -- (15.5*0.75,3*0.75) --(12.5*0.75,3*0.75) -- (12.5*0.75,-1*0.75);

\draw[dashed, gray] (15.5*0.75,3*0.75) -- (12.5*0.75,6*0.75) -- (12.5*0.75,3*0.75);
\end{tikzpicture}

\noindent Phases I and III yield the following paths from $v$ to the highest and lowest weight elements of its component within the crystal $\mathcal{C}(v)$:
\begin{align*}
  v &= \overline{3}24 \too{\eone} \overline{3}14 \too{\ethree} \overline{3}13 \too{\ethree,\etwo} \overline{4}12 = v^\uparrow \\
  v &= \overline{3}24 \too{\ftwo} \overline{2}24 \too{\fone,\ftwo} \overline{1}34 = v^\downarrow.
\end{align*}
\end{example}

Thus, associated to each $uvw$, we have a path $\vec{e}$ from $v$ to $v^\uparrow$ in the crystal $\mathcal{C}(v)$ recording the positions at which raising operators were applied, and a path $\vec{f}$ from $v$ to $v^\downarrow$ recording lowering operators. These paths can be read off from regions I and III, respectively, and depend on $u$ and $w$. The following structure encodes the portion of the promotion-evacuation diagram in regions I--IV using only crystal-theoretic data from these paths. We will see shortly that it results in a well-defined appliance corresponding to the columns of $\PM(uvw)$ with entries from $v$.

\begin{definition}\label{def:crystal-appliances}
  Let $v$ be an oscillating word with $|v| = b$. Fix increasing and decreasing paths
  \begin{align*}
    v &= v^0 \xrightarrow{e_{s_1}} v^1 \xrightarrow{e_{s_2}} \cdots \xrightarrow{e_{s_a}} v^a = v^{\uparrow}, \\
    v &= u^c \xrightarrow{f_{t_c}} u^{c-1} \xrightarrow{f_{t_{c-1}}} \cdots \xrightarrow{f_{t_1}} u^0 = v^{\downarrow}.
  \end{align*}
  Let $\vec{e} = (s_1, \ldots, s_a)$ and $\vec{f} = (t_c, \ldots, t_1)$. The \emph{crystal appliance} associated to these paths is the appliance $\rho_\bullet(\vec{e}, v, \vec{f})$ on $\{1, \ldots, a\} \sqcup \{a+1, \ldots, a+b\} \sqcup \{a+b+1, \ldots, a+b+c\}$ defined as follows:
  \begin{enumerate}[(1)]
    \item If $v^i$ is obtained from $v^{i-1}$ by $e_{s_i}$ acting on the $j$-th letter, set $\rho_{s_i}(a+j) = i$. 
    \item If $u^{i-1}$ is obtained from $u^i$ by $f_{t_i}$ acting on the $j$-th letter, set $\rho_{t_i}(a+j) = a+b+i$.
    \item If $\prom_{s}(v^{\uparrow}x)(i) = j$ for some $i,j \in [b]$ and $s \in [r-1]$, set $\rho_s(a+j) = a+i$. Here $x$ is any word which makes $v^\uparrow x$ a balanced lattice word.
  \end{enumerate}
\end{definition}

The fact that crystal appliances are well-defined appliances will be verified in the course of the proof of \Cref{lem:crystal-phases}. We will later refer to \Cref{def:crystal-appliances}(1) as phase I, (2) as phase III, and (3) as phases II and IV, which will correspond to regions I--IV in promotion appliances.

The following result relates promotion appliances to crystal appliances. Recall that $\varrho_\bullet(u, v, w)$ corresponds to the rows of $v$ in $\PM(uvw)$ while $\rho_\bullet(\vec{e}, v, \vec{f})$ corresponds to the columns of $v$ in $\PM(uvw)$. However, since $\prom_s(uvw) = \prom_{r-s}(uvw)^{-1}$, $\varrho_\bullet(u, v, w)^\top$ also corresponds to the columns of $v$ in $\PM(uvw)$.

\begin{lemma}\label{lem:crystal-phases}
  Let $uvw$ be a balanced oscillating lattice word. Then there exist $\vec{e}$, $\vec{f}$, and $h$ such that
  \begin{equation}\label{eq:crystal-phases}
    \varrho_\bullet(u, v, w)^\top = h \cdot \rho_\bullet(\vec{e}, v, \vec{f}).  
  \end{equation}
  Here $|u|=a$, $|v|=b$, $|w|=c$, $|\vec{e}|=a'$, $|\vec{f}| =c'$, and $h \colon \{1, \ldots, a'\} \sqcup \{a'+b+1, \ldots, a'+b+c'\} \to \{1, \ldots, a\} \sqcup \{a+b+1, \ldots, a+b+c\}$ is the disjoint union of two weakly increasing functions.
\end{lemma}

\begin{proof}
  Consider repeatedly applying the crystal raising algorithm to compute the promotion-evacuation diagram of $uvw$. During this procedure, the subword corresponding to $\osc(v)$ shifts leftward as initial letters are repeatedly cut off. The initial phase I corresponding to region I in \Cref{fig:crystal-phases} consists of $a$ promotion steps. After this phase, the subword corresponding to $v$ has shifted to the beginning and has become the highest weight word $v^\uparrow$. Each promotion step computes one row of $\PM(uvw)$. Raising operators applied outside of the subword do not appear in the columns of $v$ in $\PM(uvw)$ and may be ignored. If $a'$ raising operators are applied to the subword during phase I, we have a path $v \to v^\uparrow$ in $\mathcal{C}(v)$ with $a'$ steps, by the associativity of the crystal tensor product. If $e_s$ is applied at the $j$-th letter of the subword during step $i'$ of the crystal path, then $e_s$ is applied at the $(a+j)$-th letter of the original word when computing some row $i$ of $\PM(uvw)$. In this case, $\rho_s(a+j) = i'$ and $\prom_s(uvw)(i) = a+j$. The function $i' \mapsto i$ is weakly increasing and fulfills \eqref{eq:crystal-phases} for these values.

  After phase I, the next phase II consists of $b$ promotion steps. It computes the upper triangle of the rows and columns of $v$ in $\PM(uvw)$. This calculation depends only on the initial $b$ letters, namely $v^\uparrow$. Hence $\prom_s(uvw)(a+i) = a+j$ if and only if $\prom_s(v^\uparrow x)(i) = j$ for $1 \leq i < j \leq b$, where $x$ is any word such that $v^\uparrow x$ is balanced. Thus \eqref{eq:crystal-phases} also holds for these values. Note that when any particular letter of $v$ reaches the beginning of the lattice word in phase II, it has been transformed into $1$ or $\overline{r}$.

  Phases III and IV, as illustrated in \Cref{fig:crystal-phases}, determine the other half of the columns of $v$ in $\PM(uvw)$. We may imagine the promotion algorithm is applied in reverse from the bottom up since $\promotion^n(uvw) = uvw$, in which case phase III produces a decreasing path $v \to v^\downarrow$ of length $c'$. During phase IV, the relevant lower triangle of $\PM(v^\uparrow x)$ agrees with that of $\PM(uvw)$ by \cite[Thm.~8.25]{fluctuating-paper}. Thus \eqref{eq:crystal-phases} holds.

  Note that if we track the letters in a particular column of $v$ in $\PM(uvw)$ through this process from phase II to I to III to IV, we have $1 \to \cdots \to r$ or $\overline{r} \to \cdots \to \overline{1}$, so $\rho_\bullet(\vec{e}, v, \vec{f})$ is indeed well-defined. The same argument applies even if $\vec{e}$ and $\vec{f}$ have not necessarily come from some $uvw$, so crystal appliances are well-defined in general.
\end{proof}

If $v$ is a subword of a balanced oscillating lattice word $z=uvw$, we write $\vec{e}(v,z)$ and $\vec{f}(v,z)$ for a choice of tuples which satisfy 
\[
\varrho_\bullet(u,v,w)^{\top}=h \cdot \rho_\bullet(\vec{e}(v,z),v,\vec{f}(v,z))
\]
for some $h$. Such $\vec{e}(v,z)$ and $\vec{f}(v,z)$ exist by \Cref{lem:crystal-phases}.

The following result is our main tool for proving that a particular substitution is a good degeneration and, in turn, for proving the validity of the growth rules. Our setup applies for general $r$ and in particular can be used to easily re-verify the $r=3$ growth rules of Khovanov--Kuperberg \cite{Khovanov-Kuperberg}. In this paper we focus on the case $r=4$, but we aim in future work to apply these techniques for higher $r$ (and in other Lie types).

\begin{theorem}\label{thm:good_degen}
  Let $v, v'$ be words on $\mathcal{A}_r$ with $B \sqcup B' =\{1,\ldots,|v|\} \sqcup \{|v|+1,\ldots,|v|+|v'|\}$. Suppose there is a $U_q(\fsl_r)$-crystal isomorphism $\mathcal{C}(v) \cong \mathcal{C}(v')$ sending $v$ to $v'$. If $\pi_\bullet$ is a plumbing from $B$ to $B'$ such that
  \begin{equation}\label{eq:strong-equivalence}
    \rho_\bullet(\vec{e}, v, \vec{f})^\top = \rho_\bullet(\vec{e}, v', \vec{f})^\top \cdot \pi_\bullet
  \end{equation}
  for all raising and lowering paths $\vec{e}$ and $\vec{f}$, then $v \to v'$ with plumbing $\pi_\bullet$ is a good degeneration.
\end{theorem}

\begin{proof}
  Suppose $uvw$ is a balanced lattice word. Since $v$ and $v'$ are assumed to be corresponding vertices in isomorphic crystal graphs, $uv'w$ is a balanced lattice word. Moreover, from the proof of \Cref{lem:crystal-phases} and \Cref{fig:crystal-phases}, the portions of $\PM(uvw)$ and $\PM(uv'w)$ corresponding to rows and columns from $u$ or $w$ are the same. Hence \eqref{eq:good_degen.1} holds here directly. By symmetry, it now suffices to consider only the rows corresponding to $v$.

   By \Cref{lem:crystal-phases} and \eqref{eq:strong-equivalence},
  \begin{align*}
     \varrho_\bullet(u, v', w) \cdot \pi_\bullet &= h \cdot \rho_\bullet(\vec{e}, v', \vec{f})^\top \cdot \pi_\bullet \\
      &= h \cdot \rho_\bullet(\vec{e}, v, \vec{f})^\top = \varrho_\bullet(u, v, w) .
  \end{align*}
  Here the fact that there is a crystal isomorphism sending $v$ to $v'$ implies that $\vec{e}$, $\vec{f}$, and $h$ from \Cref{lem:crystal-phases} can be taken to be the same. We now have \eqref{eq:good_degen.1}. 
\end{proof}

\begin{example}\label{ex:good_degen}
  For any proposed substitution $v \to v'$, \Cref{thm:good_degen} reduces the \textit{a priori} infinite list of conditions in \eqref{eq:good_degen.1} to a finite calculation. We summarize part of this calculation for the good degeneration $\overline{3}24 \to \overline{3}\overline{1}\overline{3}$ with the following realized plumbing:

  \begin{center}
    \includegraphics[scale=1.2]{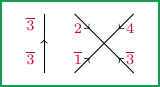}
  \end{center}
  
  \noindent First, we compute the portions of the crystals above and below $\overline{3}24$ and $\overline{3}\overline{1}\overline{3}$ and pick corresponding paths to the highest and lowest weight elements:

  \begin{center}
    \includegraphics[scale=0.9]{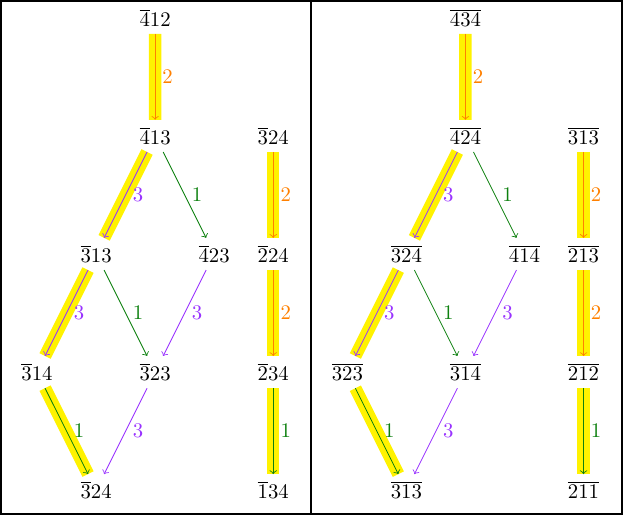}
  \end{center}

  \noindent Each path yields a possible column in $\PM(uvw)$. We then compute the crystal appliances and verify that applying $\pi_\bullet$ to $\rho_\bullet(\vec{e}, v', \vec{f})^\top$ yields $\rho_\bullet(\vec{e}, v, \vec{f})^\top$:

  \begin{center}
    \includegraphics[width=0.95\textwidth]{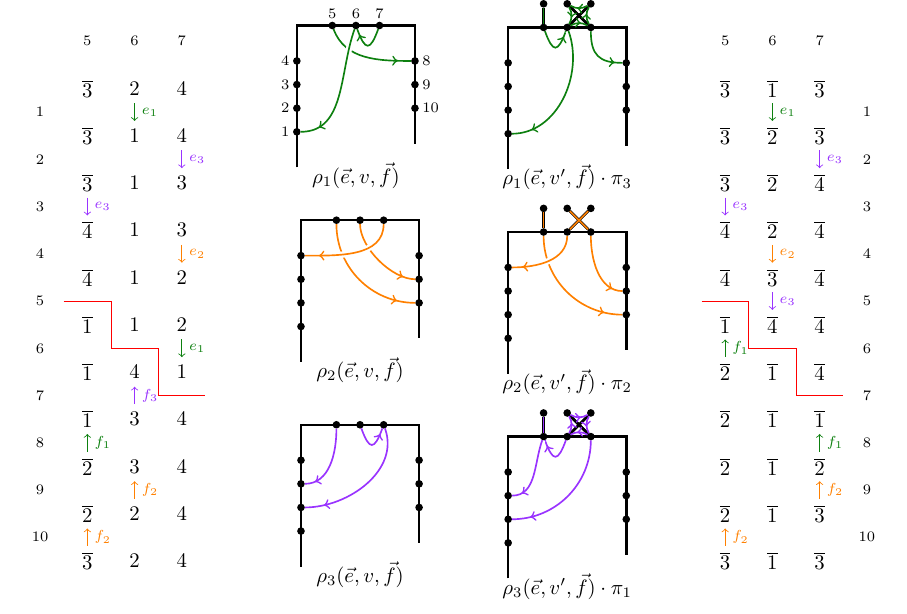}
  \end{center}

  In \Cref{ex:prom-appliances-1} and \Cref{ex:prom-appliances-2}, the promotion appliances also satisfied
    \[ \varrho_1(1\overline{4}\overline{4}, \overline{3}24, 34\overline{2}\overline{1}) = \varrho_1(1\overline{4}\overline{4}, \overline{3}\overline{1}\overline{3}, 34\overline{2}\overline{1}) \cdot \pi_1. \]
  To connect the promotion appliances and the crystal appliances as in \Cref{lem:crystal-phases}, we must use $h(1)=1, h(2)=2, h(3)=h(4)=3$ and $h(8)=h(9)=8, h(10)=10$, giving in particular
  \[ \varrho_1(1\overline{4}\overline{4}, \overline{3}24, 34\overline{2}\overline{1}) = h \cdot \rho_3(\vec{e}, \overline{3}24, \vec{f})
  \qquad\text{and}\qquad
  \varrho_1(1\overline{4}\overline{4}, \overline{3}\overline{1}\overline{3}, 34\overline{2}\overline{1}) = h \cdot \rho_3(\vec{e}, \overline{3}\overline{1}\overline{3}, \vec{f}). \]
\end{example}

\begin{example}
    For a non-example, consider the candidate unwitnessed rule $\overline{4}2 \to 2\overline{4}$ with the  plumbing $\mathsf{X}^{-+}_{+-}$ from \Cref{fig:plumbing-degens}. The witnessed version $\overline{4}24 \to 2\overline{4}4$ is one of the growth rules in \Cref{fig:growth-rules}. One may check as in \Cref{ex:good_degen} that $\overline{4}24 \to 2\overline{4}4$ is a good degeneration. However, $\overline{4}2 \to 2\overline{4}$ fails the check for one of the pairs of paths, even though $\overline{4}2 \sim_K 2\overline{4}$. Moreover, if one were to add the unwitnessed rule $\overline{4}2 \to 2\overline{4}$ to the growth algorithm, the result produces non-fully reduced hourglass plabic graphs as in \Cref{fig:bad-degeneration}.
\end{example}

\begin{figure}[htbp]
    \centering
    \raisebox{1.26cm}{\includegraphics[scale=0.85]{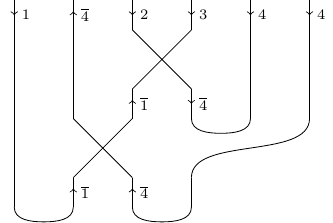}}
    \hspace{1.4cm}\includegraphics[scale=0.85]{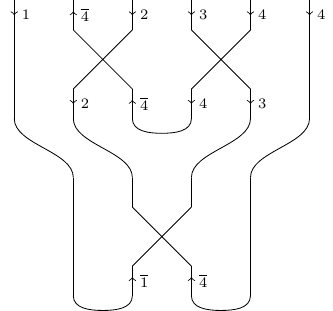}
    \caption{(\textsc{Left}) The result of applying the growth rules in \Cref{fig:growth-rules} to $1\overline{4}2344$. (\textsc{Right}) Possible output if the unwitnessed rule $\overline{4}2 \to 2\overline{4}$ were allowed. The result is non-fully reduced and fails the $\trip_\bullet = \prom_\bullet$ property.}
    \label{fig:bad-degeneration}
\end{figure}

\begin{proposition}\label{prop:short_good_degen}
  The $88$ short rules of length $2$ or $3$ in \Cref{fig:growth-rules} are good degenerations.
\end{proposition}

\begin{proof}
  In each case, the result follows from \Cref{thm:good_degen} exactly as in \Cref{ex:good_degen}. The calculations are lengthy, routine, and omitted. Using \Cref{lem:degen-symmetries} reduces the labor considerably. A computer implementation of the algorithm is available at \cite[\S3]{sl4-web-basis-code}.
\end{proof}

\subsection{Long good degenerations}\label{sec:long-good-degenerations}

We now prove that the two families of long growth rules in \Cref{fig:growth-rules} are also good degenerations. Our argument will involve the plumbings in \Cref{fig:plumbing-degens}.

\begin{figure}[htbp]
    \centering
    \includegraphics[width=0.9\textwidth]{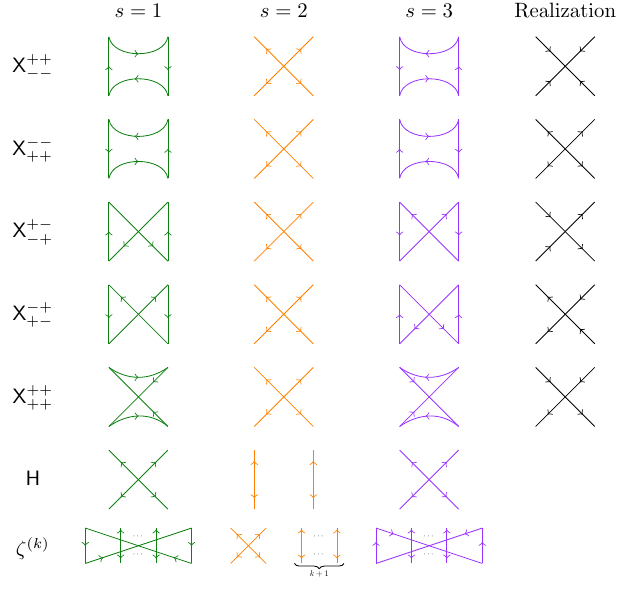}
    \caption{Some realized and virtual plumbings needed to establish the long growth rules of \Cref{fig:growth-rules}.}
    \label{fig:plumbing-degens}
\end{figure}

\begin{lemma}\label{lem:pair_degens}
  The following are good degenerations, which are their own inverses.
  \begin{enumerate}[(i)]
    \item For plumbing $\sfH$: $\overline{3}2 \leftrightarrow 2\overline{3}$, $\overline{2}2 \leftrightarrow 3\overline{3}$, $\overline{2}3 \leftrightarrow 3\overline{2}$, $\overline{3}\overline{1} \leftrightarrow 24$, $\overline{2}\overline{1} \leftrightarrow 34$, $\overline{3}\overline{2} \leftrightarrow 14$, $\overline{4}\overline{2} \leftrightarrow 13$, $\overline{4}\overline{3} \leftrightarrow 12$.
    \item For plumbing $\sfX^{+-}_{-+}$: $1\overline{2} \leftrightarrow \overline{2}1$, $2\overline{1} \leftrightarrow \overline{1}2$, $2\overline{2} \leftrightarrow \overline{1}1$.
    \item For plumbing $\sfX^{-+}_{+-}$: $\overline{3}4 \leftrightarrow 4\overline{3}$, $\overline{4}3 \leftrightarrow 3\overline{4}$, $\overline{3}3 \leftrightarrow 4\overline{4}$.
  \end{enumerate}
\end{lemma}

\begin{proof}
This is a direct verification using \Cref{thm:good_degen}. A computer implementation is available at \cite[\S4]{sl4-web-basis-code}.
\end{proof}

\begin{remark}
  The plumbings in \Cref{lem:pair_degens}(ii) and (iii) are only realized in one direction. For instance, if the substitution $\overline{2}1 \to 1\overline{2}$ had a realizable plumbing, it would need to be $\sfX^{-+}_{+-}$ in order for the six-vertex edge directions to be consistent with the label signs; however this is not a good degeneration. Instead, \Cref{lem:pair_degens} uses a different (non-realizable) plumbing. In general, good degenerations are not invertible, since plumbings need not possess inverses.
\end{remark}

\begin{lemma}\label{lem:long_good_degens}
  The following are good degenerations, for all $k \in \mathbb{Z}_{\geq 0}$.
  \begin{enumerate}[(i)]
    \item For plumbing $\sfX^{--}_{++} \times \id^{k+1}$: $\overline{3}\overline{2}\overline{2}^k4 \to 41\overline{2}^k4$.
    \item For plumbing $\sfX^{++}_{++} \times \id^{k+1}$: $14\overline{2}^k4 \to 41\overline{2}^k4$.
    \item For plumbing $\sfX^{++}_{--} \times \id^{k+1}$: $23\overline{2}^k4 \to \overline{1}\overline{4}\overline{2}^k4$.
    \item For plumbing $\zeta^{(k)}$: $41\overline{2}^k4 \to \overline{3}\overline{2}\overline{2}^k4$.
  \end{enumerate}
  Moreover, the composite $\zeta^{(k)} \circ (\sfX^{--}_{++} \times \id^{k+1})$ acts as the identity on the appliance $\rho_\bullet(\vec{e}, \overline{3}\overline{2}\overline{2}^k4, \vec{f})^\top$.
\end{lemma}

\begin{proof}
  Consider (i). In \Cref{fig:long-degens-DFA-3p2p2p4}, we depict finite state machines representing all elements reachable from $\overline{3}\overline{2}\overline{2}^k4$ by $e_i$'s or by $f_i$'s, together with information equivalent to $\prom_\bullet(v^\uparrow x)$. One may check the correctness of \Cref{fig:long-degens-DFA-3p2p2p4} directly, by using the bracketing rule in \Cref{def:bracketing-rule} for the action of crystal raising and lowering operators. A full version of a similar argument is given in the proof of \Cref{lem:long-degens-DFA-412p3p4} below. Likewise, \Cref{fig:long-degens-DFA-412p4} depicts equivalent information for the word $41\overline{2}^k4$.
  
  We may read off crystal appliances from this data, resulting in \Cref{fig:long-degens-3p2p2p4-412p4-appliances}. To do this, start at the bottom of the large state machine in \Cref{fig:long-degens-DFA-3p2p2p4} and apply $e_2$ some number $m$ times to the positions of the initial $\overline{2}^k$. This results in a thick arrow of width $m$ from these positions in $\rho_2(\vec{e}, \overline{3}\overline{2}\overline{2}^k4, \vec{f})$ to the first dot in the lower left. The rest of the appliance similarly records which operators act in which positions as we move up the finite state machine and through the rest of \Cref{fig:long-degens-DFA-3p2p2p4}. Thus \eqref{eq:strong-equivalence} holds in case (i). Now (ii) follows from (i) and \Cref{lem:pair_degens}(i) using $14 \to \overline{3}\overline{2}$ and the observation that $\mathsf{X}^{--}_{++} \circ \mathsf{H} = \mathsf{X}^{++}_{--}$. The claim about the composite may also be verified directly from \Cref{fig:plumbing-degens,fig:long-degens-3p2p2p4-412p4-appliances}. One may similarly verify (iv) using \Cref{fig:long-degens-DFA-412p4,fig:long-degens-DFA-3p2p2p4}.
  
  As for (iii), when $k=0$, this is a short growth rule and therefore covered by \Cref{prop:short_good_degen}. Now suppose $k > 0$. Consider the finite state machines in \Cref{fig:another_machine,fig:another_another_machine}, which may be verified directly from the bracketing rule. It remains to either draw analogues of \Cref{fig:long-degens-3p2p2p4-412p4-appliances}, or to verify the necessary conditions directly from \Cref{fig:another_machine,fig:another_another_machine}. In particular, the reader may check that for any $\vec{e}, \vec{f}$ paths, when operators are applied at the intermediate $k-1$ positions starting from the $4$-th letter, these operators act in the same place for both $23\overline{2}^k4$ and $\overline{1}\overline{4}\overline{2}^k4$. Thus the corresponding appliances are identical in these positions, in agreement with the $\id^{k-1}$ fragment of the plumbing. Tracking the first three positions and the final position provides the rest of the appliances.
\end{proof}

\begin{figure}[htbp]
    \centering
    \includegraphics[width=0.95\textwidth]{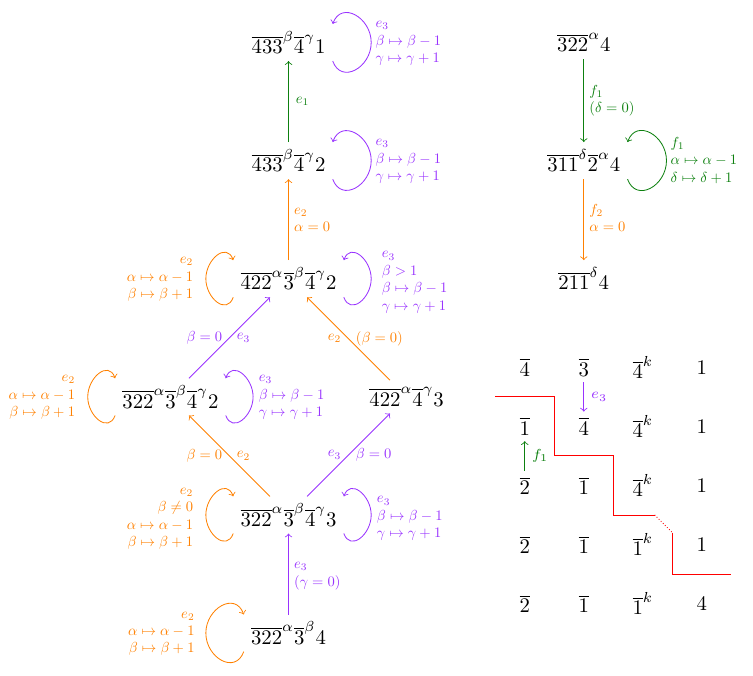}
    \caption{Finite state machines for $\overline{3}\overline{2}\overline{2}^k4$. Here $\alpha+\beta+\gamma+\delta=k$.}
    \label{fig:long-degens-DFA-3p2p2p4}
\end{figure}

\begin{figure}[htbp]
    \centering
    \includegraphics[width=0.95\textwidth]{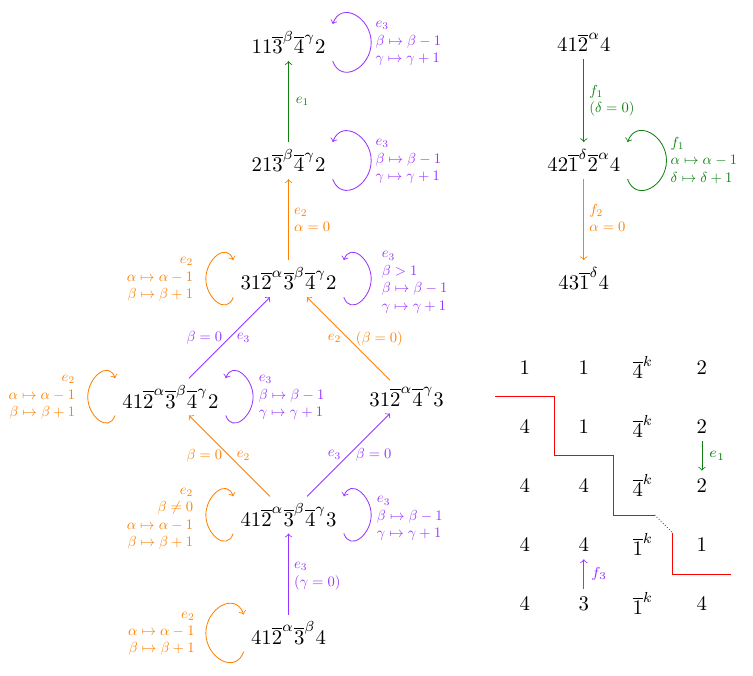}
    \caption{Finite state machines for $41\overline{2}^k4$.}
    \label{fig:long-degens-DFA-412p4}
\end{figure}

\begin{figure}[htbp]
    \centering
    \includegraphics[width=0.95\textwidth]{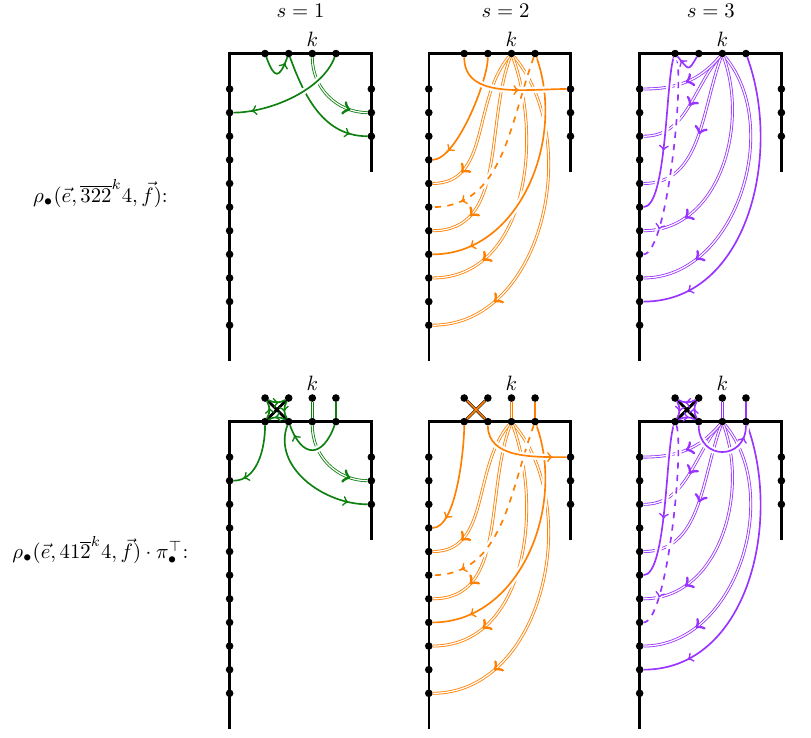}
    \caption{Crystal appliances for $\overline{3}\overline{2}\overline{2}^k4$ and $41\overline{2}^k4$. The plumbing in this case is $\pi_\bullet = \mathsf{X}^{--}_{++}  \times \id^{k+1}$.}
    \label{fig:long-degens-3p2p2p4-412p4-appliances}
\end{figure} 

\begin{figure}[htbp]
    \centering
    \includegraphics[width=0.95\linewidth]{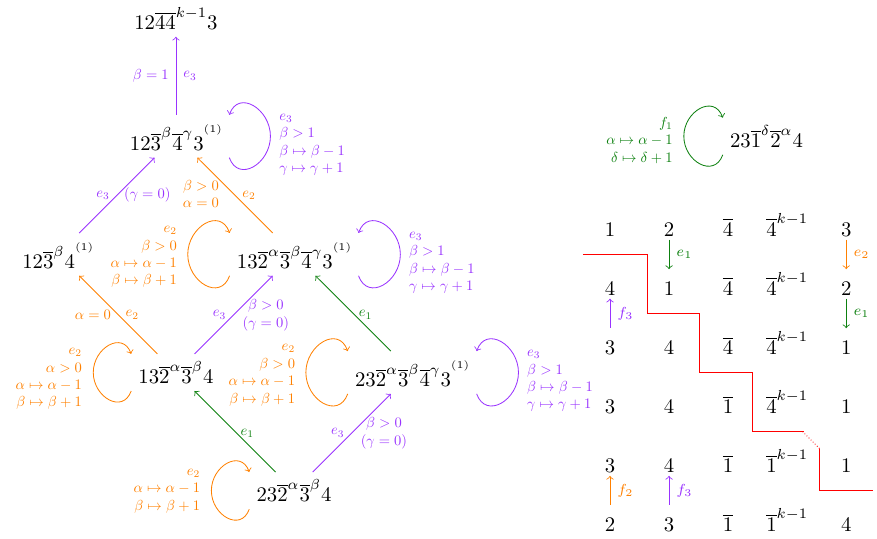}
    \caption{Finite state machines for $23\overline{2}^k 4$, where we assume $\alpha + \beta + \gamma + \delta = k > 0$. Here the superscript \raisebox{1ex}{$\scriptscriptstyle{(1)}$} on a state means that $\beta > 0$.}
    \label{fig:another_machine}
\end{figure}

\begin{figure}[htbp]
    \centering
    \includegraphics[width=0.95\linewidth]{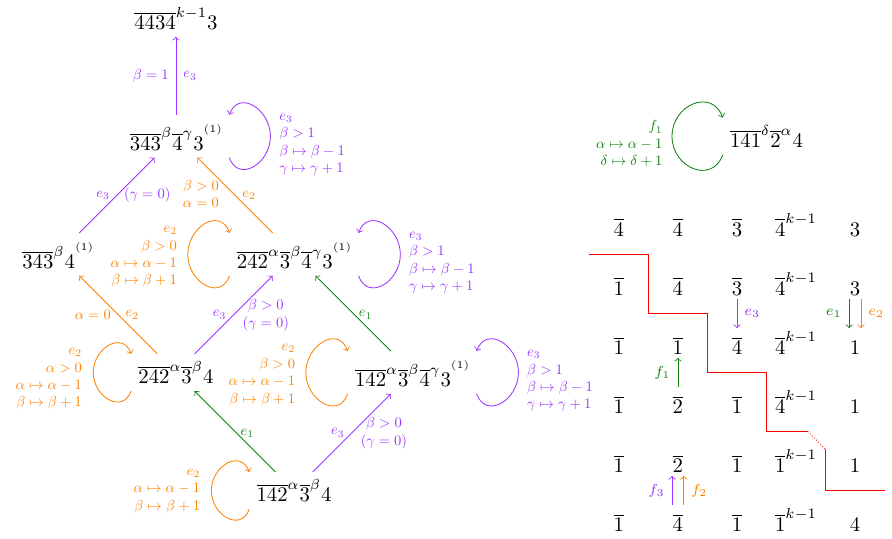}
    \caption{Finite state machines for $\overline{1} \overline{4} \overline{2}^k  4$, where we assume $\alpha + \beta + \gamma + \delta = k > 0$. Here the superscript \raisebox{1ex}{$\scriptscriptstyle{(1)}$} on a state means that $\beta>0$.}
    \label{fig:another_another_machine}
\end{figure}

\begin{theorem}\label{thm:R8}
  The growth rules in \Cref{fig:growth-rules} are all good degenerations.
\end{theorem}

\begin{proof}
 The short rules were handled in \Cref{prop:short_good_degen}. It remains to consider the two families of long rules in \Cref{fig:growth-rules}. For these, we induct on length; the length-$3$ base cases are contained among the short rules. We focus on $14wx$ where $w \in  \{\overline{3}, \overline{2}\}^*$ and $x \in \{2, 3, 4, \overline{1}\}$; the argument for $23wx$ is identical. First suppose $x=2$.
  \begin{itemize}
    \item If $w=\cdots\overline{2}$, apply $14\cdots\overline{2}2 \to 14\cdots3\overline{3} \to 41\cdots3\overline{3} \to 41\cdots\overline{2}2$. The first and third steps use the plumbings in \Cref{lem:pair_degens}(i) and the second step uses induction. The first and third plumbings cancel each other.
    \item If $w=\cdots\overline{3}$, apply $14\cdots\overline{3}2 \to 14\cdots2\overline{3} \to 41\cdots2\overline{3} \to 41\cdots\overline{3}2$ likewise.
  \end{itemize}
  The $x=3$ and $x=\overline{1}$ cases are extremely similar and are omitted. 
  
  Now suppose $x=4$. If $w=\cdots\overline{3}$, the above argument again applies, so take $w=\cdots\overline{2}$. If $w$ contains no $\overline{3}$, the result is \Cref{lem:long_good_degens}(ii). Otherwise, $w = \cdots\overline{3}\overline{2}\overline{2}^k$ for some $k \in \mathbb{Z}_{\geq 0}$. Now apply the substitutions $14\cdots\overline{3}\overline{2}\overline{2}^k4 \to 14\cdots41\overline{2}^k4 \to 41\cdots41\overline{2}^k4 \to 41\cdots\overline{3}\overline{2}\overline{2}^k4$, where the first and third steps use \Cref{lem:long_good_degens} and the second step uses induction. While the plumbings from the two applications of \Cref{lem:long_good_degens} do not cancel, nonetheless their composite acts as the identity on the promotion appliances of $\overline{3}\overline{2}\overline{2}^k4$. Hence we may remove their composite from the structure entirely, completing the induction and proof.
\end{proof}

\subsection{Completeness of the growth rules}\label{sec:completeness}

We now show that any balanced oscillating lattice word $w$ with $r=4$ has at least one applicable growth rule. We will show in \Cref{cor:growth_no_cycles} that the growth algorithm terminates.

Recall the notation $\widetilde{i}$ from \Cref{def:grevlex}, which represents either $i$ or $\overline{r-i+1}$. Note that the involutions $\tau$ and $\varepsilon$ act on $\widetilde{w}$ by reversing and performing substitutions $\tilde{1} \leftrightarrow \tilde{4}$, $\tilde{2} \leftrightarrow \tilde{3}$. The composite $\varpi$ fixes $\widetilde{w}$. 

We write $\langle a_1,\ldots,a_k \rangle$ to denote the set of finite words in the $a_i$. Given a set $X$ of words and a letter $a$, we write $Xa$ (resp.\ $aX$) for the set of words $\{xa : x \in X\}$ (resp.\ $\{ax : x \in X\}$).

\begin{lemma}\label{lem:short_patterns}
  At least one of the short growth rules in \Cref{fig:growth-rules} applies whenever any of the following substrings appears in $\widetilde{w}$:
    \[ \tilde{1}\tilde{1}\tilde{4}, \tilde{1}\tilde{2}\tilde{3}, \tilde{1}\tilde{2}\tilde{4}, \tilde{1}\tilde{3}\tilde{2}, \tilde{1}\tilde{3}\tilde{3}, \tilde{1}\tilde{3}\tilde{4}, \tilde{1}\tilde{4}\tilde{4}, \tilde{2}\tilde{2}\tilde{4}, \tilde{2}\tilde{3}\tilde{4}, \tilde{3}\tilde{2}\tilde{4}.  \]
\end{lemma}
\begin{proof}
    This follows from inspecting \Cref{fig:growth-rules}. A computer implementation is available in \cite[\S6]{sl4-web-basis-code}.
\end{proof}

\begin{lemma}\label{lem:long_patterns}
  At least one of the growth rules in \Cref{fig:growth-rules} applies whenever any of the following substrings appears in $\widetilde{w}$:
    \[ \tilde{1}\tilde{4}x\tilde{4}, \quad \tilde{1}x\tilde{1}\tilde{4}, \quad \tilde{2}\tilde{3}x\tilde{4}, \quad \tilde{1}x\tilde{2}\tilde{3}, \]
  where $x\in \langle \tilde{2},\tilde{3}\rangle$. 
\end{lemma}

\begin{proof}
  The first two substrings and last two substrings are equivalent under $\tau$. For $\tilde{1}\tilde{4}x\tilde{4}$, apply $\varpi$ to consider $1\tilde{4}x\tilde{4}$. A growth rule applies to $1\overline{1}$, so we may assume we have $14x\tilde{4}$. If $x$ contains a $2$ or $3$, a long growth rule applies to the prefix ending at the leftmost $2$ or $3$ in $x$ with that $2$ or $3$ as the witness. Otherwise, a long growth rule still applies with the final $\tilde{4}$ as witness. The argument for $\tilde{2}\tilde{3}x\tilde{4}$ is essentially identical.
\end{proof}

\begin{lemma}\label{lem:some-growth-rule-applies}
  At least one of the growth rules in \Cref{fig:growth-rules} applies to any non-empty balanced lattice word $w$.
\end{lemma}

\begin{proof}
  Consider $\widetilde{w}$, which must begin with $\tilde{1}$ and end with $\tilde{4}$. Consider the rightmost $\tilde{1}$. It is followed by a letter other than $\tilde{1}$, so we consider the following three cases.
  \begin{itemize}
    \item $\cdots\tilde{1}\tilde{2}\cdots$: this must be $\cdots\tilde{1}\tilde{2}\tilde{2}^k\tilde{3}\cdots$ or $\cdots\tilde{1}\tilde{2}\tilde{2}^k\tilde{4}\cdots$. In the first case, apply \Cref{lem:long_patterns}. In the second case, we have $\cdots\tilde{1}\tilde{2}\tilde{4}\cdots$ or $\cdots\tilde{2}\tilde{2}\tilde{4}$, so apply \Cref{lem:short_patterns}.
    \item $\cdots\tilde{1}\tilde{3}\cdots$: this must be $\cdots\tilde{1}\tilde{3}\tilde{2}\cdots$, $\cdots\tilde{1}\tilde{3}\tilde{3}\cdots$, or $\tilde{1}\tilde{3}\tilde{4}\cdots$, so apply \Cref{lem:short_patterns}.
    \item $\cdots\tilde{1}\tilde{4}\cdots$: since $w$ is a balanced lattice word, it cannot end in $14$ or $\overline{4}\overline{1}$. If it ends in $1\overline{1}$ or $\overline{4}4$, a growth rule applies. Hence we may suppose $\tilde{w}$ does not end in $\tilde{1}\tilde{4}$, so we have $\cdots\tilde{1}\tilde{4}v\tilde{4}\cdots$ where $v\in \langle \tilde{2},\tilde{3}\rangle$, so apply \Cref{lem:long_patterns}. \qedhere
  \end{itemize}
\end{proof}

\subsection{Linearized diagrams, nice labelings, and unitriangularity}\label{sec:nicelabelings}

Recall the $\tlex$-order from \Cref{def:grevlex}. Our next goal is to show that the growth labeling $\Gamma(G)$ yields the unique grevlex-maximal monomial in $[G]_q$, or equivalently that the growth labeling is the unique proper labeling with $\tlex$-minimal boundary word.

We outline the approach now. First, we show that the growth rules in \Cref{fig:growth-rules} inductively preserve the unique $\tlex$-minimal labeling property. This requires using the labels around each $\mathsf{X}$  in its orientation coming from the growth labeling. The argument allows slightly more flexibility in which $\mathsf{X}$ labelings actually appear. As we have noted, not all fully reduced hourglass plabic graphs are obtained as the output of the growth rules. We show that one may perform a sort of \emph{surgery} to simulate benzene moves while still using allowed labelings around $\mathsf{X}$'s. Since square moves do not change $[G]_q$, this is sufficient to prove the unique $\tlex$-minimal labeling property for all fully reduced hourglass plabic graphs.

Precise details follow. We begin by considering diagrams similar to those appearing at the top of \Cref{fig:growth-example}.

\begin{definition}
  A \textit{linearized diagram} is a \symm six-vertex configuration obtained by beginning with a row of directed dangling strands and at each step either (i) combining an adjacent pair of strands in an $\sfX$, or (ii) combining an adjacent pair in a $\cup$, with all other strands extended directly down in either case, until there are no more dangling strands.

  A \textit{signed} proper labeling $\phi$ of a linearized diagram is a labeling of the edges by $\pm [4]$ whose absolute value is a proper labeling (see \Cref{def:ssv-proper-labeling}) and where the sign is positive for downward-pointing arrows and negative for upward-pointing arrows. The boundary $\partial(\phi)$ of $\phi$ is the word of oscillating type obtained by reading off the topmost labels from left to right.
  
  A \textit{nice labeling} of a linearized diagram is a signed proper labeling where the labels on the $\sfX$'s and $\cup$'s are those that appear in the $\sfX$'s and $\cup$'s of the growth rules in \Cref{fig:growth-rules}, plus those appearing in \Cref{fig:extra-nice}.
\end{definition}

\begin{figure}[tbh]
    \centering
    \includegraphics[scale=1.2]{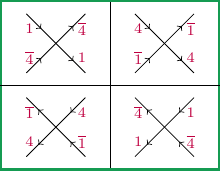}
    \caption{The additional directed $\sfX$ label configurations allowed in nice labelings.}
    \label{fig:extra-nice}
\end{figure}

\begin{remark}\label{rem:growth_linearized}
  The growth algorithm produces a linearized diagram $G$. The growth labeling $\Gamma(G)$ is a nice labeling without the configurations appearing in \Cref{fig:extra-nice} and with appropriate witnesses. In general, linearized diagrams with nice labelings need not have these features. The additional configurations in \Cref{fig:extra-nice} do not correspond to good degenerations, though we will find that they arise when applying Yang--Baxter moves to such $G$. Despite ignoring witnesses, linearized diagrams with nice labelings track more information than circular diagrams (such as in the lower-left of \Cref{fig:growth-example}), since they give a consistent set of orientations for each $\sfX$ and $\cup$. 
\end{remark}

The collection of possible nice labelings of $\sfX$'s possesses the following properties.

\begin{lemma}\label{lem:growth_lex}
  Suppose $ab \to cd$ is the portion of a good degeneration from the growth rules in \Cref{fig:growth-rules} corresponding to the $\sfX$, or similarly one of the additional labelings in \Cref{fig:extra-nice}. Then the following hold:
  \begin{enumerate}[(i)]
    \item Either $\tilde{a}=\tilde{b}=\tilde{c}=\tilde{d}$, or $\tilde{a} < \tilde{c}$ and $cd >_{\tlex} ab$.
    \item Fixing labels $ab$ at the top of the $\sfX$ and assigning labels $cd$ at the bottom is the $\tlex$-maximal way to create a signed proper labeling at these vertices.
    \item Given an $\sfX$ properly labeled by $pq$ at the top and $ts$ at the bottom, if $ts \geq_{\tlex} cd$, then $pq \geq_{\tlex} ab$.
    \item If $u, v$ are words of oscillating type such that $ucdv$ is a lattice word, then $uabv$ is a lattice word with the same weight.
  \end{enumerate}
\end{lemma}

\begin{proof}
  All properties may be checked case by case. For example, consider $ab = 2\overline{2} \to \overline{1}1 = cd$. For (i), $\tilde{a} = \tilde{2} < \tilde{4} = \tilde{c}$. For (ii), expanding the $\sfX$ in this case yields two internal vertices connected horizontally by an hourglass. Labeling the top strands with $2$'s forces any proper labeling to label both bottom strands with the same number. The $\tlex$-maximal way of doing so uses $1$'s at the bottom, as in the growth rule. For (iii), we must consider $t=\overline{1}$, $s=1, 2, 3, 4$. Using $s=1$, proper labelings can have $pq=2\overline{2}, 3\overline{3}, 4\overline{4}$. Each of these satisfies $pq \geq_{\tlex} 2\overline{2}$. All other cases are similar.
  
  For (iv), the mixed-sign cases like $2\overline{2} \to \overline{1}1$ in fact involve a pair of vertices that correspond under a crystal isomorphism. For the remaining cases like $12 \to 21$, the direction in (iv) matters. A computer implementation of these calculations is available at \cite[\S8]{sl4-web-basis-code}.
\end{proof}

\begin{remark}
    The complete list of $ab \to cd$ satisfying \Cref{lem:growth_lex} is precisely those appearing in \Cref{fig:growth-rules} together with the following:
    \begin{alignat*}{3}
        &1\overline{1} \to \overline{2}2\qquad\qquad
        && 1\overline{4} \to \overline{4}1
        \quad 4\overline{1} \to \overline{1}4\qquad\qquad
        14 \to \overline{2}\overline{3}
        \quad \overline{4}\overline{1} \to 32.
        \\
        &\overline{4}4 \to 3\overline{3}\qquad\qquad
        && \overline{1}4 \to 4\overline{1}
        \quad \overline{4}1 \to 1\overline{4}
    \end{alignat*}
    See \cite[\S8]{sl4-web-basis-code}. The middle four have already been mentioned in \Cref{fig:extra-nice}. The left two start with $1\overline{1}$ or $\overline{4}4$, which for our purposes always become end caps. The right two are mysterious. By an exhaustive check, when \Cref{lem:growth_lex}(i)--(iii) hold, condition (iv) always holds.
\end{remark}

Equality in \Cref{lem:growth_lex}(i) occurs precisely for the additional labelings in \Cref{fig:extra-nice}, which do not occur in growth labelings. The following important fact is then immediate. 
\begin{corollary}\label{cor:growth_no_cycles}
  If $w \to w'$ is obtained by applying a growth rule, then either $w'$ is shorter than $w$, or $w' >_{\tlex} w$. In particular, the growth algorithm cannot enter an infinite loop.
\end{corollary}

\begin{corollary}
    Let $G$ be a linearized diagram with a nice labeling $\psi$. Then $\partial(\psi)$ is a balanced lattice word.
\end{corollary}
\begin{proof}
    This follows from \Cref{lem:growth_lex}(iv) and induction.
\end{proof}

Linearized diagrams with nice labelings have the following key property. Note that it applies by \Cref{rem:growth_linearized} to the output $\mathcal{G}(T)$ of the growth algorithm using the growth labeling.

\begin{theorem}\label{thm:nice-unitriangularity}
  Let $G$ be a linearized diagram with a given nice labeling $\psi$ and suppose $\phi$ is a (signed) proper labeling of $G$. Then:
  \begin{enumerate}[(i)]
    \item $\partial(\phi) \geq_{\tlex} \partial(\psi)$, and
    \item if $\partial(\phi) = \partial(\psi)$, then $\phi = \psi$.
  \end{enumerate}
  In particular, linearized diagrams have at most one nice labeling.
\end{theorem}

\begin{proof}
  We induct on the number of steps in the linearized diagram. The base case is trivial. If an end cap is the topmost step in $G$, the result is obvious by induction. If instead an $\sfX$ is the topmost step in $G$, let $G'$ be the linearized diagram obtained from $G$ by removing this topmost step and let $\psi', \phi'$ be the restrictions of $\psi$ and $\phi$ to $G'$. Since $\psi'$ is a nice labeling, induction gives $\partial(\phi') \geq_{\tlex} \partial(\psi')$. By \Cref{lem:growth_lex}(a), $\partial(\psi') \geq_{\tlex} \partial(\psi)$, hence $\partial(\phi') \geq_{\tlex} \partial(\psi)$. Note that $\partial(\psi), \partial(\psi')$ and $\partial(\phi), \partial(\phi')$ differ only at the positions of the $\sfX$. Suppose that the $\sfX$ in $\psi$ is labeled by $ab$ at the top and $cd$ at the bottom. Further suppose that the $\sfX$ in $\phi$ is labeled by $pq$ at the top and $ts$ at the bottom, so that the corresponding portion of $\partial(\phi')$ is $ts$ and $\partial(\phi)$ is $pq$. If $\partial(\phi')$ does not agree with $\partial(\psi)$ up until the $\sfX$, then $\partial(\phi') \geq_{\tlex} \partial(\psi)$ implies $\partial(\phi) >_{\tlex} \partial(\psi)$ and we are done. So, suppose $\partial(\phi')$ agrees with $\partial(\psi)$ at least until the $\sfX$. Since $\partial(\phi') \geq_{\tlex} \partial(\psi')$, we have $ts \geq_{\tlex} cd$. By \Cref{lem:growth_lex}(iii), we have $pq \geq_{\tlex} ab$. If $pq \neq ab$, then $\partial(\phi) >_{\tlex} \partial(\psi)$, and (i) holds. If $pq = ab$, then by \Cref{lem:growth_lex}(ii), $cd \geq_{\tlex} ts$, so $cd = ts$. Now $\partial(\phi')$ and $\partial(\psi')$ agree through at least the $\sfX$, so $\partial(\phi') \geq_{\tlex} \partial(\psi')$ implies $\partial(\phi) \geq_{\tlex} \partial(\psi)$, giving (i).

  Now suppose $\partial(\phi) = \partial(\psi)$. In the notation above, this gives $ab=pq$ and so, as before, $cd=ts$. Thus $\partial(\phi') = \partial(\psi')$, so $\phi' = \psi'$ by induction, and hence $\phi = \psi$, giving (ii).
\end{proof}

Our next goal is to show that one may apply Yang--Baxter moves to linearized diagrams in such a way that nice labelings are preserved.

\begin{lemma}\label{lem:linearized-benzene}
  The possible oriented triangles in linearized diagrams are those configurations depicted in \Cref{fig:linearized-benzene}.
\end{lemma}

\begin{proof}
  The argument involves a straightforward, though tedious, case-by-case check. The details are omitted.
\end{proof}

\begin{figure}[tbh]
  \centering
  \includegraphics[width=\linewidth]{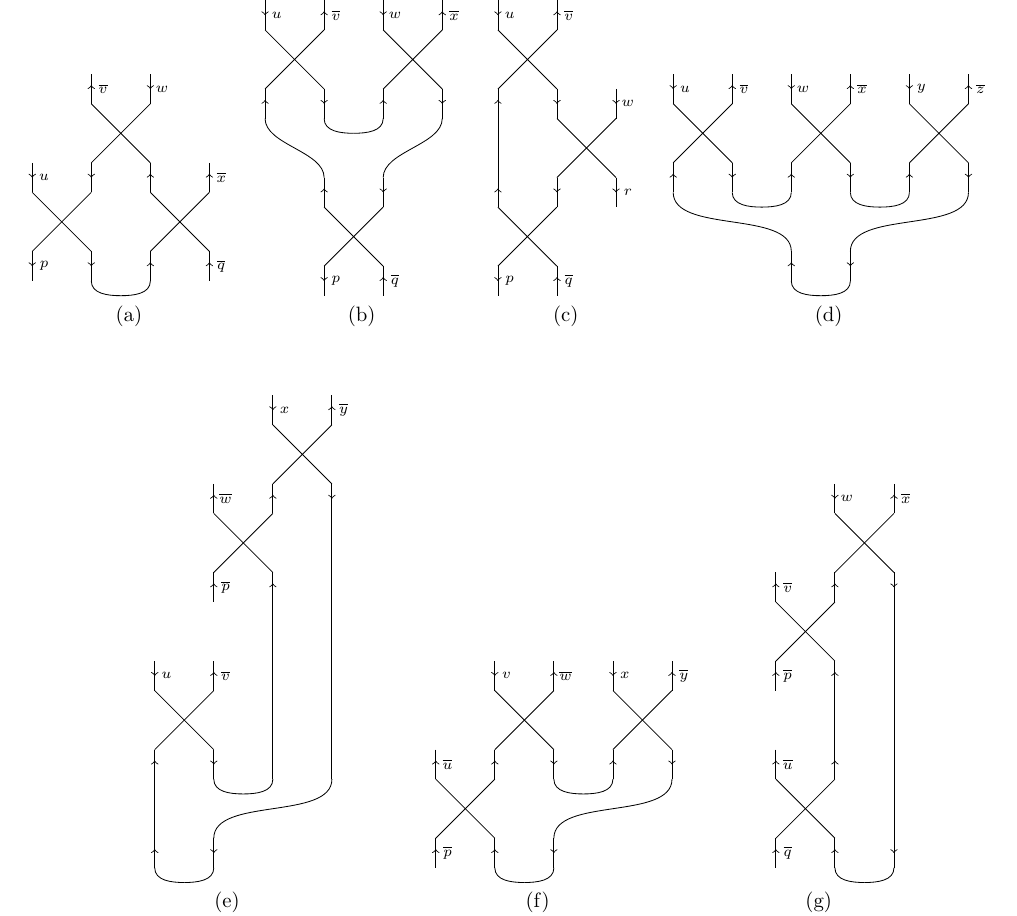}
  \caption{Configurations of oriented triangles in linearized diagrams. Those configurations obtained by reflecting through the vertical axis, reversing arrows, or both are also allowed, as are configurations obtained by shifting $\sfX$'s while preserving the linearized structure and $\cup$'s.}
  \label{fig:linearized-benzene}
\end{figure}

\begin{table}[tbh]
\centering
\begin{tabular}{c|c|c}
  Case(s)
    & Orientation(s)
    & $u, v, \ldots$ to $p, q, \ldots$ boundary conditions \\
    \toprule
  (a), (b)
    & CCW, CW
    & $\begin{aligned}
      1\overline{3}3\overline{1} \to 4\overline{4}
        && 1\overline{4}2\overline{1} \to 2\overline{4}
        && 1\overline{4}3\overline{1} \to 3\overline{4}
        && 1\overline{3}4\overline{1} \to 4\overline{3}
        && 1\overline{2}4\overline{1} \to 4\overline{2} \\
      \overline{4}2\overline{2}4 \to \overline{1}1
        && \overline{4}1\overline{3}4 \to \overline{3}1
        && \overline{4}1\overline{2}4 \to \overline{2}1
        && \overline{4}2\overline{1}4 \to \overline{1}2
        && \overline{4}3\overline{1}4 \to \overline{1}3
      \end{aligned}$ \\
    \midrule
  \multirow{2}{*}[-0.2em]{(b)}
    & CCW
    & $\begin{aligned}
      \overline{3}3\overline{1}4 \to \overline{1}4
        \qquad \overline{4}1\overline{3}3 \to \overline{4}1
      \end{aligned}$ \\
    \cmidrule{2-3}
  {}
    & CW
    & $\begin{aligned}
      2\overline{2}4\overline{1} \to 4\overline{1}
        \qquad 1\overline{4}2\overline{2} \to 1\overline{4}
      \end{aligned}$ \\
    \midrule
  (c)
    & CCW, CW
    & $\begin{aligned}
      \overline{4}3\overline{1} \to \overline{1}3\overline{4}
        && \overline{3}4\overline{1} \to \overline{1}4\overline{3}
        && \overline{3}3\overline{1} \to \overline{1}4\overline{4}
        && \overline{4}2\overline{1} \to \overline{1}2\overline{4} \\
      1\overline{3}4 \to 4\overline{3}1
        && 1\overline{4}3 \to 3\overline{4}1
        && 1\overline{3}3 \to 4\overline{4}1
        && 1\overline{2}4 \to 4\overline{2}1 \\
      \overline{4}2\overline{2} \to \overline{1}1\overline{4}
        && \overline{2}4\overline{1} \to \overline{1}4\overline{2}
        && \overline{4}1\overline{3} \to \overline{3}1\overline{4}
        && \overline{4}1\overline{2} \to \overline{2}1\overline{4} \\
      2\overline{2}4 \to 4\overline{1}1
        && 1\overline{4}2 \to 2\overline{4}1
        && 3\overline{1}4 \to 4\overline{1}3
        && 2\overline{1}4 \to 4\overline{1}2
      \end{aligned}$ \\
    \midrule
  \multirow{2}{*}[-0.2em]{(d)}
    & CCW
    & $\begin{aligned}
      \overline{4}1\overline{3}3\overline{1}4 \to \varnothing
      \end{aligned}$ \\
    \cmidrule{2-3}
  {}
    & CW
    & $\begin{aligned}
      1\overline{4}2\overline{2}4\overline{1} \to \varnothing
      \end{aligned}$ \\
    \midrule
  (e)
    & 
    & (None) \\
    \midrule
  \multirow{3}{*}[-2.1em]{(f)}
    & CCW, CW
    & $\begin{aligned}
      \overline{4}1\overline{2}4\overline{1} \to \overline{2}
        && \overline{4}1\overline{3}4\overline{1} \to \overline{3} \\
      1\overline{4}3\overline{1}4 \to 3
        && 1\overline{4}2\overline{1}4 \to 2
      \end{aligned}$ \\
    \cmidrule{2-3}
  {}
    & CCW
    & $\begin{aligned}
      1\overline{4}2\overline{2}4 \to 1 \\
      \overline{4}1\overline{3}3\overline{1} \to \overline{4}
      \end{aligned}$ \\
    \cmidrule{2-3}
  {}
    & CW
    & $\begin{aligned}
      \overline{4}2\overline{2}4\overline{1} \to \overline{1} \\
      1\overline{3}3\overline{1}4 \to 4
      \end{aligned}$ \\
    \midrule
  \multirow{2}{*}[-1.0em]{(g)}
    & CCW
    & $\begin{aligned}
      1\overline{4}24 \to 21
        && 1\overline{4}34 \to 31 \\
      \overline{4}1\overline{3}\overline{1} \to \overline{3}\overline{4}
        && \overline{4}1\overline{2}\overline{1} \to \overline{2}\overline{4}
      \end{aligned}$ \\
    \cmidrule{2-3}
  {}
    & CW
    & $\begin{aligned}
      \overline{4}\overline{2}4\overline{1} \to \overline{1}\overline{2}
        && \overline{4}\overline{3}4\overline{1} \to \overline{1}\overline{3} \\
      13\overline{1}4 \to 43
        && 12\overline{1}4 \to 42
      \end{aligned}$ \\
    \midrule
\end{tabular}
\caption{Possible boundary conditions of nice labelings of oriented triangles in \Cref{fig:linearized-benzene}.}
\label{tab:benzene-puzzles}
\end{table}

\begin{lemma}\label{lem:benzene}
  Suppose one of the oriented triangles in \Cref{fig:linearized-benzene} appears in a nice labeling of a linearized diagram. The complete list of possible boundary conditions is in \Cref{tab:benzene-puzzles}.
\end{lemma}

\begin{proof}
  Consider configuration (a) in \Cref{fig:linearized-benzene}, with boundary conditions $u\overline{v}w\overline{x} \to p\overline{q}$. To be a fragment of a nice labeling, the $\cup$ must be labeled $1$. Examining the allowed $\sfX$ nice labelings, we find $u=1=x$, the uppermost vertex is $\overline{v}w \to p\overline{q}$, and at least one of $p$ or $q$ is $4$. If $p=q=4$, we must have $v=w=3$, so $1\overline{3}3\overline{1} \to 4\overline{4}$ is a valid nice labeling boundary condition for this counterclockwise-oriented type (a) triangle. All other cases are similar and are omitted. The fundamental involutions provide symmetries which simplify the argument. Specifically, $\varepsilon, \varpi$ interchange clockwise and counterclockwise orientations by reflecting or reversing arrows, respectively, and their composite $\tau$ fixes clockwise and counterclockwise orientations.

  Now consider configuration (g) in \Cref{fig:linearized-benzene} as another example. Here the boundary conditions give $\overline{u} \overline{v} w \overline{x} \to \overline{p} \overline{q}$. From the orientation of the end cap at the lower right, we see that it must be labeled by $\overline{4}$ and $4$. Hence, the configuration at the upper right can only be nicely labeled with $w = 4$ and $\overline{x} = \overline{1}$, using one of the additional nice labelings of \Cref{fig:extra-nice}. Similarly, in the lower left corner, we must have $\overline{u} = \overline{4}$, while $\overline{q} \in \{ \overline{1}, \overline{2}, \overline{3}\}$. This value $\overline{q}$ must propagate up to the intersection with $\overline{p}$ and we see that $\overline{p} = \overline{1}$, so that $\overline{q} = \overline{1}$ is not in fact possible. Moreover $\overline{v} = \overline{q}$. Thus, the valid nice labeling boundary conditions are exactly $\overline{4}\overline{2}4\overline{1} \to \overline{1}\overline{2}$
        and $\overline{4}\overline{3}4\overline{1} \to \overline{1}\overline{3}$.  Again, the other orientations of this type are similar after applying symmetries.
The treatment of the remaining configurations of \Cref{fig:linearized-benzene} is analogous and omitted.
\end{proof}

\begin{remark}
  The need for the additional nice labelings from \Cref{fig:extra-nice} may be seen in \Cref{fig:new-nice-labelings}. On the left is a counterclockwise-oriented triangle with boundary conditions $1\overline{4}2\overline{1} \to 2\overline{4}$, which may arise in the output of the growth algorithm. On the right is a clockwise-oriented triangle with the same boundary conditions, which cannot arise as the output of the growth algorithm since it involves the highlighted additional nice labeling $1\overline{4} \to \overline{4}1$. Moreover, this is the only clockwise-oriented triangle with these boundary conditions. Hence in order to apply a Yang--Baxter move while staying within the set of linearized diagrams with nice labelings, we must allow the additional nice labeling. While the involutions $\varepsilon$ and $\varpi$ reverse triangle orientations, they also change boundary conditions.
\end{remark}

\begin{figure}[tbh]
  \centering
  \includegraphics{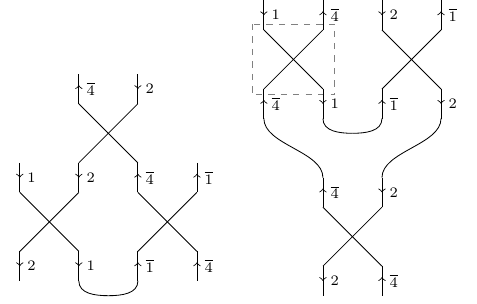}
  \caption{Linearized diagrams of oriented triangles with nice labelings and boundary conditions $1\overline{4}2\overline{1} \to 2\overline{4}$. The highlighted $\mathsf{X}$ involves the additional nice labeling from \Cref{fig:extra-nice}.}
  \label{fig:new-nice-labelings}
\end{figure}

\begin{lemma}\label{lem:nice-benzenes}
  Suppose $G$ and $G'$ are \symm six-vertex configurations related by a Yang--Baxter move. If $G$ can be written as a linearized diagram with a nice labeling, then the same is true of $G'$. Moreover, the nice labelings of $G$ and $G'$ are the same except inside the triangle at which the move was applied.
\end{lemma}

\begin{proof}
  By \Cref{lem:linearized-benzene}, every oriented triangle in a linearized diagram looks locally like one of the configurations in \Cref{fig:linearized-benzene}. It suffices to show that every such local configuration can be replaced by one with the same boundary conditions and opposite orientation. In \Cref{fig:new-nice-labelings}, we have a sample benzene triangle of type (a) with boundary conditions $1\overline{4}2\overline{1} \to 2\overline{4}$ and counterclockwise orientation, and a benzene triangle of type (b) with the same boundary conditions but clockwise orientation. A benzene move may be applied here by performing this surgery.
  
  Most of the possible boundary conditions listed in \Cref{tab:benzene-puzzles} are symmetric and do not require structural changes outside of the configurations (a)--(g). A sample exception is the clockwise-oriented configuration $1\overline{4}2\overline{2}4\overline{1} \to \varnothing$ of type (d). This may be replaced by the counterclockwise-oriented configuration on the left in \Cref{fig:linearized-tilting}, which arises from appending end caps to a configuration of type (c). Exceptions of type (a) or (b) (resp.\ (f)) may be handled similarly using this type (c) configuration but using only the rightmost end cap (resp.\ the left and right end caps). Exceptions of type (g) may be handled with the configuration on the right of \Cref{fig:linearized-tilting}.
\end{proof}

\begin{figure}[tbh]
  \centering
  \includegraphics{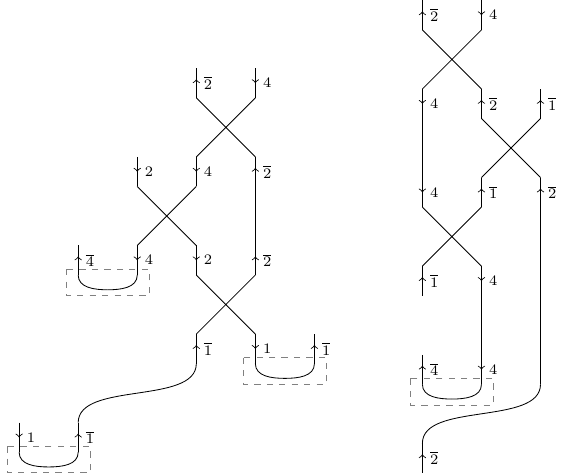}
  \caption{Appending end caps to oriented triangles in linearized diagrams.}
  \label{fig:linearized-tilting}
\end{figure}

\begin{remark}
    The surgery procedure in the proof of \Cref{lem:nice-benzenes} simulates benzene moves for linearized diagrams with nice labelings. One may wonder if the same can be done for square moves. The answer is no: only one of the two elements of the move-equivalence class of the square associated to $11223344$ has a nice labeling, namely the one obtained by applying the growth rules. Recall the good degeneration $2233 \to 2323$ from \Cref{rem:complete-redundant}. It labels the $\mathsf{X}$ with $23 \to 32$, which does not satisfy \Cref{lem:growth_lex}(iii). If we were to add this rule to \Cref{fig:growth-rules}, we would not be able to use our inductive argument to prove the unique $\tlex$-minimal boundary labeling property.
\end{remark}

\subsection{Growth rules and descents}\label{sec:descents}

Here, we define descents of fluctuating tableaux and show how they appear in the growth algorithm.

\begin{definition}\label{def:Des}
  Let $w$ be a lattice word on $\mathcal{A}_r$. Say $\osc(w) = w_1 \cdots w_n$. The \emph{descent set} of $w$ is
    \[ \Des(w) = \{i \in [n-1] : 0 < w_i < w_{i+1} \text{ or } w_i < w_{i+1} < 0\}. \]
\end{definition}

\begin{lemma}\label{lem:descents-are-connected}
  Let $G$ be an hourglass plabic graph obtained by repeatedly applying growth rules starting from a balanced oscillating lattice word $w = w_1 \cdots w_n$. If $i \in \Des(w)$, then the boundary vertices $b_i$ and $b_{i+1}$ are connected to a common vertex in $G$.
\end{lemma}

\begin{proof}
  Consider a descent $w_i w_{i+1} = 14$. Growth rules with $14 \to 41$ have the required property. None of the growth rules can change only the $4$ or only the $1$, so $14 \to 41$ must eventually be applied.
  
  Next consider a descent $w_i w_{i+1}=13$. Growth rules with $13 \to 31$ or $13 \to \overline{2}\overline{4}$ have the required property. Other growth rules may only change the $3$, namely $3\overline{1} \to \overline{1}3$, $34 \to \overline{1}\overline{2}$, $34 \to 43$. The first two will add an $\mathsf{X}$ and result in $13\cdots \to 1\overline{1}\cdots$, and no further growth rules can change the $1 \overline{1}$ until an end cap is eventually applied; hence the boundary vertices will both be connected to the center of the $\mathsf{X}$ in the symmetrized six-vertex configuration and this remains true in the hourglass plabic graph. The third would require $134 \to 143$, but that rule requires witnesses $3$ or $\overline{4}$ rather than $1$, so it does not apply.

  For a descent $w_i w_{i+1}=12$, we similarly find the growth rules which may change the $2$ are $2\overline{1} \to \overline{1}2$, $2\overline{2} \to \overline{1}1$, $23 \to \overline{1}\overline{4}$, $24 \to \overline{1}\overline{3}$, and $24 \to 42$. All except the last result in a $1\overline{1}$, allow us to conclude as in the previous case. For the last case, we only have a growth rule $24 \to 42$ with a left witness from $\{2,3,\overline{4}\}$, so in particular this cannot be applied here since we have a $1$ preceding the $2$.

  For a descent $23$, the growth rules which can change only the $2$ involve $\overline{4}2 \to 2\overline{4}$, $12 \to 21$, and $12 \to \overline{3}\overline{4}$, but none of them allow $3$ as a witness on the right. Similarly the growth rules which can change only the $3$ involve $3\overline{1} \to \overline{1}3$, $34 \to 43$, and $34 \to \overline{1}\overline{2}$, none of which allow $2$ as a witness on the left.

  Applying the involutions covers the remaining cases.
\end{proof}

\subsection{Growth rules and fully reduced graphs}\label{sec:growth-reduced}

We now show that the output $\mathcal{G}(T)$ of the growth algorithm is fully reduced. By \Cref{cor:fully-reduced-iff-monotonic}, we may equivalently show that $\mathcal{G}(T)$ is monotonic (see \Cref{def:hourglass-monotonicity}). Our argument is inductive and verifies,  by further analyzing crystal appliances, that growth rules preserve monotonicity. We begin with the significantly simpler $\trip_2$ conditions.

\begin{definition}
  Let $\{a, c\}, \{b, d\}$ be four distinct numbers in $[n]$. By relabeling if necessary, we may suppose $a < c$, $b < d$, and $a < b$. We say they form a \emph{crossing} if $a < b < c < d$. 
\end{definition}

\begin{lemma}\label{lem:growth-rules-trip2}
Consider a growth rule $v \to v'$ from \Cref{fig:growth-rules} applied to an oscillating word $w$ to obtain a balanced lattice word $w'$. Suppose that $G'$ is an hourglass plabic graph of the same boundary type as $w'$ and with $\trip_\bullet(G') = \prom_\bullet(w')$. Suppose further that $G$ is obtained by gluing the diagram for the growth rule onto $G'$. If the $\trip_2$-strands in $G'$ have no self-intersections or double crossings, then the same is true of $G$.
\end{lemma}
\begin{proof}
  First assume $v\to v'$ is a short rule. Consider self-intersections of $\trip_2$-strands in $G$. Since $G'$ has no $\trip_2$ self-intersections, the only way to introduce a self-intersection by attaching an $\sfX$ is for a $\trip_2$-strand to leave and re-enter the $\sfX$ through the bottom two edges. That is, $\trip_2(G')(a') = b'$ where $a' < b'$ are the consecutive positions corresponding to the $\sfX$ in $w'$. Since $\trip_\bullet(G') = \prom_\bullet(w')$, we have $\prom_2(w')(a') = b'$. Let $\rho_\bullet=\rho_\bullet(\vec{e}(v',w'), v', \vec{f}(v',w'))$; then $\rho_2(a) = b$, where $a-|\vec{e}|$ (respectively, $b-|\vec{e}|$) is the index of the left (respectively, right) part of the $\sfX$ in $v'$. Graphically, $\rho_2$ would need an arc between the bottom vertices of the $\sfX$. We see in \Cref{ex:good_degen} that this does not occur, and one may check directly that the same is true for all short rules, so $G$ does not contain $\trip_2$-strand self-intersections in these cases. A computer implementation of these verifications and the ones below is available at \cite[\S7]{sl4-web-basis-code}.

  Now, consider double crossings of $\trip_2$-strands in $G$. Since $G'$ has no $\trip_2$-strand double crossings, they arise in $G$ only from the two $\trip_2$-strands leaving the bottom two edges of the $\sfX$ and crossing in $G'$. This configuration occurs if and only if $\{a', \trip_2(G')(a')\}, \{b', \trip_2(G')(b')\}$ form a crossing. As before, such a crossing would be apparent as a crossing in $\rho_2$; we see this does not occur in \Cref{ex:good_degen}, and we may check it does not occur for any short rule for any crystal appliance for $v'$. Hence, $G$ does not contain $\trip_2$-strand double crossings in these cases. The $\trip_2$ conditions may be summarized as requiring that $\rho_2(a) \neq b$ and that $\{a, \rho_2(a)\}$, $\{b, \rho_2(b)\}$ do not form a crossing.

  Finally, consider the long rules. By using the $\tau, \varpi, \epsilon$ involutions and \Cref{lem:degen-symmetries}, it suffices to consider the rules of the form $14zx \to 41zx$ and $\overline{3}\overline{2}zx \to 41zx$ where $z \in \langle \overline{2}, \overline{3} \rangle$ and $x \in \{4, 3, 2, \overline{1}\}$.

  Let $\rho_\bullet$ be any crystal appliance of $41zx$. Recall that $\rho_s(i)$ records when the operator $e_s$ or $f_s$ is applied at the position corresponding to $i$ during the construction of the crystal appliance. For instance, $\rho_s(i) < i$ if this occurs during phases I or III and $\rho_s(i) > i$ if this occurs during phases II or IV. One can read off the action of $e_s$ during phase I from finite state machines like \Cref{fig:long-degens-DFA-3p2p2p4}. We note that $\rho_s(a) < a$ for $s=1, 2, 3$, since $v' = 41\cdots$ begins with $4$, which must be turned into a $1$ during phase I. Similarly, $\rho_s(b) > b$. Hence $\{a, \rho_s(a)\}, \{b, \rho_{s'}(b)\}$ do not cross for any $s, s' \in \{1, 2, 3\}$, in particular for $s=s'=2$.
\end{proof}

We give a similar but more technical argument for the $\trip_1$, $\trip_2$ monotonicity conditions. In order to rule out violations of the $\trip_1, \trip_2$ conditions in \Cref{def:hourglass-monotonicity}, we introduce the following notion for appliances which intuitively corresponds to crossings of trip strands. We then verify that the growth rules never introduce new violations of the $\trip_1, \trip_2$ conditions by examining all possible crystal appliances for growth rules. The most interesting case is exemplified by $\overline{3}\overline{2}\overline{2}\overline{3}4 \to 41\overline{2}\overline{3}4$, where configurations as in \Cref{fig:monotonicity-preserved} (right) must be ruled out.

\begin{definition}\label{def:intersecting}
  Let $a \neq b$. We say $(a, x)$ \textit{intersects} $(b, y)$ if
  \begin{enumerate}[(i)]
      \item $x = b$; or
      \item $y = a$; or
      \item $x = y$; or
      \item $a, b, x, y$ are  all distinct and $\{a, x\}$, $\{b, y\}$ form a crossing.
  \end{enumerate}
  Further, for $g_\bullet$ an appliance, or a $\prom_\bullet$, or a $\trip_\bullet$, we say $g_i(a)$ \textit{intersects} $g_j(b)$ if $(a, g_i(a))$ intersects $(b, g_j(b))$.
\end{definition}

\begin{lemma}\label{lem:touching-rho}
  Let $uvw$ be a balanced oscillating lattice word. Let $a \neq b$ be indices from $1,\ldots,|v|$. If $\prom_i(uvw)(|u|+a)$ intersects $\prom_j(uvw)(|u|+b)$, then there is a crystal appliance $\rho_\bullet = \rho_\bullet(\vec{e}, v, \vec{f})$ such that $\rho_{r-i}(|\vec{e}|+a)$ intersects $\rho_{r-j}(|\vec{e}|+b)$.
\end{lemma}

\begin{proof}
  Assume without loss of generality that $a<b$. Let $a' = |u|+a$ and $b' = |u|+b$. If $a'$, $\prom_i(uvw)(a')$, $b'$, and $\prom_j(uvw)(b')$ are all distinct, by assumption $\{a', \prom_i(uvw)(a')\}$, $\{b', \prom_j(uvw)(b')\}$ form a crossing. Hence the crystal appliance $\rho_\bullet=\rho_\bullet (\vec{e}(v,uvw),v,\vec{f}(v,uvw))$ from \Cref{lem:crystal-phases} has a crossing by the monotonicity of associated function $h$, and hence an intersection.

  Now suppose $a', \prom_i(uvw)(a'), b'$, and $\prom_j(uvw)(b')$ are not all distinct. Since $\prom_\bullet(uvw)$ is fixed-point-free, we must have $\prom_i(uvw)(a') = b'$, or $\prom_j(uvw)(b') = a'$, or $\prom_i(uvw)(a') = \prom_j(uvw)(b') = k$ for some $k$. In the first two cases, $\rho_\bullet$ again has a crossing, since $a'$ and $b'$ are indices corresponding to $v$ and the relevant part of the crystal appliance is defined in phase II or IV. The same applies in the third case if $k$ is an index corresponding to $v$ in $uvw$. Hence we may suppose $k$ is an index corresponding to $u$ or $w$. Suppose $k > |u|+|v|$ and $(uvw)_k < 0$, the other cases being similar. If $i=j$ then we would have $a'=b'$, but we have assumed $a < b$. By \Cref{prop:prom-cyclically-monotonic} applied to $\prom_\bullet(uvw)(k)$, we see $i < j$.
  
  Now let $y$ be obtained by replacing the $k$-th letter of $uvw$ with its standardized complement $\kappa$, e.g.~$\overline{3}$ would be replaced with $124\cdots r$, which is an $\SL_r$-Knuth equivalence. Since $\prom_i(uvw)(a') = k$, we have $\prom_i(y)(a') \in [k, k+r-1)$. That is, among columns $[k, k+r-1)$ of the promotion matrix $\PM_i(y)$, there is a $1$ in row $a'$. Note that $a' < k$, so this $1$ is not among rows $[k, k+r-1)$. We claim that $\prom_i(y)(a') = k+i-1$, i.e.~the $1$ in row $a'$ is in column $k+i-1$. To see this, consider the $(r-1) \times (r-1)$ submatrix $D_i$ of $\PM_i(y)$ with rows and columns $[k, k+r-1)$. We find directly that $\kappa^\uparrow = 123\cdots(r-1)$. Using \Cref{lem:crystal-phases}, it is not difficult to see that $D_i$ is the first $r-1$ rows and columns of $\PM_i(\kappa^\uparrow r)$. Moreover, $\PM_i(123\cdots r)$ is the $r \times r$ matrix with $1$'s along the $i$-th super-diagonal and $(r-i)$-th sub-diagonal and $0$'s elsewhere. Hence $D_i$ is obtained by deleting the final row and column from $\PM_i(123\cdots r)$, so $D_i$ is missing a $1$ precisely in the $i$-th column. This missing $1$ corresponds precisely to the $1$ in row $a'$ of $\PM_i(y)$, so it must be in column $k+i-1$ as claimed.

  Repeating this argument for $j$, we now have $\prom_i(y)(a') = k+i-1$ and $\prom_j(y)(b') = k+j-1$. Since $i < j$, we find $a' < b' < k \leq \prom_i(y)(a') < \prom_j(y)(b')$, so $\{a', \prom_i(y)(a')\}$ and $\{b', \prom_j(y)(b')\}$ form a crossing. Letting $\rho'_\bullet = \rho_\bullet(\vec{e}(v,y), v, \vec{f}(v,y))$, by monotonicity of the associated function $h'$, we have that $\rho'_{r-i}(|\vec{e}|+a)$ intersects $\rho'_{r-j}(|\vec{e}|+b)$.
\end{proof}

\begin{lemma}\label{lem:star-witness}
Let $v = 41zx$ where $z\in \langle \overline{2},\overline{3} \rangle$ and $x\in\{4,3,2,\overline{1}\}$. In a crystal appliance $\rho_\bullet(\vec{e},v,\vec{f})$, we have $\rho_1(\vec{e},v,\vec{f})(a+2)=a+b$, where $a = |\vec{e}|$ and $b = |v|$.
\end{lemma}
\begin{proof}
If $x = \overline{1}$, we have a crystal isomorphism which sends $41z\overline{1}$ to $41z234$; when applying $e_3$ to this word, the final $4$ is matched with the preceding $3$, and when applying $e_2$, this $3$ is matched with the preceding $2$. Thus, by considering the prefix $41z2$, we reduce to the case $x=2$ . 

Now suppose $x \in \{4, 3, 2\}$. This guarantees that \Cref{lem:long-degens-DFA-412p3p4} applies to $v$; hence the highest weight element in $\mathcal{C}(v)$ is of the form $v^\uparrow = 11w' 2$ where $w' \in \langle \overline{3},\overline{4} \rangle$ is a lattice word. As in \Cref{def:crystal-appliances}(3), let $y$ be such that $v^\uparrow y$ is a balanced lattice word. It is easy to see that $\promotion(v^\uparrow y) = 1w' 2y'$ and $\promotion^2(v^\uparrow y) = w' 1y''$ for some $y'$ and $y''$. In particular this gives $\prom_2(v^\uparrow y)(b)=2$ since the $2$, located in position $b$, changed to $1$ in the second application of promotion, which proves the claim.
\end{proof}

\begin{figure}[htbp]
    \centering
    \noindent
    \begin{minipage}{0.55\textwidth}
    \includegraphics[width=\textwidth]{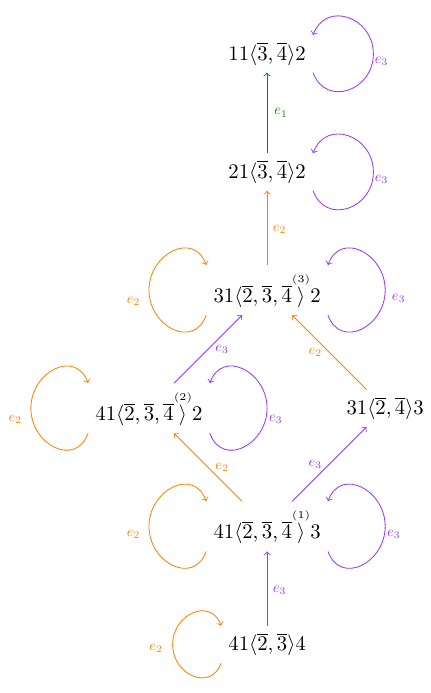}
    \end{minipage}
    \begin{minipage}{0.4\textwidth}
    \begin{itemize}
      \item[\textsuperscript{(1)}] No $\overline{4} \cdots \overline{3}$ subsequence.      
      \item[\textsuperscript{(2)}] No $\overline{4} \cdots \overline{3} \cdots \overline{2}$ subsequence.
      \item[\textsuperscript{(3)}] No $\overline{3} \cdots \overline{2}$ subsequence.
    \end{itemize}
    \end{minipage}
    \caption{The finite state machine, from \Cref{lem:long-degens-DFA-412p3p4}, for describing the part of $\mathcal{C}(w)$ reachable from $w \in 41\langle\overline{2}, \overline{3}\rangle4$ 
    by applying raising operators.}
    \label{fig:long-degens-DFA-412p3p4}
\end{figure}

\begin{lemma}\label{lem:growth-rules-monotonic-mostly}
  Under the hypotheses of \Cref{lem:growth-rules-trip2}, if $G'$ is monotonic, then $G$ is monotonic.
\end{lemma}

\begin{proof}
  By \Cref{lem:growth-rules-trip2}, we know that the $\trip_2, \trip_2$ condition from \Cref{def:hourglass-monotonicity} holds. It remains to show that the $\trip_1, \trip_2$ condition also holds. As before, a violation of this condition in $G$ must involve one of the $\trip_1$-strands in the additional plumbing through the attached $\mathsf{X}$. We now provide sufficient conditions on the level of crystal appliances of $v'$ to rule out such violations.

  \begin{claim}
  Suppose for all crystal appliances $\rho_\bullet = \rho_\bullet(\vec{e}, v', \vec{f})$ the following conditions hold.  Here $a=|\vec{e}|+1, b=|\vec{e}|+2$ and $\pi_\bullet$ is the associated plumbing of the growth rule $v \to v'$.
  \begin{enumerate}[(A)]
    \item If $\pi_\bullet = \mathsf{X}^{+-}_{-+} \times \id^k$ for $k \geq 0$, none of the following pairs intersect: $\rho_1(b), \rho_2(a)$; $\rho_3(a), \rho_2(b)$; $\rho_3(b), \rho_2(a)$; $\rho_1(a), \rho_2(b)$.
    \item If $\pi_\bullet = \mathsf{X}^{--}_{++} \times \id^{k+1}$ or $\pi_\bullet = \mathsf{X}^{++}_{++} \times \id^{k+1}$  for $k \geq 0$, none of the following pairs intersect: $\rho_1(b), \rho_2(a)$; $\rho_3(a), \rho_2(b)$; $\rho_3(b), \rho_2(a)$. Furthermore, $\rho_1(c) = b$ for some $c > b$ in $v'$, and for all $t \in [b, c]$, $\rho_1(a)$ and $\rho_2(t)$ do not intersect.
  \end{enumerate}
Then the $\trip_1, \trip_2$ condition also holds for $G$, so $G$ is monotonic.
\end{claim}

\begin{proof}[Proof of Claim]
  Suppose there were a violation of the $\trip_1$, $\trip_2$ condition in $G$. If the $\trip_1$-strand does not involve the $\mathsf{X}$, then we would have a $\trip_1$, $\trip_2$ condition violation in $G'$. Consider $\pi_\bullet = \mathsf{X}^{+-}_{-+} \times \id^k$. There are four $\trip_1$-strands through the $\mathsf{X}$. Let $a'$ and $b'$ be the indices in $w'$ corresponding to the left and right part of the $\mathsf{X}$, respectively. Consider the strand $\trip_1(G)(b') = \trip_1(G')(a')$ as in \Cref{fig:monotonicity-preserved} (left). If this strand has a violation with some $\trip_2$-strand, it must be $\trip_2(G)(a') = \trip_2(G')(b')$ since otherwise it was a violation in $G'$. Since $\trip_1(G)(b')$ and $\trip_2(G)(a')$ cross in the $\mathsf{X}$ and $G'$ is monotonic, the violation must come from $\trip_1(G')(a')$ and $\trip_2(G')(b')$ intersecting in $G'$. Hence $\prom_1(w')(a')$ intersects $\prom_2(w')(b')$, so by \Cref{lem:touching-rho} there is a crystal appliance $\rho_\bullet = \rho_\bullet(\vec{e}, v', \vec{f})$ such that $\rho_3(a)$ intersects $\rho_2(b)$, contradicting (A). The remaining conditions for (A) are similar.

  The argument for the first three conditions in (B) is also similar. The remaining conditions arise from examining the $\trip_1$-strand that enters and exits the bottom of $\mathsf{X}$ at $a'$ and $b'$, respectively. A violation with a $\trip_2$-strand either is a violation in $G'$ or arises from crossing both $\trip_3(G')(a')$ and $\trip_1(G')(b')$. Since $\rho_1(c) = b$ by assumption, $\trip_1(G')(b') = c'$ where $c'-b' = c-b$. If a $\trip_2$-strand crosses $\trip_1(G')(b')$, it must have an endpoint $t' \in [b', c']$ by monotonicity in $G'$. If it additionally crosses $\trip_3(G')(a')$ as in \Cref{fig:monotonicity-preserved} (right), then as before we have an intersection of $\rho_1(a)$ and $\rho_2(t)$ for some crystal appliance $\rho_\bullet=\rho_\bullet(\vec{e},v',\vec{f})$ of $v'$, where $t'-a' = t-a$.
\end{proof}

    \begin{figure}[tbh]
        \centering
        \includegraphics[width=0.8\textwidth]{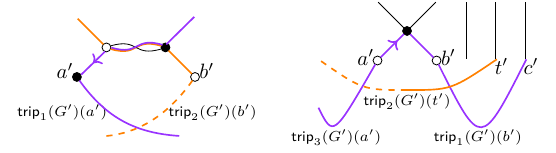}
        \caption{Forbidden configurations in the proof of the claim in \Cref{lem:growth-rules-monotonic-mostly} illustrating conditions (A) [left] and (B) [right].}
        \label{fig:monotonicity-preserved}
    \end{figure}

  We now return to the proof of the lemma. By using the $\tau, \varpi, \epsilon$ involutions and \Cref{lem:degen-symmetries}, it suffices to check the conditions in the claim for growth rules such that $\pi_\bullet$ is of the form in (A) or (B). Note that under the conditions of the claim, $G'$ is monotonic. The conditions (A) or (B) may be checked directly for all the possible crystal appliances of the short rules in \Cref{fig:growth-rules}. A computer implementation of these verifications is available at \cite[\S7]{sl4-web-basis-code}.

  It again remains to consider the long rules (which never satisfy the hypotheses of (A)).  As in the proof of \Cref{lem:growth-rules-trip2}, it suffices to consider those which apply $\mathsf{X}^{++}_{++}$ or $\mathsf{X}^{--}_{++}$ with target $v'$ of the form $41wx$, where $w \in \langle \overline{2}, \overline{3} \rangle$ and $x \in \{4, 3, 2, \overline{1}\}$. Let $\rho_\bullet$ be any crystal appliance $\rho_\bullet(\vec{e}, v',\vec{f})$. By the claim, all that remains is to verify condition (B) for $\rho_\bullet$. If $x = \overline{1}$, we have a crystal isomorphism sending $41w\overline{1}$ to $41w234$, and we may apply the $x=2$ case on the prefix, which appropriately preserves $\rho_\bullet$ for the purposes of condition (B). Thus we take $x \in \{4, 3, 2\}$. As in the proof of \Cref{lem:growth-rules-trip2}, we have $\rho_s(a) < a$ and $\rho_s(b) > b$ for $s \in \{1, 2, 3\}$. Hence none of $\{\rho_3(a), \rho_2(b)\}$, $\{\rho_1(b), \rho_2(a)\}$, and $\{\rho_3(b), \rho_2(a)\}$ intersect.
  
  From \Cref{lem:star-witness}, $\rho_1(c) = b$, where $c$ corresponds to the position of $x$. For the conditions on $t \in [b, c]$ in (B), note that $\rho_2(a) < \rho_1(a)$ since $e_1$  must be applied to the first position of $41wx$ after $e_2$ is. Hence $\rho_2(a) < \rho_1(a) < a < t$. We now further claim that either $\rho_2(t) < \rho_2(a)$ or $\rho_2(t) > a$. Indeed, by \Cref{lem:long-degens-DFA-412p3p4} below, we may read off when $e_2$ is applied to the first position when using crystal raising operators starting from $41wx$ from the finite state machine in \Cref{fig:long-degens-DFA-412p3p4}. By inspection, this occurs only when we obtain a word in $21\langle \overline{3},\overline{4} \rangle 2$ and no further applications of $e_2$ occur afterwards. Hence for $t$ at positions where $e_2$ is applied during phase I, we have $\rho_2(t) < \rho_2(a)$. For $t$ at positions where $e_2$ is not applied during phase I, $e_2$ must be applied during a later phase, so $\rho_2(t) > a$.

  Finally, if $\rho_2(t) < \rho_2(a)$, then $\rho_2(t) < \rho_1(a) < a < t$, and $\rho_1(a)$ does not intersect $\rho_2(t)$. Hence we may suppose $\rho_2(t) > a$. We have $\rho_2(t) \neq t$, so either $\rho_2(t) > t$ or $\rho_2(t) < t$. In the first case, $\rho_1(a) < a < t < \rho_2(t)$, while in the second case, $\rho_1(a) < a < \rho_2(t) < t$, and again $\rho_1(a)$ does not intersect $\rho_2(t)$.
\end{proof}

All that remains is to prove the correctness of the finite state machine in \Cref{fig:long-degens-DFA-412p3p4}. 

\begin{lemma}\label{lem:long-degens-DFA-412p3p4}
Let $v$ be an element of one of the sets of words labeling the vertices of \Cref{fig:long-degens-DFA-412p3p4}. Then $v^\uparrow \in 11\langle \overline{3}, \overline{4} \rangle2$ and any upward path in $\mathcal{C}(v)$ from $v$ to $v^{\uparrow}$ corresponds to a walk in the figure.
\end{lemma}

\begin{proof}
  We verify the correctness of the diagram in \Cref{fig:long-degens-DFA-412p3p4} using the bracketing rule (see \Cref{def:bracketing-rule}) step by step. Define sets
  \begin{alignat*}{5}\
    A &= \langle\overline{2}, \overline{3}\rangle &&\qquad
    B = \langle\overline{2}, \overline{3}, \overline{4}\overset{\scriptscriptstyle{(1)}}{\rangle} &&\qquad
    C = \langle\overline{2}, \overline{3}, \overline{4}\overset{\scriptscriptstyle{(2)}}{\rangle} \\
    D &= \langle\overline{2}, \overline{4}\rangle &&\qquad
    E = \langle\overline{2}, \overline{3}, \overline{4}\overset{\scriptscriptstyle{(3)}}{\rangle} &&\qquad
    F = \langle\overline{3}, \overline{4}\rangle.
  \end{alignat*}
As indicated in \Cref{fig:long-degens-DFA-412p3p4}, the superscripts $^{\scriptscriptstyle{(1)}}$, $^{\scriptscriptstyle{(2)}}$, and $^{\scriptscriptstyle{(3)}}$ have the following meanings: $B$ consists of all words in $\langle\overline{2}, \overline{3}, \overline{4}\rangle$ where there is no subsequence $\overline{4}\cdots\overline{3}$; $C$ consists of all words in $\langle\overline{2}, \overline{3}, \overline{4}\rangle$ where there is no subsequence $\overline{4}\cdots\overline{3}\cdots\overline{2}$; and $E$ consists of all words in $\langle\overline{2}, \overline{3}, \overline{4}\rangle$ where there is no subsequence $\overline{3}\cdots\overline{2}$.
  
\begin{claim} 
If a crystal raising operator applies to a word in $C$ or $E$, then the result is again in $C$ or $E$, respectively. Furthermore, if a crystal raising operator applies to a word in $B3$, then the result is again in $B3$ or in $B2 \subset C2$.
\end{claim}

\begin{proof}[Proof of Claim]
No $e_1$ operators apply to $B$, $C$ or $E$ since there are no $\overline{1}$'s or $2$'s.

Now consider applying $e_2$ to $w \in C$. To apply $e_2$ here requires matching $\overline{3}, \overline{2}$ pairs and changing the rightmost unmatched $\overline{2}$ into a $\overline{3}$. Suppose to the contrary $e_2(w)$ has a subsequence $\overline{4}\cdots\overline{3}\cdots\overline{2}$ at indices $i < j < k$. The $e_2$ then must have been applied at $w_j = \overline{2}$, so $w_j$ is unmatched and $w_k=\overline{2}$ is matched. However, this implies there is some $t$ with $j < t < k$ such that $w_t=\overline{3}$, contradicting $w \in C$. 

Consider applying $e_3$ to $w \in C$ giving a forbidden subsequence $\overline{4}\cdots\overline{3}\cdots\overline{2}$ at indices $i < j < k$. The $e_3$ then must have been applied at $w_i=\overline{3}$. As $w_i$ is the rightmost unmatched $\overline{3}$ in $w$, there is some $t$ with $i<t<j$ such that $w_t=\overline{4}$, contradicting $w \in C$.

Next consider applying $e_2$ to $w \in E$ giving a forbidden subsequence  $\overline{3}\cdots\overline{2}$ at indices $i < j$. The $e_2$ then must have been applied at $w_i=\overline{2}$. As $w_i$ is the rightmost unmatched $\overline{2}$ in $w$, there is some $t$ with $i<t<j$ such that $w_t=\overline{3}$, contradicting $w \in E$.

Applying $e_3$ to $w \in E$ never gives a forbidden subsequence $\overline{3}\cdots\overline{2}$ since it must change a $\overline{3}$ to $\overline{4}$.

Finally, consider $w \in B3$. Applying $e_2$ either changes the $3$ to $2$, giving an element of $B2$, or it applies inside $B$. In the latter case $e_2$ changes some $w_j=\overline{2}$ to $\overline{3}$. Assume this results in a subsequence $\overline{4}\cdots \overline{3}\cdots 3$ at indices $i < j < k$. As $w_j$ is the rightmost unmatched $\overline{2}$ or $3$ in $w$, there is some $t$ with $j<t<k$ such that $w_t=\overline{3}$, contradicting $w\in B3$.
Assume applying $e_3$ to $w \in B3$ gives a forbidden subsequence $\overline{4}\cdots\overline{3}$ at indices $i < j$. The $e_3$ then must have been applied at $w_i=\overline{3}$. As $w_i$ is the rightmost unmatched $\overline{3}$ in $w$, there is some $t$ with $i<t<j$ such that $w_t=\overline{4}$, contradicting $w\in B3$.
\end{proof}

We next verify the possible transitions between states in \Cref{fig:long-degens-DFA-412p3p4} are as claimed.
\begin{itemize}
    \item Let $w \in 41A4$.
    \begin{itemize}
        \item[$e_1$:] does not apply.
        \item[$e_2$:] only applies to $A$ and preserves it.  
        \item[$e_3$:] always acts on the rightmost $4$, as it is unmatched, which gives a transition to a word in $41A3 \subset 41B3$.
    \end{itemize}
    \item Let $w \in 41B3$.
    \begin{itemize}
        \item[$e_1$:] does not apply.
        \item[$e_2$:] either is applied at the $3$, resulting in $41B2 \subset 41C2$, or it is applied inside $B$. In the latter case, the associativity of the tensor product of crystals guarantees that the action of $e_2$ on $w$ is compatible with the action on the subword from $B$. Thus we obtain a word in $41B3$ by the claim. 
        \item[$e_3$:] either is applied inside $B$, giving again a word in $41B3$, or it is applied at the $4$. The latter case occurs if and only if every $\overline{3}$ is matched with a $\overline{4}$ to its left. As $B$ avoids the subsequence $\overline{4}\cdots\overline{3}$, we must have that $w$ contains no $\overline{3}$, so $e_3(w) \in 31D3$.
    \end{itemize}
    \item Let $w \in 41C2$.
    \begin{itemize}
        \item[$e_1$:] does not apply, as the ending $2$ is always matched.
        \item[$e_2$:] is applied inside $C$, which is preserved by $e_2$ by the claim.
        \item[$e_3$:] either is applied inside $C$, which is preserved by $e_3$, or it is applied at the $4$. The latter case occurs if and only if every $\overline{3}$ is matched with a $\overline{4}$ on its left. As $C$ avoids the subsequence $\overline{4}\cdots\overline{3}\cdots\overline{2}$, we must have $w \in 41E2$ in this case, so $e_3(w) \in 31E2$.
    \end{itemize}
    \item Let $w \in 31D3$.
    \begin{itemize}
        \item[$e_1$:] does not apply.
        \item[$e_2$:] acts on the ending $3$ as $w$ does not contain a $2$ or a $\overline{3}$ that could match it, giving a word in $31D2\subset 31E2$. 
        \item[$e_3$:] does not apply.
    \end{itemize}
    \item Let $w\in 31E2$.
    \begin{itemize}
        \item[$e_1$:] does not apply, as the ending $2$ is always matched.
        \item[$e_2$:] either acts on $E$, which is preserved by the claim, or on the $3$. In the latter case, all $\overline{2}$'s must be matched to the left by $\overline{3}$'s. As $E$ contains no $\overline{3}\cdots\overline{2}$ subsequence, this can only be the case when $w$ contains no $\overline{2}$, thus application of $e_2$ gives a word in $21F2$.
        \item[$e_3$:] can only act on $E$, which is preserved by the claim.
    \end{itemize}
    \item Let $w\in 21F2$.
    \begin{itemize}
        \item[$e_1$:] acts on the initial $2$ as the ending $2$ is matched by the $1$ in the second position, giving a word in $11F2$.
        \item[$e_2$:] does not apply.
        \item[$e_3$:] can only act on $F$, which is preserved by $e_3$.
    \end{itemize}
    \item Let $w\in 11F2$.
    \begin{itemize}
        \item[$e_1$:] does not apply as the ending $2$ is matched by the $1$ in the second position.
        \item[$e_2$:] does not apply.
        \item[$e_3$:] can only act on $F$, which is preserved by $e_3$. \qedhere
    \end{itemize}
\end{itemize}
\end{proof}

\subsection{Proof of the growth algorithm theorem}\label{sec:growth-theorem-proof}

We now piece together the preceding results to establish the growth algorithm theorem.

\begin{proof}[Proof of {\Cref{thm:growth_algorithm}}]
  We treat the oscillating case, the extension to the general fluctuating case being immediate. \Cref{alg:growth} terminates in an hourglass plabic graph by \Cref{lem:some-growth-rule-applies} and \Cref{cor:growth_no_cycles}, using induction on lattice words of oscillating type under length and then $\tlex$-order. The same induction using \Cref{thm:good_degen} gives condition (ii), $\trip_\bullet(G) = \prom_\bullet(T)$ for $G = \mathcal{G}(T)$ with $w = L(T)$ of oscillating type. Condition (i) now follows from \Cref{lem:growth-rules-monotonic-mostly} and \Cref{cor:fully-reduced-iff-monotonic}. Conditions (iii) and (iv) are \Cref{thm:nice-unitriangularity}. The forward implication of (v) is \Cref{lem:descents-are-connected}. Conversely, suppose $b_i$ and $b_{i+1}$ of $G$ are connected to the same vertex, where $i \in [n-1]$. Their boundary labels must have the same sign, and they must increase by the $\tlex$-minimality condition (iv), which completes the argument for (v). 
\end{proof}

\section{The hourglass web basis}
\label{sec:bijection-and-basis}

We are now ready to prove our main theorems. \Cref{thm:main-bijection} establishes a bijection between move-equivalence classes of fully reduced hourglass plabic graphs and rectangular fluctuating tableaux. Here, the key facts are that a tableau is determined by the antiexcedances of its promotion permutations, while the separation labeling of the boundary edges of an hourglass plabic graph is determined by the antiexcedances of its trip permutations. Since hourglass plabic graphs are move-equivalent exactly when they share the same tuple of trip permutations, this means that the separation labeling of the boundary edges is a complete invariant of the move-equivalence class. In this way, a tableau and class of hourglass plabic graphs correspond when they share the same antiexcedance data.

\Cref{thm:unitriangular} shows, via the growth algorithm properties of \Cref{thm:growth_algorithm}, that the tensor invariant of a fully reduced hourglass plabic graph has a leading term given by the boundary word of its separation labeling. Moreover, the coefficient on this term is a unit in the ring $\mathbb{Z}[q, q^{-1}]$.

\Cref{thm:web-basis} obtains a rotation-invariant web basis for the invariant space $\Inv_{U_q(\fsl_4)} (\bigwedge_q^{\underline{c}} V_q)$ as the invariants of fully reduced hourglass plabic graphs. Here, we choose a rotation-invariant representative from each move-equivalence class; we call these representatives ``top'' elements because they are maximal under a natural lattice structure on the move-equivalence class.

\subsection{The main bijection}
\label{sec:main-bijection}
\Cref{thm:sep-label-promotion-equivariant} shows that the separation labeling determines a map $\mathcal{T}$ from contracted fully reduced hourglass plabic graphs to rectangular fluctuating tableaux of the same type. Since move-equivalent graphs have the same trip permutations (\Cref{prop:trip-of-underlying-plabic}), and thus the same separation labels on the boundary edges (\Cref{prop:sep-boundary-label-from-trips}), we can in fact consider this function as a map
\[\mathcal{T}: \crg(\underline{c})/{\sim} \to \rft(\underline{c}).\]
\Cref{thm:growth_algorithm} shows that growth rules define a map $\mathcal{G}: \rft(\underline{c})\to \crg(\underline{c})$, which induces a map
\[\ \mathcal{G}:\rft(\underline{c}) \to \crg(\underline{c})/{\sim}.\] \Cref{thm:main-bijection} below shows that $\mathcal{T}$ and $\mathcal{G}$ are mutually inverse bijections; this is our main combinatorial result, enabling the construction of our web basis.

\begin{theorem}
\label{thm:main-bijection}
The maps $\mathcal{T}:\crg(\underline{c})/{\sim} \to \rft(\underline{c})$ and $\mathcal{G}:\rft(\underline{c}) \to \crg(\underline{c})/{\sim}$ are mutually inverse bijections. Furthermore, this bijection satisfies $\trip_{\bullet}(G)=\prom_{\bullet}(\mathcal{T}(G))$. Consequently, it intertwines promotion and evacuation of tableaux with rotation and reflection of hourglass plabic graphs.
\end{theorem}
\begin{proof}
Let $T \in \rft(\underline{c})$ and $G=\mathcal{G}(T)$. By the properties of the growth algorithm established in \Cref{thm:growth_algorithm}, we have $G \in \crg(\underline{c})$ is a contracted fully reduced hourglass plabic graph satisfying $\trip_{\bullet}(G)=\prom_{\bullet}(T)$. By properties of separation labeling given in \Cref{thm:sep-label-promotion-equivariant}, $\mathcal{T}(G) \in \rft(\underline{c})$ is a fluctuating tableau. \Cref{prop:sep-boundary-label-from-trips} explains how the separation labeling of boundary edges is determined by antiexcedances of trip permutations, while \cite[Thm.~6.12]{fluctuating-paper} explains how fluctuating tableaux are determined by antiexcedances of their promotion permutations. Combining these facts, we have $\mathcal{T}(G)=T$, so the composition $\mathcal{T} \circ \mathcal{G}$ is the identity on the set $\rft(\underline{c})$ of fluctuating tableaux. 

Now let $G \in \crg(\underline{c})$ and let $T=\mathcal{T}(G)$. By \Cref{thm:sep-label-promotion-equivariant}, $T \in \rft(\underline{c})$ is a fluctuating tableau. By \Cref{thm:sep-label-promotion-equivariant}, we therefore have that $\mathcal{T}(\rot^i(G))=\mathcal{P}^i(T)$ for all $i$. Let $\sigma=(1 2 \ldots n) \in \mathfrak{S}_n$. It is easy to see that if $\aexc(\sigma^i \pi \sigma^{-i}) = \aexc(\sigma^i \pi' \sigma^{-i})$ for all $i$, then $\pi=\pi'$. Thus by \cite[Thms.~6.7 \& 6.12]{fluctuating-paper} and \Cref{prop:sep-boundary-label-from-trips}, we have $\prom_{\bullet}(T)=\trip_{\bullet}(G)$. Let $G'=\mathcal{G}(T)$; by \Cref{thm:growth_algorithm}, we have $\prom_{\bullet}(T)=\trip_{\bullet}(G')$. Thus by \Cref{thm:hourglass-trips-determine-move-equivalence}, showing that trip permutations determine move-equivalence classes, we have $G' \sim G$, so that $\mathcal{G} \circ \mathcal{T}$ is the identity on $\crg(\underline{c})/{\sim}$.

We have now shown that $\mathcal{T}$ and $\mathcal{G}$ are mutually inverse bijections satisfying $\trip_{\bullet}=\prom_{\bullet}$. The rest of the theorem then follows from \Cref{thm:sep-label-promotion-equivariant}.
\end{proof}

\begin{remark}
    \Cref{thm:growth_algorithm}(iv) and (v) imply that under the bijection from \Cref{thm:main-bijection}, descents of tableau correspond to boundary vertices $b_i$ and $b_{i+1}$ sharing an internal vertex. 
\end{remark}

\subsection{Unitriangularity from separation words}

We now make precise the conversion from fully reduced hourglass plabic graphs to webs. Note that the following tagging convention only affects $[W]_q$ by a sign.

\begin{definition}
\label{def:hourglass-tagging-convention}
Let $W \in \crg(\underline{c})$. We view $W$ as a web of type $\underline{c}$ (see \Cref{def:U-q-web}) by considering each $m$-hourglass as an edge of multiplicity $m$, and by tagging the vertices $v$ of $W$ as follows: 
\begin{itemize}
    \item If $v$ is incident to a 2-hourglass and two simple edges, place the tag between the simple edges.
    \item If $v$ is incident to four simple edges, then the two $\trip_2$-strands passing through $v$ divide the disk into four sectors (they do not double-cross since $W$ is fully reduced). Place the tag in the sector containing the base face of $W$.
    \item If $v$ is incident to a boundary 3-hourglass and a simple edge, place the tag on the side of the $\trip_2$-strand through the simple edge which contains the base face.
    \item If $v$ is incident to two boundary 2-hourglasses, place the tag on the side of the base face.
\end{itemize}
\noindent Note that there is no choice when $v$ is incident to a boundary $4$-hourglass.
\end{definition}

The following unitriangularity result is analogous to \cite[Thm.~2]{Khovanov-Kuperberg} (see also \cite[Thm.~3.2]{LACIM}). 

\begin{theorem}\label{thm:unitriangular}
  Let $G \in \crg(\underline{c})$. Then the separation word $w = \bw(G)$ is the unique $\tlex$-minimal boundary word among all proper labelings of $G$. Thus
  \begin{equation}\label{eq:unitriangular}
    [G]_q = \pm q^a\,x_w + \sum_{v >_{\tlex} w} d_w^v(q)\,x_v
  \end{equation}
  for some $a \in \mathbb{Z}$ and $d_w^v(q) \in \mathbb{Z}[q, q^{-1}]$, where $v$ ranges over words of type $\underline{c}$.
\end{theorem}

\begin{proof}
  If $G = \mathcal{G}(T)$ arises from the growth algorithm, the result follows from \eqref{eq:q.Laplace} and \Cref{thm:growth_algorithm}(iii)--(iv). (The tagging only results in a global sign change.) If $G'$ is benzene-equivalent to such $G$, the result holds for $G'$ by \Cref{thm:nice-unitriangularity} and \Cref{lem:nice-benzenes}. Finally, if $G''$ is square-move-equivalent to such $G'$, then $[G']_q = [G'']_q$. By \Cref{cor:benzene-then-square}, we may apply benzene moves before square moves; hence, by \Cref{thm:main-bijection}, this covers all $G \in \crg(\underline{c})$.
\end{proof}

\begin{remark}
Note that for $G$ and $G'$ differing by a benzene move, the invariants $[G]_q$ and $[G']_q$ are not necessarily equal (see the relation in \Cref{fig:benzene-relation}), however, since $\bw(G)=\bw(G')$, they have the same $\tlex$-leading term by \Cref{thm:unitriangular}.
\end{remark}

\subsection{Distinguished representatives and the web basis}

The bijection from \Cref{thm:main-bijection} implies that $|\crg(\underline{c})/{\sim}| = |\rft(\underline{c})|$. By \Cref{prop:fluc-tab-give-dimension}, this quantity is also equal to $\dim_{\mathbb{C}} \Inv(\bigwedge\nolimits^{\underline{c}} V)=\dim_{\mathbb{C}(q)}\Inv_{U_q(\fsl_4)}(\bigwedge_q\nolimits^{\underline{c}} V_q)$. In this section, we obtain our ${U_q(\fsl_4)}$-web basis $\mathscr{B}_q^{\underline{c}}$ by producing a distinguished web invariant from each move-equivalence class of $\crg(\underline{c})$.

\begin{definition}
A \emph{benzene face} in an hourglass plabic graph is a face that admits a benzene move. We say a benzene face is \emph{clockwise} (resp.\ \emph{counterclockwise}) if in each hourglass edge the white vertex precedes the black vertex in clockwise (resp.\ counterclockwise) order; equivalently, if the corresponding triangle in the \symm six-vertex configuration is oriented clockwise (resp.\ counterclockwise). Let $\mathcal{B}$ be a benzene-move-equivalence class of contracted fully reduced hourglass plabic graphs. For $x, y \in \mathcal{B}$,  write $x \precdot y$ and say $y$ \emph{covers} $x$ when $y$ is obtained by applying a benzene move to a clockwise benzene face in $x$, resulting in a counterclockwise benzene face in $y$. Let $\preceq$ be the reflexive, transitive closure of this relation on~$\mathcal{B}$.
\end{definition}

\begin{proposition}
\label{prop:benzene-max-exists}
  The relation $(\mathcal{B}, \preceq)$ is a partially ordered set with a unique maximum, which may be reached by starting at an arbitrary element and applying covering relations arbitrarily until no more apply.
\end{proposition}
\begin{proof}
Consider a sequence of covers in $(\mathcal{B},\preceq)$, corresponding to benzene moves on faces $F_1,F_2,\ldots$. Between any two benzene moves at the same face $F$, a benzene move must be applied to all faces $F'$ sharing an edge with $F$. Since no moves can be applied to boundary faces, this implies that $F$ appears finitely many times in the sequence.  This guarantees antisymmetry, so that $\preceq$ is a partial order. 

  Now, $\mathcal{B}$ possesses the \emph{diamond property}: if $x \precdot y_0$ and $x \precdot y_1$, then there is some $z \in \mathcal{B}$ with $y_0 \precdot z$ and $y_1 \precdot z$. Indeed, $y_0$ and $y_1$ are obtained from $x$ by benzene moves on two distinct faces which, by the orientation condition, must be non-adjacent, so the two benzene moves commute. Newman's Lemma \cite[Thm.~1]{Newman} asserts that this diamond property implies the existence of a unique maximal element.
\end{proof}

\begin{remark}
Standard techniques (see e.g. \cite[\S 3.1]{Kenyon}) can be used to prove that $\mathcal{B}$ is in fact a distributive lattice, but we will not use this fact.
\end{remark}

\begin{definition}
\label{def:top-fully-reduced}
A contracted fully reduced hourglass plabic graph $G$ that is maximal in its benzene-move-equivalence class (in the sense of \Cref{prop:benzene-max-exists}) is called \emph{top fully reduced}. Equivalently, $G \in \crg(\underline{c})$ is top fully reduced if and only if it has no clockwise benzene faces.
\end{definition}

\begin{theorem}
\label{thm:web-basis}
The collection 
\[
\mathscr{B}_q^{\underline{c}} \coloneqq \{[W]_q : \text{ $W$ a top fully reduced hourglass plabic graph of type $\underline{c}$}\}
\]
is a rotation-invariant web basis for $\Inv_{U_q(\fsl_4)}(\bigwedge_q\nolimits^{\underline{c}} V_q)$.
\end{theorem}
\begin{proof}
Let $\mathcal{C}$ be a move-equivalence class of graphs from $\crg(\underline{c})$. By \Cref{cor:benzene-then-square}, square moves and benzene moves on fully reduced graphs commute: any $G,G' \in \mathcal{C}$ may be connected by first applying a sequence of benzene moves and then a sequence of square moves. By \cite[Eq.~2.10]{Cautis-Kamnitzer-Morrison}, square moves do not affect the web invariant; thus any top fully reduced hourglass plabic graphs $W,W' \in \mathcal{C}$ satisfy $[W]_q=[W']_q$. We write $[\mathcal{C}]_q$ for this common value.

By \Cref{thm:q.Laplace} and \Cref{thm:unitriangular}, the invariants $[\mathcal{C}]_q$ for $\mathcal{C} \in \crg(\underline{c})/{\sim}$ are linearly independent, since they have distinct leading terms. Since 
\[
|\crg(\underline{c})/{\sim}| =|\rft(\underline{c})|=\dim_{\mathbb{C}(q)} \Inv_{U_q(\fsl_4)}(\bigwedge_q\nolimits^{\underline{c}} V_q)
\]
by \Cref{thm:main-bijection} and \Cref{prop:fluc-tab-give-dimension},  the invariants $[\mathcal{C}]_q$ form a web basis.

Rotation invariance follows by observing that the rotation of a top fully reduced graph is again top fully reduced.
\end{proof}

\begin{corollary}
  For any complete collection of representatives $\mathscr{R}$ for $\crg(\underline{c})/{\sim}$, the corresponding invariants $[G]_q$ for $G \in \mathscr{R}$ form a $\mathbb{C}(q)$-basis for $\Inv_{U_q(\fsl_4)}(\bigwedge\nolimits_q^{\underline{c}} V_q)$.
\end{corollary}

\begin{proof}
  They are linearly independent by \Cref{thm:unitriangular}, and there are the right number by \eqref{eq:fluc-tab-give-dimension} and \Cref{thm:main-bijection}.
\end{proof}

We thank Greg Kuperberg for suggesting \Cref{thm:unitriangular.canonical} towards making connections with the dual canonical basis.

\begin{corollary}\label{thm:unitriangular.canonical}
  Let $E = ([x_v] [G_w]_q)$ be a matrix whose rows are indexed by words $v$ of type $\underline{c}$ listed in decreasing $\tlex$ order and whose columns are indexed by $w \in \rft(\underline{c})$ listed in decreasing $\tlex$ order, where $G_w \in \crg(\underline{w})$ with $\bw(G) = w$. Then $E$ is in row echelon form, and each column has a pivot of the form $\pm q^a$ for $a \in \mathbb{Z}$.
\end{corollary}

\begin{proof}
  This follows by inspecting the matrix $E$ and using \Cref{thm:unitriangular}.
\end{proof}

In \Cref{sec:reduction-rules} we give reduction relations which allow one to express the invariant of a (possibly non-planar) tensor diagram in our web basis. This is of particular interest in the computation of quantum link invariants. Moreover, these relations allow us to prove a weaker version of \Cref{thm:web-basis} in a way that avoids relying on the growth rules of \Cref{sec:forwards-map}; see \Cref{sec:bypassing_growth_rules} for details.

\section{Uncrossing and reduction relations}
\label{sec:reduction-rules}

\subsection{Reduction relations}
\label{sec:reduction_relations}
Using the braiding of the representation category for $U_q(\fsl_r)$ (see \cite[\S 6]{Cautis-Kamnitzer-Morrison}), one may assign invariants $[X]_q$ to \emph{tensor diagrams} $X$, variants of webs in which edges may cross over one another. When $X$ is the projection of a (colored) knot, link, or tangle, $[X]_q$ yields quantum knot invariants (e.g., \cite{Jones, Kauffman,Turaev,Khovanov:Jones,Khovanov:sl3,Morrison.Nieh}) of great interest. The reader may take the relations in \Cref{fig:uncrossing-relation} (adapted to our conventions from \cite[Cor.~6.2.3]{Cautis-Kamnitzer-Morrison}) as defining $[X]_q$ in the case $r=4$. 

In \Cref{alg:reduction-algorithm} below, we describe how the invariant $[X]_q$ of an arbitrary $U_q(\fsl_4)$-tensor diagram may be expanded in the web basis $\mathscr{B}_q^{\underline{c}}$ of \Cref{thm:web-basis}. The relations appearing in this section are drawn from \cite{Cautis-Kamnitzer-Morrison} and adapted to our conventions. As noted by Kuperberg \cite[\S 8.3]{Kuperberg}, such reduction algorithms to a web basis allow for the efficient computation of quantum link and tangle invariants. In the case $\underline{c}=(1,\ldots,1)$ and $q=1$, where the representation of $\mathfrak{S}_n$ on $\Inv(\bigwedge\nolimits^{\underline{c}} V)$ via permutation of tensor factors is a rectangular Specht module, the uncrossing relations also allow for the computation of the actions of the simple reflections in the basis $\mathscr{B}^{\underline{c}}$.

\begin{figure}[tbh]
    \centering
    \includegraphics[scale=1.2]{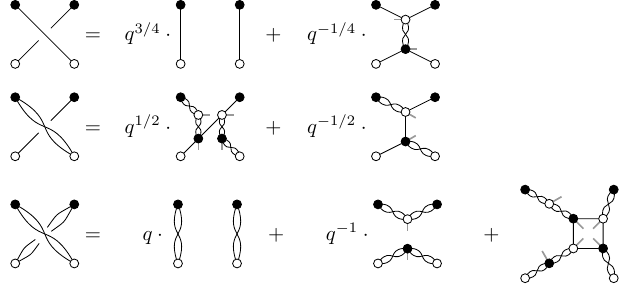}
    \caption{The uncrossing relations for invariants of $U_q(\fsl_4)$-tensor diagrams.}
    \label{fig:uncrossing-relation}
\end{figure}

Let $[k]_q \coloneqq (q^k-q^{-k})/(q-q^{-1})$.

\begin{algorithm}[\textsc{Reduction algorithm}]
\label{alg:reduction-algorithm}
Let $X$ be any $U_q(\fsl_4)$-tensor diagram.
\begin{enumerate}
    \item[(1)] Apply uncrossing relations (\Cref{fig:uncrossing-relation}) to $X$ to remove any crossings.
    \item[(2)] For each web appearing in this expansion, apply relations from \Cref{fig:bad-4-cycle-relation} to decompose all forbidden $4$-cycles.
    \item[(3)] For each web appearing in this expansion, remove any loops using the relations from \Cref{fig:loop-relation}.
    \item[(4)] For each web $W$ appearing in this expansion, if $W$ is not fully reduced, apply the benzene relation of \Cref{fig:benzene-relation} to create a forbidden $4$-cycle and return to Step (2). Otherwise continue.
    \item[(5)] Apply benzene relations to express a fully reduced web invariant as $$[W]_q=[W']_q+\left(\sum_{\alpha} \pm [W_{\alpha}]_q \right),$$ where $W'$ is top fully reduced and the $W_{\alpha}$ have fewer faces. Return to Step (2) to decompose each $W_{\alpha}$.
\end{enumerate}
\end{algorithm}

\begin{figure}[tbh]
    \centering
    \includegraphics{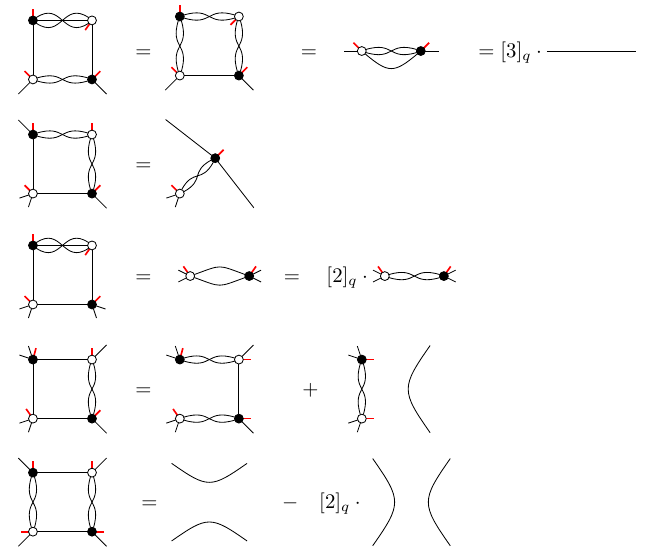}
    \caption{Relations decomposing the forbidden $4$-cycles of $U_q(\fsl_4)$-webs. Here, $[k]_q \coloneqq (q^k-q^{-k})/(q-q^{-1})$. Tags are drawn in red ($\textcolor{red}{\blacksquare}$) for clarity.}
    \label{fig:bad-4-cycle-relation}
\end{figure}

\begin{figure}[th]
    \centering
    \includegraphics{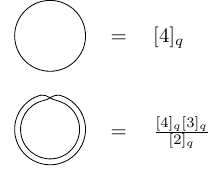}
    \caption{The loop deletion relations for $U_q(\fsl_4)$-webs.}
    \label{fig:loop-relation}
\end{figure}

\begin{figure}[tbh]
    \centering
    \includegraphics{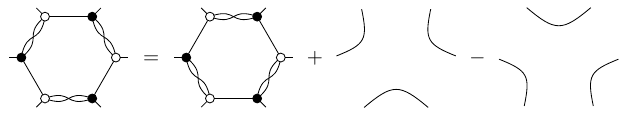}
    \caption{The benzene relation on invariants of $U_q(\fsl_4)$-webs.}
    \label{fig:benzene-relation}
\end{figure}

\begin{theorem}
\label{thm:reduction-algorithm-expands-into-basis}
For $X$ a tensor diagram of type $\underline{c}$, \Cref{alg:reduction-algorithm} expresses $[X]_q$ in the basis $\mathscr{B}_q^{\underline{c}}$.
\end{theorem}
\begin{proof}
If $X$ is a tensor diagram, then the uncrossing relations unambiguously express $[X]_q$ as a combination of web invariants. By \Cref{def:fully-reduced}, each web that is not fully reduced is modified by one of the Steps (2--4). If a web is modified only by a benzene move in Step (4), then it will also be modified by Step (2) on the subsequent iteration. Since the loop deletion and $4$-cycle relations (\Cref{fig:bad-4-cycle-relation,fig:loop-relation}) all decrease the number of faces, the algorithm terminates in a fully reduced web. Finally, by \Cref{prop:benzene-max-exists}, there is a sequence of benzene moves converting any fully reduced web into a basis web (and in fact this sequence may be chosen arbitrarily).
\end{proof}

\subsection{Another proof of the web basis}
\label{sec:bypassing_growth_rules}
\Cref{thm:reduction-algorithm-expands-into-basis} allows us to prove weaker versions of some of the main results of \Cref{sec:bijection-and-basis} while avoiding the growth rules of \Cref{sec:forwards-map}. 

\begin{theorem}[cf.\ \Cref{thm:main-bijection} and \Cref{thm:web-basis}]
\label{thm:basis-without-growth-rules}
    The map $\mathcal{T}:\crg(\underline{c})/{\sim} \to \rft(\underline{c})$ is a bijection satisfying $\trip_{\bullet}(G)=\prom_{\bullet}(\mathcal{T}(G))$; consequently it intertwines promotion and evacuation of tableaux with rotation and reflection of hourglass plabic graphs. Furthermore, 
    \[   \mathscr{B}_q^{\underline{c}} \coloneqq \{[W]_q : \text{ $W$ a top fully reduced hourglass plabic graph of type $\underline{c}$}\}
\]
is a rotation-invariant web basis for $\Inv_{U_q(\fsl_4)}(\bigwedge_q\nolimits^{\underline{c}} V_q)$. 
\end{theorem}
\begin{proof}
    Let $G \in \crg(\underline{c})$ and let $T=\mathcal{T}(G)$. As argued in the second paragraph of the proof of \Cref{thm:main-bijection}, the results of \Cref{sec:backwards-map} imply that $T \in \rft(\underline{c})$ and that $\trip_{\bullet}(G)=\prom_{\bullet}(T)$. As before $T \in \rft(\underline{c})$ is determined by $\prom_{\bullet}(T)$ by \cite[Thm.~6.12]{fluctuating-paper}. Therefore we have $\mathcal{T}(G)=\mathcal{T}(G')$ if and only if $\trip_{\bullet}(G)=\trip_{\bullet}(G')$; by \Cref{thm:hourglass-trips-determine-move-equivalence}, this happens exactly when $G \sim G'$. Thus the induced map $\mathcal{T}:\crg(\underline{c})/{\sim} \to \rft(\underline{c})$ is well-defined and injective.

    By \Cref{prop:fluc-tab-give-dimension} we have $|\rft(\underline{c})|=\dim_{\mathbb{C}(q)} \Inv_{U_q(\fsl_4)}(\bigwedge_q\nolimits^{\underline{c}} V_q)$. By (a quantum version of) the Fundamental Theorem of Invariant Theory (see \cite{Cautis-Kamnitzer-Morrison} and \cite{Weyl}), the set of invariants $[X]_q$ of tensor diagrams spans $\Inv_{U_q(\fsl_4)}(\bigwedge_q\nolimits^{\underline{c}} V_q)$. Hence, by \Cref{thm:reduction-algorithm-expands-into-basis}, $\mathscr{B}_q^{\underline{c}}$ also spans over $\mathbb{C}(q)$. But \[
    |\mathscr{B}_q^{\underline{c}}|=|\crg(\underline{c})/{\sim}| \leq |\rft(\underline{c})|=\dim_{\mathbb{C}(q)} \Inv_{U_q(\fsl_4)}(\bigwedge_q\nolimits^{\underline{c}} V_q)\]
    by injectivity. Thus $\mathscr{B}_q^{\underline{c}}$ is a $\mathbb{C}(q)$-basis.
\end{proof}

\begin{remark}\label{remark:what_we_lose}
Although \Cref{thm:basis-without-growth-rules} allows for a much shorter proof of the web basis, the growth rules were essential for the following results from \Cref{sec:bijection-and-basis}: 
    \begin{enumerate}
        \item the explicit inverse to the bijection $\mathcal{T}$ and  explicit process for constructing $\crg(\underline{c})$,
        \item the proof of \Cref{thm:unitriangular} and consequently the fact that the separation word gives the leading term of $[W]_q$ and that this term has coefficient which is a unit in $\mathbb{Z}[q,q^{-1}]$, and
        \item the integrality of the basis $\mathscr{B}_q^{\underline{c}}$.
    \end{enumerate}
\end{remark}

\section{Combinatorial applications of the web basis}
\label{sec:cc-asm-pp}
In this section, we collect results related to our web basis $\mathscr{B}_q^{\underline{c}}$ that are of particular interest in combinatorics.
We first re-prove a cyclic sieving result. We then summarize two particular extreme move-equivalence classes relating to alternating sign matrices and plane partitions, respectively. 

\subsection{Cyclic sieving and basis webs}
\Cref{thm:web-basis} can be used to easily recover the $4$-row case of a celebrated cyclic sieving result of Rhoades \cite{Rhoades}, originally proved using Kazhdan--Lusztig theory. The $3$-row case was re-proven in \cite{Petersen-Pylyavskyy-Rhoades}, using Kuperberg's $U_q(\fsl_3)$-web basis. For the definition of the \emph{$q$-hook length polynomial} below, see \cite[Cor.~7.21.6]{Stanley:EC2} or \cite{Rhoades}. 

\begin{corollary}
\label{cor:cyclic-sieving}
Let $n = 4k$, $\lambda=(k,k,k,k)$, and $\zeta=e^{2 \pi i/ n}$. The number of standard tableaux of shape $\lambda$ fixed by $\mathcal{P}^{d}$ is $f^{\lambda}(\zeta^d)$, where $f^{\lambda}(q)$ is the $q$-hook length polynomial.
\end{corollary}
\begin{proof}
Let $\sigma=(1 \: 2 \; \cdots \: n) \in \mathfrak{S}_{n}$. By Springer's theory of regular elements \cite[Prop. 4.5]{Springer}, we have $\chi^{\lambda}(\sigma^d)=(-1)^df^{\lambda}(\zeta^d)$, where $\chi^{\lambda}$ is the character of the Specht module $\mathsf{S}^{\lambda} \cong \Inv(V^{\otimes n})$. On the other hand, since the basis $\mathscr{B}^{(1^{n})}$ is permuted up to sign by the action of $\sigma$ and is in $\mathcal{P}$-equivariant bijection with standard tableaux of shape $\lambda$, this character value is a signed count of the number of $\mathcal{P}^d$-fixed tableaux.

Moreover, each fixed point contributes a sign of precisely $(-1)^d$ as follows. Suppose $W \in \crg((1^{n}))$. Using our tagging conventions, $[\rot(W)] = (-1)^k \cdot \shuffle([W])$ where $k$ is the number of vertices of simple degree $4$ on the $\trip_2(1)$-strand of $W$, and $\shuffle$ is the cyclic shift isomorphism from \Cref{sec:intro_bases}. For a \symm six-vertex configuration $\varphi(W)$, this strand has two incoming boundary edges. There must be an odd number of direction changes of the arrows along the strand, which occur precisely at sources or sinks, so $k$ is odd. The result follows.
\end{proof}

\subsection{Alternating sign matrices and square moves}
\label{sec:asm}
The set $\ASM_n$ of \emph{alternating sign matrices} consists of $n \times n$ matrices with entries from $\{-1, 0, 1\}$ where the nonzero entries of each row and column begin and end with $1$ and alternate between $1$ and $-1$. Alternating sign matrices have long been of interest in enumerative combinatorics (see, e.g., \cite{Mills.Robbins.Rumsey, Zeilberger, Kuperberg_ASM_pf, ProppManyFaces,Kuperberg:roof}). Recently, they have obtained further significance through connections \cite{Weigandt} to \emph{back stable Schubert calculus} \cite{Lam.Lee.Shimozono}.

Permutation matrices are instances of alternating sign matrices. Indeed, $\ASM_n$ is the MacNeille completion of the strong Bruhat order on $\mathfrak{S}_n$ \cite{Lascoux.Schutzenberger}, i.e.~$\ASM_n$ is the smallest lattice containing Bruhat order. More strongly, $\ASM_n$ is a distributive lattice with covering relations given by adding {\small $\begin{pmatrix}-1 & 1 \\ 1 & -1\end{pmatrix}$} where possible. The sub-poset of join-irreducible elements can be realized as a certain tetrahedral poset embedded in three-dimensional space (see \cite[Fig.~5]{RazStrogRow} and \cite{ELKP92,Striker.tetrahedralposet}). Hence $\ASM_n$ can be viewed as the order ideals of this tetrahedral poset. 

\begin{proposition}\label{prop:ASM}
  Let $T$ be the ``superstandard'' $4$-row standard tableau with lattice word $1^n 2^n 3^n 4^n$. The move-equivalence class of $\mathcal{G}(T)$ is in bijection with $\ASM_n$, with square moves corresponding to covering relations. No benzene moves apply to elements of this class.
\end{proposition}

\begin{proof}
Consider the \symm six-vertex configurations obtained by placing $4n$ incoming strands around a square grid and assigning orientations to the middle in all possible ways. Such configurations are in bijection with configurations of the classical six-vertex model (with \emph{domain wall} boundary conditions) by reversing all horizontal arrows, and these in turn are in bijection with $\ASM_n$ by \cite{Kuperberg_ASM_pf}. The composite bijection from the \symm six-vertex configurations to $\ASM_n$ may be realized by replacing sinks with $1$, sources with $-1$, and transmitting vertices with $0$. See \Cref{fig:ASM-3x3} for examples. 

These \symm six-vertex configurations are clearly well oriented by \Cref{prop:trip2_6V} and admit no Yang--Baxter moves, since they contain no big triangles. It is easy to check that the corresponding fully reduced hourglass plabic graphs (see \Cref{thm:6-vertex-hourglass-correspondence}) have separation word $1^n 2^n 3^n 4^n$.
\end{proof}

\begin{figure}[hbtp]
    \centering
    \includegraphics[width=\linewidth]{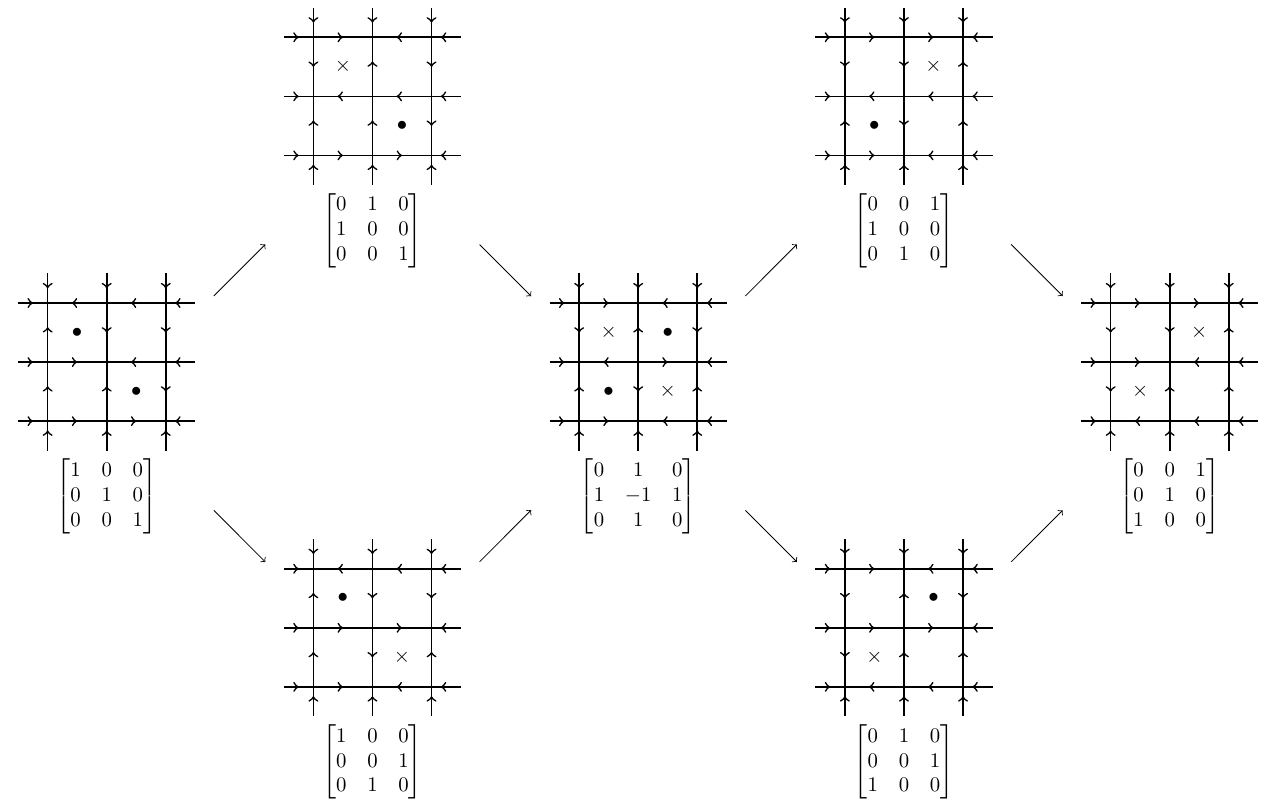}
    \caption{The $7$ elements of $\ASM_3$ together with the corresponding \symm six-vertex configurations in the move-equivalence class of $\varphi(\mathcal{G}(T))$. Here $T$ is the standard tableau with lattice word $L = 111222333444$. ASM moves going up (resp. down) in the poset are marked with $\bullet$ (resp. $\times$).}
    \label{fig:ASM-3x3}
\end{figure}

See \Cref{remark:ASM_Yang_Baxter} for further connections between square moves and alternating sign matrix dynamics.

\subsection{Plane partitions and benzene moves}
\label{sec:plane-partitions}
The set $\PP(a \times b \times c)$ of \emph{plane partitions} in the $a \times b \times c$ box consists of the $a \times b$ matrices of nonnegative integers that weakly decrease along rows and columns and have maximum entry at most $c$. Like alternating sign matrices, plane partitions have long been prominent in enumerative combinatorics (see, e.g, \cite{Andrews:MacMahon,Andrews:Macdonald,Stembridge:Pfaffian}) and representation theory (see, e.g., \cite{Proctor:pp,Garver.Patrias.Thomas}). They are now one of the central objects of study in the burgeoning area of dynamical algebraic combinatorics (e.g., \cite{Striker.Williams,Striker:DAC,Hopkins:CSP,Patrias.Pechenik:CFDF}).

The plane partitions $\PP(a \times b \times c)$ may be equivalently viewed as the configurations of stacked unit cubes in the prism $[0, a] \times [0, b] \times [0, c]$ where gravity acts parallel to $(-1, -1, -1)$. Such plane partitions form a distributive lattice and may be seen as order ideals of the product of chains poset $[a]\times[b]\times[c]$. They have further alternative interpretations as rhombus tilings or perfect matchings on particular subsets of the hexagonal lattice.

\begin{proposition}
\label{prop:pp}
  Let $a, b, c \in \mathbb{Z}_{\geq 1}$. If $a \geq c$, let 
  $
  L = 1^a \, \overline{4}\vphantom{4}^b \, 2^c \, \overline{1}\vphantom{1}^{a-c} \, \overline{2}\vphantom{2}^c \, 4^b \, \overline{1}\vphantom{1}^c,
  $
  and if $a \leq c$, let $L = 1^a \, \overline{4}\vphantom{4}^b \, 2^a \, 1^{c-a} \, \overline{2}\vphantom{2}^a \, 4^b \, \overline{1}\vphantom{1}^c$. Let $T$ be the tableau with lattice word $L$. The move-equivalence class of $\mathcal{G}(T)$  is in bijection with $\PP(a \times b \times c)$, with benzene moves corresponding to covering relations. No square moves apply to elements of this class.
\end{proposition}

\begin{proof}
  In the hexagonal lattice, begin with a ``large hexagon''  with $a, b, c, a, b, c$ hexagons, respectively, on each side in cyclic order; see \Cref{fig:PP-tiling}. The perfect matchings on this large hexagon can be converted to rhombus tilings by considering the dual graph and omitting edges of the dual graph that cross matched edges of the large hexagon. In this way, benzene moves correspond to adding or removing a box in the plane partition; see \Cref{fig:PP-2x2x2}. Alternatively, perfect matchings on the large hexagon may be converted to hourglass plabic graphs by replacing matched edges with hourglass edges. In this move-equivalence class, there are no square faces, so no square moves apply, and hence the graphs are fully reduced.
  
  We show that the separation labeling is $L = 1^a \, \overline{4}\vphantom{4}^b \, 2^a \, 1^{c-a} \, \overline{2}\vphantom{2}^a \, 4^b \, \overline{1}\vphantom{1}^c$ in the case that $a \leq c$ as follows (the case $a>c$ is similar). The $\trip_1$-strands are determined by turning alternately left at white and right at black vertices on the underlying hexagonal grid plabic graph. Thus, the only $\trip_1$-strands that separate a boundary face from the base face occur at the vertical, length $b$, sides of the hexagon. Similarly, the $\trip_3$-strands are determined by turning alternately right at white and left at black on the underlying plabic graph. Thus, the only $\trip_3$-strands that separate a boundary face from the base face occur at the vertical, length $b$, sides of the hexagon, as well as along the southeast, length $c$, side of the hexagon for the first $a$ boundary vertices and along the southwest, length $a$, side of the hexagon.

  The $\trip_2$-strand from a boundary vertex travels along only two of the three types of cube faces on the plane partition, tracing along all blocks at that ``level". As illustrated in \Cref{fig:PP-tiling}, the $\trip_2$-strands from boundary vertices on the east, length $b$, side of the hexagon travel along the dark and light gray rhombi until they reach the west, length $b$, side. Similarly, the $\trip_2$-strands travel from length $a$ side to the other length $a$ side along light gray and white rhombi, and $\trip_2$-strands travel from length $c$ side to length $c$ side along dark gray and white rhombi. Thus, one may observe that $\trip_2$-strands are only separating when they begin on a length $b$ side.
\end{proof}

\begin{figure}[htbp]
    \centering
    \includegraphics{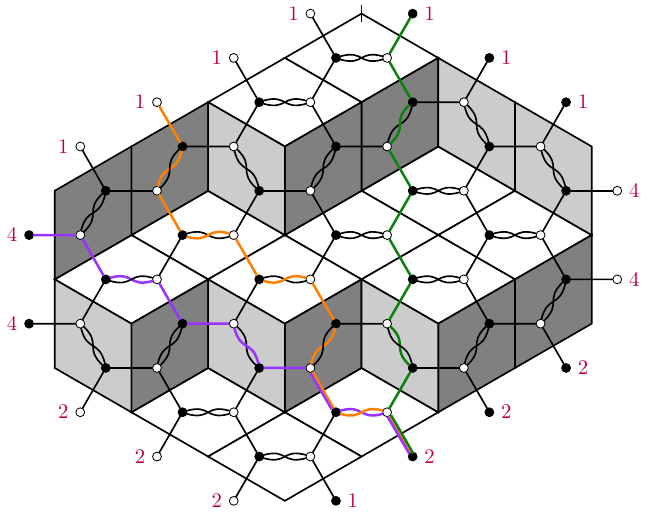}
    \caption{The large hexagon associated to $\PP(3 \times 2 \times 4)$ as in \Cref{prop:pp}. A sample hourglass plabic graph in the move-equivalence class of $\mathcal{G}(T)$ has been drawn, together with the dual plane partition and the trip strands from a boundary vertex. Here $T$ is the oscillating tableau with lattice word $L = 111\overline{4}\overline{4}2221\overline{2}\overline{2}\overline{2}44\overline{1}\overline{1}\overline{1}\overline{1}$, which corresponds to the separation labels of each boundary edge, as shown.}
    \label{fig:PP-tiling}
\end{figure}

\begin{figure}[htbp]
    \centering
    \includegraphics[width=0.9\textwidth]{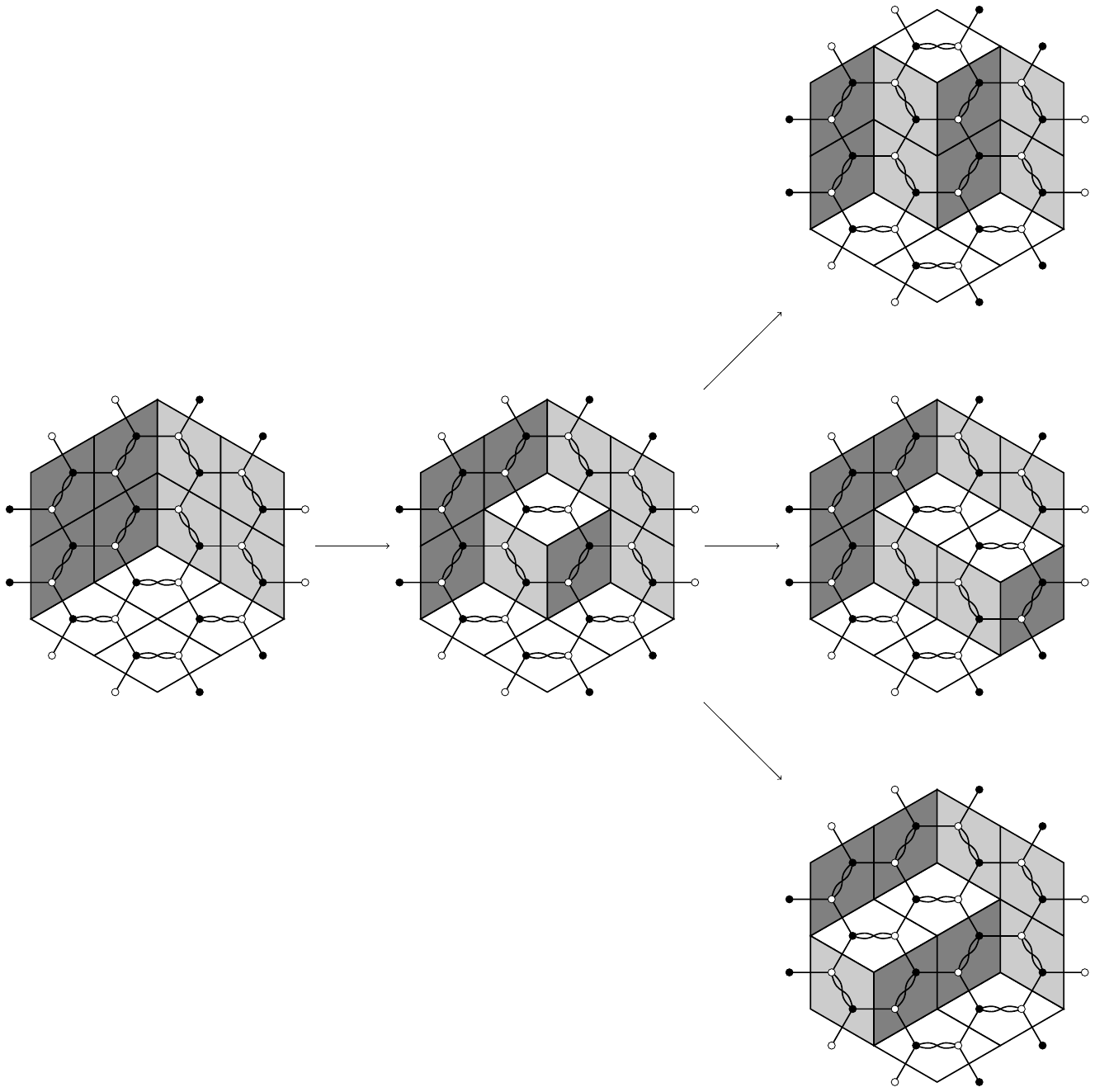}
    \caption{The $5$ elements in the first three ranks of $\PP(2 \times 2 \times 2)$ drawn as rhombus tilings, together with the corresponding hourglass plabic graphs in the move-equivalence class of $\mathcal{G}(T)$. Here $T$ is the oscillating tableau with lattice word $L = 11\overline{4}\overline{4}22\overline{2}\overline{2}44\overline{1}\overline{1}$.}
    \label{fig:PP-2x2x2}
\end{figure}

\section{Hourglass plabic graphs recover known web bases}
\label{sec:two_column}

In this section, we discuss how all known rotation-invariant web bases fit within our framework of hourglass plabic graphs, forbidden $4$-cycles, and the $\trip_{\bullet}=\prom_{\bullet}$ property, and give an extension of these bases to the ``semistandard" setting.

\subsection{A uniform characterization for small $r$}
We may extend \Cref{def:hourglass-plabic-graph} by defining an \emph{$U_q(\fsl_r)$-hourglass plabic graph} exactly as before, but with all internal vertices of degree $r$. We define the trip permutations $\trip_{\bullet}(G)=(\trip_1(G),\ldots, \trip_{r-1}(G))$ of such a graph $G$ in the natural way, extending \Cref{def:trip_perms}.

\begin{example}\label{ex:FLL-HPG-example}
    Recall the $U_q(\fsl_5)$-web from \Cref{ex:CKM-FLL-example}, drawn at left. The corresponding $U_q(\fsl_5)$-hourglass plabic graph is drawn at right. 
    \begin{center}
        \includegraphics[scale=1.2]{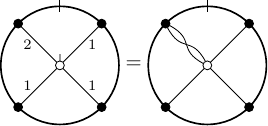}
    \end{center}
    Since tags only change the overall sign, we have suppressed the tag in the $U_q(\fsl_5)$-hourglass plabic graph.
\end{example}

In general, $U_q(\fsl_r)$-hourglass plabic graphs admit contraction moves as in \Cref{fig:contraction-hourglass-moves}, which preserve $\trip_{\bullet}$. As before, we say a $U_q(\fsl_r)$-hourglass plabic graph $G$ is \emph{contracted} if all possible contraction moves have been applied. For $r=2$ or $3$, the contraction moves are the only moves on hourglass plabic graphs. 

\begin{definition}
\label{def:general-forbidden-4-cycles}
A $4$-cycle of an $U_q(\fsl_r)$-hourglass plabic graph $G$ is called \emph{forbidden} if the sum of its edge multiplicities exceeds $r$. For $r=2,3,$ or $4$, we write $G \sim G'$ for the move-equivalence relation. We say such a graph $G$ with no isolated components is \emph{fully reduced} if no $G' \sim G$ contains a forbidden $4$-cycle. \emph{Top} fully reduced graphs are distinguished contracted representatives in each move-equivalence class of fully reduced graphs (as in \Cref{def:top-fully-reduced}); when $r=2$ or $3$, every contracted fully reduced graph is top.
\end{definition}

\Cref{thm:unified-web-basis} below shows how our hourglass plabic graph perspective with forbidden $4$-cycles and $\trip_{\bullet}=\prom_{\bullet}$, unifies the Temperley--Lieb $U_q(\fsl_2)$-basis, the Kuperberg's $U_q(\fsl_3)$-basis, and our $U_q(\fsl_4)$-basis from \Cref{thm:web-basis}.

\begin{theorem}
\label{thm:unified-web-basis}
Let $r=2,3,$ or $4$. Then $\crg(\underline{c};r)/{\sim}$ is in bijection with $\rft(\underline{c};r)$ via $\trip_{\bullet}=\prom_{\bullet}$ and this bijection intertwines rotation and promotion. Furthermore,
\[
\mathscr{B}_q^{\underline{c},r} \coloneqq \{[W]_q : \text{ $W$ a top fully reduced $U_q(\fsl_r)$-hourglass plabic graph of type $\underline{c}$}\},
\]
is a rotation-invariant web basis of $\Inv_{U_q(\fsl_r)}(\bigwedge_q\nolimits^{\underline{c}} V_q)$. 
\end{theorem}
\begin{proof}
When $r=2$, all hourglass plabic graphs consist of a non-crossing matching of the boundary vertices, with some internal vertices along each arc. Such a graph is contracted if there are $0$ or $1$ internal vertices on each arc; all such graphs are in the basis and this is exactly the Temperley--Lieb basis. The $\trip_{\bullet}=\prom_{\bullet}$ condition is an easy verification, and this property implies that rotation corresponds to promotion. 

When $r=3$, the fully reduced graphs are those whose contraction contains no $2$-cycles or $4$-cycles, since a $2$-cycle may be uncontracted into a $4$-cycle with edge multiplicities $1,1,1,2$. Thus the fully reduced graphs are exactly Kuperberg's \emph{non-elliptic} web basis. Hopkins--Rubey \cite{Hopkins-Rubey,Hopkins:talk} showed that the standard bijection between these webs and $3$-row standard Young tableaux (for example, via the growth rules of \cite{Khovanov-Kuperberg,Petersen-Pylyavskyy-Rhoades}) is characterized by $\trip_{\bullet}=\prom_{\bullet}$. Finally, the $r=4$ case is \Cref{thm:web-basis}.
\end{proof}

\subsection{The 2-column case}
In \cite{Fraser-2-column}, Fraser found a rotation-invariant $\SL_r$-web basis for $\Inv(V^{(1^{2r})})$ with $r$ arbitrary, the ``2-column" case. He describes a bijection $\varphi_r$ from rectangular standard tableaux of shape $r \times 2$ to a basis $\mathcal{W}_r$ of $\SL_r$-webs, which are only characterized as the outputs of the map $\varphi_r$ (see \Cref{fig:fraser_web} for an example). For $G = \varphi_r(T) \in \mathcal{W}_r$, all internal faces are $4$-cycles. The map $\varphi_r$ in fact requires certain choices to be made, but all possible outputs are connected by the square relation of \cite[Eq.~6.9]{Fraser-Lam-Le}, so that the invariant $[G]$ is well-defined.

\ytableausetup{boxsize=0.6cm}
\begin{figure}[htb]
    \centering
    \includegraphics[scale=0.64]{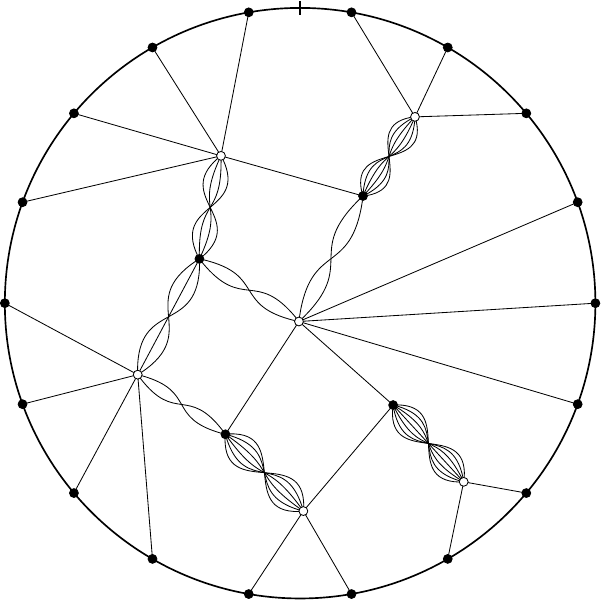}
    \qquad \raisebox{5cm}{\ytableaushort{14,27,39,5{11},6{12},8{15},{10}{16},{13}{17},{14}{18}}}
    \caption{The $\SL_9$-web from \cite[Fig.~1]{Fraser-2-column} converted to an $\SL_9$-hourglass plabic graph and the corresponding tableau. }
    \label{fig:fraser_web}
\end{figure}

In follow-up work \cite{two-column}, we have extended the perspective of this paper to show that the bijection $\varphi_r$ satisfies $\trip_{\bullet}(\varphi_r(T))=\prom_{\bullet}(T)$ and that after lifting to hourglass plabic graphs, the general square relation preserves trip permutations. Furthermore, the web basis elements $[W]_q$ for $ W \in \mathcal{W}_r$ are characterized as the web invariants of suitably defined fully reduced $U_q(\fsl_r)$-hourglass plabic graphs; full reducedness again includes the prohibition of the forbidden $4$-cycles of \Cref{def:general-forbidden-4-cycles} in move-equivalent graphs.

\subsection{The semistandard web basis} 
\label{sec:semistandard}
We now show how the growth rules and the reduction algorithm allow for the extension of the top fully reduced bases to bases for general graded pieces of the coordinate ring $\mathbb{C}[\widetilde{\mathrm{Gr}}_r(\mathbb{C}^n)]$ of the (affine cone over) the Pl\"{u}cker embedding of the Grassmannian, for $r \leq 4$. See \cite[\S6]{Fraser-2-column} for similar results in the 2-column setting.

For the remainder of this section, we consider tensor diagrams at $q=1$ all of whose boundary edges have multiplicity one, but such that a boundary vertex $b_i$ may be incident to any number $\mu_i \geq 0$ of such edges. Any such tensor diagram $X$ determines an element $[X] \in \Inv_{\SL_r}(\Sym^{\underline{\mu}}(V))$, where $\Sym^{\underline{\mu}}(V)\coloneqq \Sym^{\mu_1}(V) \otimes \cdots \otimes \Sym^{\mu_n}(V)$.  This element is the restriction of the invariant $[\check{X}] \in \Inv_{\SL_r}(\bigotimes_i V^{\otimes \mu_i})$ of the diagram $\check{X}$ obtained from $X$ by splitting $b_i$ into $\mu_i$ separate boundary vertices, each incident to a single edge. In terms of the polynomial expressions from \Cref{sec:q.Laplace}, $[X]$ can be computed from $[\check{X}]$ by setting equal, for each $i$, the corresponding $\mu_i$ vectors of boundary variables. Call such a tensor diagram $X$ a \emph{web of semistandard type} if it is planar; as before, we move freely between the language of webs and of hourglass plabic graphs. We now extend the definition of top full reducedness to such graphs.

\begin{definition}
    We say an hourglass plabic graph $G$ of semistandard type is \emph{top fully reduced} if the hourglass plabic graph $\check{G}$ is top fully reduced and if no boundary vertex of $G$ is adjacent to a single internal vertex by more than one edge.
\end{definition}

\begin{theorem}
\label{thm:semistandard}
    Let $r=2,3,$ or $4$. Then the collection $\mathscr{B}^r(\underline{\mu})$ of invariants of top fully reduced $r$-hourglass plabic graphs of semistandard type with boundary multiplicities $\underline{\mu}$ is a rotation-invariant web basis of $\mathbb{C}[\widetilde{\mathrm{Gr}}_r(\mathbb{C}^n)]_{\underline{\mu}}$. 
\end{theorem}

\begin{remark}
    Fomin and Pylyavskyy \cite{Fomin-Pylyavskyy-advances} conjecture close connections between the $\SL_3$-web basis and the cluster algebra structure on $\mathbb{C}[\widetilde{\mathrm{Gr}}_3(\mathbb{C}^n)]$, having assumed that the $\SL_3$-web basis extends to the semistandard setting as we describe above. In light of \Cref{thm:semistandard}, it would be very interesting to study whether these conjectures might also hold for the $\SL_4$-web basis and the cluster algebra structure on $\mathbb{C}[\widetilde{\mathrm{Gr}}_4(\mathbb{C}^n)]$. The rotation-invariance of our basis is particularly important in this context, as rotation induces an automorphism of the cluster structure.
\end{remark}

\begin{proof}[Proof of \Cref{thm:semistandard}]
    Suppose that $\underline{\mu}$ is a composition of $rN$ into $n$ parts. The First Fundamental Theorem of Invariant Theory (see \cite{Weyl}) implies that the graded piece $\mathbb{C}[\widetilde{\mathrm{Gr}}_r(\mathbb{C}^n)]_{\underline{\mu}} \cong \Inv_{\SL_r} \Sym^{\underline{\mu}}(V)$ is spanned by products $\Delta_{I_1}\cdots \Delta_{I_N}$ of Pl\"{u}cker coordinates such that $i$ lies in exactly $\mu_i$ of the sets $I_j$ for $i=1,\ldots,n$. Such a product equals $[X]$ where $X$ is the tensor diagram with boundary multiplicities $\underline{\mu}$ obtained by overlaying $N$ $r$-valent star graphs using boundary vertices $I_1,\ldots,I_N$. Using the relations from \Cref{alg:reduction-algorithm} when $r=4$ or using the relations from \cite{Kuperberg} when $r=2$ or $3$, which apply equally well in the semistandard setting, we can write $[X]$ as a linear combination of invariants $[W]$ of webs $W$ of semistandard type $\underline{\mu}$ such that $\check{W}$ is top fully reduced. If such a $W$ has a boundary vertex incident to a single internal vertex by more than one edge, then $[W]=0$ since $[\check{W}]$ is alternating in the corresponding two arguments, so we may in fact write $[X]$ as a linear combination of elements of $\mathscr{B}^r(\underline{\mu})$. This shows that $\mathscr{B}^r(\underline{\mu})$ spans $\mathbb{C}[\widetilde{\mathrm{Gr}}_r(\mathbb{C}^n)]_{\underline{\mu}}$.
    
    It is well-known (see e.g. \cite[Ch.~1]{Seshadri}) that $\dim \mathbb{C}[\widetilde{\mathrm{Gr}}_r(\mathbb{C}^n)]_{\underline{\mu}}$ is equal to the cardinality of the set $\ssyt(r \times N, \underline{\mu})$ of semistandard tableaux of shape $r \times N$ and content $\underline{\mu}$. Thus, to prove that $\mathscr{B}^r(\underline{\mu})$ is a basis, it suffices to give an injection from the top fully reduced graphs of semistandard type, modulo moves, to $\ssyt(r \times N, \underline{\mu})$. We define such a map $\psi$ as follows: given a top fully reduced graph $G$ of semistandard type, form the standard tableau $\mathcal{T}(\check{G})$ and make it of content $\underline{\mu}$ by relabeling the boxes labeled $1+\sum_{j=1}^{i-1} \mu_j, \ldots, \mu_i + \sum_{j=1}^{i-1} \mu_j$ by $i$; call this tableau $\psi(G)$. 
    
    We first argue that $\psi(G)$ is indeed semistandard. If boxes labeled $a$ and $a+1$ are both relabeled by $i$, then the boundary vertices $\check{b}_a$ and $\check{b}_{a+1}$ of $\check{G}$ both come from splitting the same boundary vertex $b_i$ of $G$. Since $G$ is top fully reduced of semistandard type, this means that $\check{b}_a$ and $\check{b}_{a+1}$ do not share a neighbor in $\check{G}$. By \Cref{thm:growth_algorithm}(v), this is equivalent to $a \not \in \Des(\mathcal{T}(\check{G}))$, so $a+1$ is not in a lower row than $a$ in $\mathcal{T}(\check{G})$. Applying this for each $a$ in $1+\sum_{j=1}^{i-1} \mu_j, \ldots, \mu_i + \sum_{j=1}^{i-1} \mu_j$, we see that the boxes with these labels in $\psi(G)$ occur left to right in a skew strip having no two boxes in the same column. Thus $\psi(G)$ is semistandard. We now check that $\psi$ is injective. Clearly $G_1 \sim G_2$ if and only if $\check{G_1} \sim \check{G_2}$, so, since $\mathcal{T}$ is injective on $\crg/{\sim}$ by \Cref{thm:unified-web-basis}, it only remains to check that the relabeling step is injective. But this also follows from the observation about the skew strip.

    Finally, it is clear that rotation of graphs sends the basis $\mathscr{B}^r(\underline{\mu})$ to the basis $\mathscr{B}^r((\mu_2,\ldots,\mu_{n},\mu_1))$ for the rotated weight.
\end{proof}

\section*{Acknowledgements}
This project began during the 2021 BIRS Dynamical Algebraic Combinatorics program hosted at UBC Okanagan, and we are very grateful for the excellent research environment provided there. At that conference, Sam Hopkins and Martin Rubey introduced us to their perspective from \cite{Hopkins-Rubey} of webs-as-plabic graphs which was foundational in this work. We completed the main result at NDSU (partially supported by NSF DMS-2000592), for whose hospitality we are very thankful. Finally, we are grateful for the resources provided at ICERM, where much of this paper was written.

We are indebted to the anonymous referee for their insightful and thorough comments which significantly improved the paper. We also thank the following people for their helpful comments: Ashleigh Adams, Ben Elias, Sergey Fomin, Chris Fraser, Pavel Galashin, Joel Kamnitzer, Rick Kenyon, Mikhail Khovanov, Allen Knutson, Greg Kuperberg, Thomas Lam, Ian Le, Gregg Musiker, Rebecca Patrias, Alex Postnikov, Kevin Purbhoo, Pavlo Pylyavskyy, Brendon Rhoades, Anne Schilling, Travis Scrimshaw, David Speyer, Hugh Thomas, Julianna Tymoczko, Bruce Westbury, Lauren Williams, Haihan Wu, and Paul Zinn-Justin.

\bibliographystyle{amsalphavar}
\bibliography{v6}
\end{document}